\documentclass[reqno]{amsart}
\usepackage[margin=1.15in]{geometry}
\usepackage[backref=page]{hyperref}
\usepackage{cite}
\usepackage{enumitem}

\theoremstyle{plain}
\newtheorem{theorem}{Theorem}
\newtheorem{proposition}[theorem]{Proposition}
\newtheorem{corollary}[theorem]{Corollary}
\newtheorem{lemma}[theorem]{Lemma}

\theoremstyle{definition}

\theoremstyle{remark}
\newtheorem{remark}{Remark}

\allowdisplaybreaks

\begin{document}

\author{Phuong Le}
\address{Phuong Le$^{1,2}$ (ORCID: 0000-0003-4724-7118)\newline
	$^1$Faculty of Economic Mathematics, University of Economics and Law, Ho Chi Minh City, Vietnam; \newline
	$^2$Vietnam National University, Ho Chi Minh City, Vietnam}
\email{phuongl@uel.edu.vn}


\title[Phase transition beyond $\sigma=2p$]{Phase transition, sharp asymptotics and symmetry for critical quasilinear Hartree equations beyond $\sigma=2p$}

\begin{abstract}
We establish the sharp asymptotic and symmetry theory for the critical quasilinear Hartree equation
\[
-\Delta_pu=\big(|x|^{-\sigma}*u^{p^*_{\sigma}}\big)u^{p^*_{\sigma}-1}
\quad\text{in }\mathbb R^N,
\qquad
p^*_{\sigma}=\frac{p(2N-\sigma)}{2(N-p)},
\]
in the complementary range $2p<\sigma<N$, beyond the quasilinear theory previously available for $\sigma\le2p$ [Math. Ann., 391(2):2653–2708, 2025]. This range is structurally different: $p^*_{\sigma}<p$, so the effective potential obtained by factoring out $u^{p-1}$ becomes singular as $u\to0$. Nevertheless, whenever $p^*_{\sigma}>1$, every nonnegative finite-energy weak solution is automatically locally bounded. Our proof provides a quasilinear regularity mechanism adapted to this complementary range, where the potential-theoretic route available for $\sigma\le2p$ breaks down. For positive finite-energy solutions with $p^*_{\sigma}>1$ we moreover uncover a sharp transition governed by the sign of $\sigma(p-2)+2p$. If this quantity is positive, the classical $p$-harmonic profile persists. At the threshold,
\[
u(x)\simeq |x|^{-\gamma_h}(\log|x|)^{1/(p-p^*_{\sigma})},
\qquad \gamma_h=\frac{N-p}{p-1},
\]
with the matching two-sided gradient rate. Below the threshold,
\[
u(x)=A_*|x|^{-\gamma_*}+o(|x|^{-\gamma_*}),
\qquad
\gamma_*=\frac{2(N-p)(\sigma-p)}{p(\sigma-2p)}<\gamma_h,
\]
and the rescaled gradients converge uniformly to the corresponding radial profile. In all three regimes the solutions are radially symmetric about some point and strictly decreasing in the radial variable. The transition is genuinely quasilinear: for $p=2$ the sign is always positive, whereas the resonant and slow phases occur only in the singular range $1<p<2$. The proof develops four mechanisms beyond the earlier potential-theoretic approach: a Morrey--Adams--Moser regularity scheme, a nonlinear decay bootstrap that detects the transition, a scale-invariant blow-down rigidity argument for the exact slow profile, and a singular moving-plane method closed at the critical tail by a strict Hardy coercivity margin.
\end{abstract}

\maketitle

\section{Introduction}\label{sec_intro}

\subsection{The equation and the complementary regime}

Let $N\ge3$, $1<p<N$ and $0<\sigma<N$. Set
\begin{equation}\label{eq:psigma}
p^{*}_{\sigma}:=\frac{p(2N-\sigma)}{2(N-p)},
\qquad
p^{*}:=\frac{Np}{N-p},
\qquad q:=p^*_{\sigma}.
\end{equation}
We study nonnegative finite-energy weak solutions of
\begin{equation}\label{eq:main}
\begin{cases}
-\Delta_{p}u=\big(|x|^{-\sigma}*u^{q}\big)u^{q-1} &\text{in }\mathbb{R}^{N},\\
u\in D^{1,p}(\mathbb{R}^{N}),\quad u\ge0 &\text{in }\mathbb{R}^{N},
\end{cases}
\end{equation}
where
\[
D^{1,p}(\mathbb R^N)=\Big\{u\in L^{p^*}(\mathbb R^N):\nabla u\in L^p(\mathbb R^N)\Big\}.
\]
The equation is the Euler--Lagrange equation associated with the critical Sobolev--Hardy--Littlewood--Sobolev inequality \cite{Lieb1983}
\begin{equation}\label{eq:hls-emb}
\Big(\int_{\mathbb{R}^{N}}\int_{\mathbb{R}^{N}}\frac{u^{q}(x)u^{q}(y)}{|x-y|^{\sigma}}\,dx\,dy\Big)^{\frac{1}{2q}}
\le C_{N,p,\sigma}\|\nabla u\|_{L^{p}(\mathbb{R}^{N})},
\end{equation}
and is invariant under translations, rotations and the critical scaling $u(x)\mapsto\lambda^{(N-p)/p}u(\lambda x)$. Writing
\[
W[u]:=|x|^{-\sigma}*u^q,
\]
we call $u\in D^{1,p}(\mathbb R^N)$ a weak solution if $W[u]u^{q-1}\in L^1_{\rm loc}$ and
\begin{equation}\label{eq:weak}
\int_{\mathbb{R}^{N}}|\nabla u|^{p-2}\langle\nabla u,\nabla\varphi\rangle\,dx
=\int_{\mathbb{R}^{N}}W[u]u^{q-1}\varphi\,dx
\end{equation}
for every $\varphi\in C_0^\infty(\mathbb R^N)$. When $q\le1$ we use this definition only for solutions positive almost everywhere, with the local integrability requirement included explicitly.

The convolution term in \eqref{eq:main} is the homogeneous Riesz--Hartree interaction underlying the Choquard family of equations. In the classical semilinear setting, related stationary Choquard--Pekar equations arise from mean-field descriptions of polarons and from Hartree--Fock-type plasma models; they are also closely related to the Schr\"odinger--Newton model of self-gravitating matter. Lieb's foundational variational analysis \cite{Lieb1977} and the later qualitative theory developed, for example, in \cite{MorozVanSchaftingen2013} established the importance of existence, symmetry and decay questions for such nonlocal equations. Equation \eqref{eq:main} is a quasilinear counterpart in which the linear Laplacian is replaced by the $p$-Laplacian, while the exponent $p^*_{\sigma}$ is selected by the Hardy--Littlewood--Sobolev critical scaling. Thus the problem combines degenerate or singular diffusion with a scale-invariant long-range interaction, making sharp regularity and asymptotic information especially delicate.

For the local critical equation $-\Delta_pu=u^{p^*-1}$, the classification of positive finite-energy solutions and the sharp asymptotic theory were developed in \cite{DamascelliMerchanMontoroSciunzi2014,Vetois2016,Sciunzi2016}; see also \cite{CiraoloFigalliRoncoroni2020,CatinoMonticelliRoncoroni2023} for related extensions. Natural semilinear precursors in the Hartree setting are the independent classification results of Du and Yang \cite{DuYang2019} and Guo, Hu, Peng and Shuai \cite{GuoHuPengShuai2019}. In the notation of Du and Yang, for $p=2$ they classified the positive solutions of
\[
-\Delta u=\big(I_\mu*u^{2^*_\mu}\big)u^{2^*_\mu-1},
\qquad 2^*_\mu=\frac{2N-\mu}{N-2},
\]
where $I_\mu(x)=C_{N,\mu}|x|^{-\mu}$. Thus $\mu=\sigma$ in our notation. Those works cover $0<\sigma<N$ when $N=3,4$ and $0<\sigma\le4$ when $N\ge5$ at the finite-energy level. The remaining semilinear range was subsequently covered in \cite{LeHartree2020}, which removes the restriction on $\sigma$ but assumes from the outset that the solution is positive and classical. We therefore do not claim a new semilinear regularity principle here: Moroz and Van Schaftingen developed a nonlocal Br\'ezis--Kato regularity scheme for weak $H^1$ solutions of semilinear Choquard equations with upper critical growth \cite{MorozVanSchaftingen2015}. What is important for the present paper is that this mechanism is tied to the linear Laplacian. In the quasilinear setting, Ambrosio proved boundedness and local $C^{1,\alpha}$ regularity for a massive Choquard equation when $p\ge2$ and the Riesz order satisfies $\alpha>(N-2p)_+$ \cite{Ambrosio2023}; when $N>2p$ and $\alpha=N-\sigma$, this condition corresponds to $\sigma<2p$. Ambrosio also observed that the semilinear Br\'ezis--Kato approach does not transfer directly to $p\ne2$ because of the nonlinear character of the $p$-Laplacian. Dai, Li and Liu later developed sharp asymptotic and symmetry theory for finite-energy weak solutions of the critical quasilinear problem and proved in \cite{DaiLiLiu2025} that positive finite-energy solutions satisfy
\begin{equation}\label{eq:sharp-u}
\frac{c_{0}}{1+|x|^{\frac{N-p}{p-1}}}\le u(x)\le\frac{C_{0}}{1+|x|^{\frac{N-p}{p-1}}},
\end{equation}
\begin{equation}\label{eq:sharp-grad}
\frac{c_{0}}{|x|^{\frac{N-1}{p-1}}}\le|\nabla u(x)|\le\frac{C_{0}}{|x|^{\frac{N-1}{p-1}}}
\qquad(|x|\gg1),
\end{equation}
together with radial symmetry and strict radial monotonicity, under
\begin{equation}\label{eq:old-restriction}
0<\sigma<N,\qquad \sigma\le2p.
\end{equation}
Since
\begin{equation}\label{eq:q-minus-p}
q-p=\frac{p(2p-\sigma)}{2(N-p)},
\end{equation}
this is exactly the range $q\ge p$. The present paper treats the complementary region
\begin{equation}\label{eq:new-range}
2p<\sigma<N,
\qquad\text{equivalently}\qquad
\frac{p^*}{2}<q<p,
\end{equation}
which is nonempty precisely when $1<p<N/2$. For $p=2$ this is the range $4<\sigma<N$, whose classical-solution theory is contained in \cite{LeHartree2020}. The focus here is the quasilinear complementary regime, where neither the effective-potential method of \cite{DaiLiLiu2025} nor the semilinear Br\'ezis--Kato mechanism provides the required regularity theory. Moreover, when $1<p<2$, this complementary range contains a second threshold at which the $p$-harmonic decay law breaks down and new resonant and slow phases emerge.

\subsection{Why the complementary range is different}

The passage from $\sigma\le2p$ to $2p<\sigma<N$ is not a routine extension of the parameter range. In the potential formulation used in \cite{DaiLiLiu2025},
\begin{equation}\label{eq:VForm}
-\Delta_pu=V(x)u^{p-1},
\qquad
V(x)=W[u](x)u^{q-p}(x),
\end{equation}
the factor $u^{q-p}$ becomes singular at small values of $u$ as soon as $q<p$. Consequently, the global $L^{N/p}$ control of the effective potential and the tail estimates underlying the general potential theorem of Dai--Li--Liu no longer follow from finite energy. Proposition \ref{prop:frozen} below shows that this is a genuine analytic obstruction rather than a defect of a particular estimate.

More importantly, the complementary range is not governed by a single continuation of the $p$-harmonic decay law. The asymptotic mechanism itself changes. Local boundedness must first be recovered directly from the nonlocal equation; we do this through a self-improving hole-filling/Morrey--Adams--Moser argument. The far-field decay is then organized by a nonlinear bootstrap map whose fixed-point geometry detects a resonance and, beyond it, a slower nonlinear balance. In the slow phase the leading coefficient is selected by a scale-invariant blow-down rigidity argument. Finally, the symmetry proof cannot be obtained by simply repeating the moving-plane decomposition used for $q\ge p$: we split at the level of $u^{q-1}$, obtain strict separation at the limiting plane from the Hartree term itself, and in the singular range $p<2$ close the critical tail by combining a sharp monotonicity inequality with a strict Hardy coercivity margin. In particular, the argument avoids the version of the strong comparison principle that would impose $p>(2N+2)/(N+2)$.

This change of mechanism produces a new asymptotic threshold inside the complementary range. Define
\begin{equation}\label{eq:gammas}
\gamma_h:=\frac{N-p}{p-1},
\qquad
\gamma_e:=\frac{N-p}{p},
\qquad
\gamma_*:=\frac{\sigma-p}{p-q}=\frac{2(N-p)(\sigma-p)}{p(\sigma-2p)},
\end{equation}
\begin{equation}\label{eq:Lambda}
\Lambda:=\sigma+(q-1)\gamma_h-N
=\frac{\sigma(p-2)+2p}{2(p-1)},
\end{equation}
and
\begin{equation}\label{eq:hdef}
h(\nu):=\nu^{p-1}\big(N-p-(p-1)\nu\big).
\end{equation}
Here $\gamma_h$ is the $p$-harmonic rate, $\gamma_e$ is the scaling-critical rate, and $\gamma_*$ is the nonlinear balance rate after the Hartree potential has saturated. The identity
\[
\Lambda>0\iff\gamma_*>\gamma_h
\]
shows that $\Lambda=0$ is precisely the collision of the nonlinear balance with the $p$-harmonic exponent.

\subsection{Main results}

Our main result shows that the complementary range has a sharp three-regime asymptotic structure rather than a single continuation of the classical profile.

\begin{theorem}[Asymptotic phase transition and symmetry]\label{thm:three-regime}
Let $N\ge3$, $p>1$ and $2p<\sigma<N$, and assume that $q:=p^*_{\sigma}>1$. Let $u$ be a positive $D^{1,p}(\mathbb{R}^{N})$-weak solution of \eqref{eq:main}. Then exactly one of the following alternatives occurs:
\begin{enumerate}[label=\textup{(\roman*)}]
\item If $\sigma(p-2)+2p>0$ (equivalently $\Lambda>0$), then
\[
u\simeq |x|^{-\gamma_h},\qquad |\nabla u|\simeq |x|^{-\gamma_h-1},
\]
and $u$ is radially symmetric about some point and strictly decreasing in the radial variable.
\item If $\sigma(p-2)+2p=0$ (equivalently $\Lambda=0$), then
\[
u\simeq |x|^{-\gamma_h}(\log|x|)^{1/(p-q)},
\qquad
|\nabla u|\simeq |x|^{-\gamma_h-1}(\log|x|)^{1/(p-q)},
\]
and $u$ is radially symmetric about some point and strictly decreasing in the radial variable.
\item If $\sigma(p-2)+2p<0$ (equivalently $\Lambda<0$), then, with $m_0=\int_{\mathbb R^N}u^q$ and
$A_*=(m_0/h(\gamma_*))^{1/(p-q)}$,
\[
r^{\gamma_*}u(r\omega)\to A_*,
\qquad
r^{\gamma_*+1}\nabla u(r\omega)\to-A_*\gamma_*\omega
\]
uniformly for $\omega\in S^{N-1}$. In particular
\[
u\simeq |x|^{-\gamma_*},
\qquad
|\nabla u|\simeq |x|^{-\gamma_*-1},
\qquad \gamma_*<\gamma_h,
\]
and $u$ is radially symmetric about some point and strictly decreasing in the radial variable.
\end{enumerate}
\end{theorem}

A key point in Theorem \ref{thm:three-regime} is that no local boundedness or pointwise regularity is assumed. This is supplied by the following independent regularity theorem, which closes the first obstruction created by the singular factor $u^{q-p}$.

\begin{theorem}[Automatic local boundedness]\label{thm:local-boundedness}
Let $N\ge3$, $p>1$ and $2p<\sigma<N$, and assume that $q:=p^*_{\sigma}>1$. Then every nonnegative $D^{1,p}(\mathbb R^N)$-weak solution of \eqref{eq:main} belongs to $L^\infty_{\mathrm{loc}}(\mathbb R^N)$.
\end{theorem}

The semilinear case also provides a useful benchmark for the asymptotic transition. When $p=2$, one has $\sigma(p-2)+2p=4>0$ for every $0<\sigma<N$, so no resonant or slow phase occurs. The new transition is therefore a genuinely quasilinear effect tied to singular diffusion.

Theorem \ref{thm:three-regime} also identifies the actual threshold governing the asymptotic behavior. When $p\ge2$, one has $\sigma(p-2)+2p>0$ throughout $2p<\sigma<N$, so the $p$-harmonic profile persists across the former boundary $\sigma=2p$. When $1<p<2$, however, the complementary range contains a second threshold,
\[
\sigma=\frac{2p}{2-p},
\]
at which the $p$-harmonic profile acquires a logarithmic correction and beyond which it is replaced by the slower power $|x|^{-\gamma_*}$. The latter regime carries an exact leading coefficient determined by the total Hartree mass $m_0$.

The automatic regularity in Theorem \ref{thm:local-boundedness} is itself part of the new quasilinear mechanism. A hole-filling estimate first yields Morrey decay of $|\nabla u|^p$; Adams' theorem then produces an integrability gain beyond $L^{p^*}$, and a self-improving Moser iteration gives arbitrary local integrability and finally local boundedness. This route is different from both the semilinear Br\'ezis--Kato scheme of \cite{MorozVanSchaftingen2015} and the quasilinear regularity argument of \cite{Ambrosio2023}, whose parameter condition lies on the opposite side of the interface $\sigma=2p$ in the critical scaling considered here. Proposition \ref{prop:frozen} explains why boundedness cannot be recovered from the natural Lebesgue class of the Hartree potential alone. Positive finite-energy weak solutions are provided, in particular, by extremals of the critical minimization problem; see \cite{Lions1985} and the discussion in \cite[Section 1.3]{DaiLiLiu2025}.

By Lemma \ref{lem:qbigger1}, the condition $q>1$ is automatic throughout the classical phase $\Lambda>0$, and it is also automatic at the resonant boundary $\Lambda=0$. Thus $q>1$ is a genuine restriction only in the strict slow phase. When $q\le1$, the local factor $u^{q-1}$ is singular and the problem falls outside the asymptotic--symmetry theory developed here.

\subsection{Proof strategy and organization}

The proof has five main ingredients. Section \ref{sec:structure} establishes a blow-up dichotomy for the possible singular set and records the frozen-weight obstruction to a purely Lebesgue-space regularity argument. Section \ref{sec:localboundedness} closes that gap by proving Theorem \ref{thm:local-boundedness}. Section \ref{sec:powergain} then gives a genuine power gain at infinity, and Section \ref{sec:bootstrap} develops the decay map that separates the three asymptotic phases. Section \ref{sec:classical} recovers the Dai--Li--Liu profile and symmetry when $\Lambda>0$, while Section \ref{sec:nonclassical} treats $\Lambda=0$ and $\Lambda<0$ by logarithmic barriers, blow-down rigidity and a singular moving-plane argument. The key new point in the slow phase is a strict Hardy coercivity margin at the critical tail.

Throughout, $C$ denotes a positive constant whose value may change from line to line, and $B_R=B_R(0)$. For positive functions $f$ and $g$ defined outside a ball, we write $f\simeq g$ as $|x|\to\infty$ if there exist constants $c,C,R>0$ such that $c\,g(x)\le f(x)\le C\,g(x)$ for every $|x|\ge R$.

\section{Preliminaries}\label{sec:prelim}

\subsection{The exponents}

\begin{lemma}\label{lem:exponents}
Let $N\ge3$, $p>1$ and $2p<\sigma<N$, and let $q,\gamma_{h},\gamma_{e},\gamma_{*},\Lambda$ be as in \eqref{eq:psigma}, \eqref{eq:gammas} and \eqref{eq:Lambda}. Then
\begin{enumerate}[label=\textup{(\alph*)}]
\item $p^{*}/2<q<p$ and $2q>p$;
\item $\gamma_{e}<\gamma_{h}$ and $\gamma_{*}>\gamma_{e}$;
\item $\Lambda>0\iff\sigma(2-p)<2p\iff\gamma_{*}>\gamma_{h}$, and $\Lambda>0$ holds whenever $p\ge2$;
\item $\sigma-(p-q)\gamma_{h}=p+\Lambda$ and $q\gamma_{h}+\sigma-N-(p-q)\gamma_{h}=p+\frac{N+p-\sigma}{p-1}$;
\item $q\gamma_{h}>N$ if and only if $\sigma<2N/p$; moreover if $\Lambda\le0$ then $\sigma<2N/p$.
\end{enumerate}
\end{lemma}

\begin{proof}
(a) From \eqref{eq:psigma}, $q-p^{*}/2=p(N-\sigma)/(2(N-p))>0$ and $q-p^{*}=-p\sigma/(2(N-p))<0$, while $q-p<0$ by \eqref{eq:q-minus-p} and $\sigma>2p$. Since $q>p^{*}/2>p/2$ we get $2q>p$.

(b) The first is clear from $p-1<p$. For the second, $\gamma_{*}-\gamma_{e}=\sigma(p-N)/\big(p(2p-\sigma)\big)>0$ because $p<N$ and $\sigma>2p$.

(c) A computation gives
\[
\gamma_{*}-\gamma_{h}=\frac{(p-N)\big(p\sigma+2p-2\sigma\big)}{p(p-1)(2p-\sigma)} ,
\]
and both $p-N$ and $2p-\sigma$ are negative, so the sign of $\gamma_{*}-\gamma_{h}$ is that of $2p-\sigma(2-p)$. The identity $\sigma+(q-1)\gamma_{h}-N=(\sigma(p-2)+2p)/(2(p-1))$ is a direct computation from \eqref{eq:psigma}, and shows that $\Lambda$ has the same sign. If $p\ge2$ then $\sigma(2-p)\le0<2p$.

(d) Both identities follow from $p-q=p(\sigma-2p)/(2(N-p))$ and $\gamma_{h}=(N-p)/(p-1)$.

(e) $q\gamma_{h}-N=(N-p\sigma/2)/(p-1)$, which is positive if and only if $\sigma<2N/p$. If $\Lambda\le0$ then $p<2$ by (c), whence $2N/p>N>\sigma$.
\end{proof}

\begin{lemma}\label{lem:qbigger1}
Under the assumptions of Lemma \ref{lem:exponents} together with $\Lambda\ge0$, one has $q>1$.
\end{lemma}

\begin{proof}
From \eqref{eq:psigma}, $q>1$ is equivalent to $\sigma<2N(p-1)/p+2$. Suppose first that $p\ge2N/(N+2)$. Then $p(N-2)\le2N(p-1)$, that is, $N\le2N(p-1)/p+2$, and the claim follows from $\sigma<N$. Suppose next that $p<2N/(N+2)$, so in particular $p<2$. Then $2p<N(2-p)$, hence $2/(2-p)<N/p$, hence
\[
\frac{2p}{2-p}-2=\frac{4(p-1)}{2-p}<\frac{2N(p-1)}{p},
\]
so $2p/(2-p)\le2N(p-1)/p+2$. Since $\Lambda\ge0$ gives $\sigma\le2p/(2-p)$ by Lemma \ref{lem:exponents}(c), while the preceding inequality is strict in the present case, the claim follows.
\end{proof}

\subsection{Imported tools}

We shall use the Hardy--Littlewood--Sobolev inequality in the form of \cite[Theorem 2.1]{DaiLiLiu2025}, and the following pointwise estimates for Riesz type convolutions, which are \cite[Lemma 2.2]{DaiLiLiu2025} and are proved there for all $0<\sigma<N$ and all $q>0$.

\begin{lemma}[{\cite[Lemma 2.2]{DaiLiLiu2025}}]\label{lem:conv}
Let $0<\sigma<N$, $q>0$ and $R>0$.
\begin{enumerate}[label=\textup{(\roman*)}]
\item If $u\ge0$ in $\mathbb{R}^{N}$ and $u\ge C_{1}|x|^{-\beta}$ in $\mathbb{R}^{N}\setminus B_{R}$ with $C_{1},\beta>0$ and $q\beta>N-\sigma$, then $|x|^{-\sigma}*u^{q}\ge C|x|^{N-\sigma-q\beta}$ in $\mathbb{R}^{N}\setminus B_{R}$.
\item If $u\in L^{\infty}(B_{R})$ and $0\le u\le C_{1}|x|^{-\beta}$ in $\mathbb{R}^{N}\setminus B_{R}$ with $q\beta>N-\sigma$, then in $\mathbb{R}^{N}\setminus B_{2R}$,
\[
|x|^{-\sigma}*u^{q}\le
\begin{cases}
C|x|^{N-\sigma-q\beta} &\text{if }N-\sigma<q\beta<N,\\
C|x|^{-\sigma}\log|x| &\text{if }q\beta=N,\\
C|x|^{-\sigma} &\text{if }q\beta>N.
\end{cases}
\]
\item If $c(1+|x|^{\beta})^{-1}\le f\le C(1+|x|^{\beta})^{-1}$ in $\mathbb{R}^{N}$ with $\beta>N-\sigma$, then $C_{1}\le|x|^{-\sigma}*f\le C_{2}$ in $B_{R}$.
\end{enumerate}
\end{lemma}

We use the interior regularity theory of DiBenedetto \cite{DiBenedetto1983} and Tolksdorf \cite{Tolksdorf1984}, the weak Harnack inequality of Trudinger \cite{Trudinger1967} and Serrin \cite{Serrin1964}, and the lower bound for positive $p$-superharmonic functions of Serrin and Zou \cite{SerrinZou2002}, in the form of \cite[Lemma 3.4]{DaiLiLiu2025}:
\begin{equation}\label{eq:SZ}
u>0 \text{ and } -\Delta_{p}u\ge0 \text{ in }\mathbb{R}^{N}
\quad\Longrightarrow\quad
u(x)\ge\frac{c}{|x|^{\gamma_{h}}}\quad\text{for }|x|\ \text{large}.
\end{equation}
Finally we shall quote \cite[Theorem 1.3]{DaiLiLiu2025} in Section \ref{sec:sharp}. For the degenerate moving-plane argument we also use the weighted Sobolev theory associated with the weight $|\nabla u|^{p-2}$ developed by Damascelli and Sciunzi \cite{DamascelliSciunzi2004,DamascelliSciunzi2006}. The local weighted embedding needed below is recorded in Lemma \ref{lem:weighted-sobolev-local}; its inverse-gradient hypothesis is verified explicitly in our setting.

None of the moving-plane arguments in this paper uses a strong comparison principle, a Hopf lemma, or the connectivity of the complement of the critical set: the strict separation needed at the limiting position of the plane is supplied by the nonlocal term, see Lemma \ref{lem:strict-separation}, and strict radial monotonicity follows from the radial form of the equation. A strong comparison principle is used in one place only, at a contact point where the gradient does not vanish, in Lemma \ref{lem:slow-rigidity}. There we use \cite[Lemma 2.7]{DaiLiLiu2025}, which is \cite[Theorem 1.4]{Damascelli1998}. It holds for every $1<p<\infty$ and asserts equality on the connected component, containing the contact point, of the complement of the common critical set. We never use \cite[Lemma 2.6]{DaiLiLiu2025}, which is taken there from Damascelli and Sciunzi and requires $\frac{2N+2}{N+2}<p<2$ or $p>2$: in the nonclassical regimes $\Lambda\le0$ one has $p<\frac{2N}{N+2}<\frac{2N+2}{N+2}$, so that version is never available there.

In the classical phase with $p\ge2$ we use the critical-set estimates \cite[Corollary 2.10, Remark 2.11]{DaiLiLiu2025} to obtain the local inverse-gradient control for the degenerate weight $|\nabla u|^{p-2}$ and to know that the critical set has measure zero. These are the inputs for the fixed-domain weighted Sobolev inequality recorded below. Since the cited critical-set results are stated there for the Hartree equation with $\sigma=2p$, we record explicitly why their local assumptions remain valid here.

\begin{lemma}[Local regularity of the Hartree right-hand side]\label{lem:rhs-local}
Let $q>1$ and let $u$ be a positive locally bounded $D^{1,p}(\mathbb R^N)$-weak solution of \eqref{eq:main}. Then, for every compact set $K\subset\mathbb R^N$,
\[
W[u]\in W^{1,\infty}(K),\qquad
f:=W[u]u^{q-1}\in W^{1,\infty}(K),\qquad
\inf_K f>0.
\]
In particular, if $\alpha$ is a local $C^{1,\alpha}$ exponent furnished by the interior regularity theory, then one may choose any
\[
0<\mu<\min\{\alpha(p-1),\tfrac12\}
\]
so that the local assumptions (2.7)--(2.9), including condition $(I_\alpha)$, in \cite[Remark 2.3, Lemma 2.9]{DaiLiLiu2025} hold on every bounded subdomain compactly contained in $\mathbb R^N$.
\end{lemma}

\begin{proof}
Interior regularity gives $u\in C^{1,\alpha}_{\mathrm{loc}}(\mathbb R^N)$. Fix a compact set $K\subset\mathbb R^N$ and choose $\chi\in C_c^\infty(\mathbb R^N)$ with $\chi=1$ on a neighborhood of $K$. Write
\[
W[u](x)=\int_{\mathbb R^N}|x-y|^{-\sigma}\chi(y)u(y)^q\,dy
      +\int_{\mathbb R^N}|x-y|^{-\sigma}(1-\chi(y))u(y)^q\,dy
      =:W_{\rm near}(x)+W_{\rm far}(x).
\]
Since $|x|^{-\sigma}\in L^1_{\rm loc}(\mathbb R^N)$ and $\chi u^q\in C_c^1(\mathbb R^N)$, differentiation in the sense of distributions and integration by parts give
\[
\partial_iW_{\rm near}=|\cdot|^{-\sigma}*\partial_i(\chi u^q),
\]
which is bounded on $K$. For the far part, set $\delta:=\operatorname{dist}(K,\operatorname{supp}(1-\chi))>0$. The kernel may then be differentiated directly. With
\[
a:=\frac{p^*}{q}=\frac{2N}{2N-\sigma},
\qquad a'=\frac{2N}{\sigma},
\]
H\"older's inequality yields, uniformly for $x\in K$,
\[
|\nabla W_{\rm far}(x)|
\le C\|u^q\|_{L^a}
\left(\int_{|x-y|\ge\delta}|x-y|^{-(\sigma+1)a'}\,dy\right)^{1/a'}<\infty,
\]
because $(\sigma+1)a'>N$. Thus $W[u]\in W^{1,\infty}(K)$. Since $u$ is positive and continuous, it is bounded below by a positive constant on $K$; moreover $W[u]$ is positive and continuous because the kernel is positive and $u\not\equiv0$. Hence $\inf_KW[u]>0$. The assertions for $f=W[u]u^{q-1}$ follow from the product rule, using again the positive lower bound for $u$ on $K$ when $1<q<2$.

It remains only to match precisely the local lower-bound assumption used in \cite[Lemma 2.9]{DaiLiLiu2025}. Let $\Omega$ be a bounded domain with compact closure in $\mathbb R^N$, choose a compact $K$ whose interior contains $\overline\Omega$, and fix
$0<\mu<\min\{\alpha(p-1),\tfrac12\}$. Since $c_K:=\inf_K f>0$, for every $x_0\in\Omega$ and every ball $B_\rho(x_0)\subset K$ one has
\[
f(x)\ge c_K\ge c_K\rho^{-\mu}|x-x_0|^\mu
\qquad (x\in B_\rho(x_0)).
\]
Moreover $f\in W^{1,\infty}(K)$, hence $f\in W^{1,m}(\Omega)$ for every finite $m$, in particular for some
$m>N/[2(1-\mu)]$, and also $f\in L^s(\Omega)$ for every $s<\infty$. Thus (2.7)--(2.9) and condition $(I_\alpha)$ of \cite[Remark 2.3]{DaiLiLiu2025} are satisfied.
\end{proof}

Consequently, on every bounded domain compactly contained in $\mathbb R^N$, \cite[Lemma 2.9]{DaiLiLiu2025} gives the local inverse-gradient estimate (2.13). The finite-covering argument used in the proof of \cite[Corollary 2.10]{DaiLiLiu2025} then yields the corresponding estimate (2.14) for the weight $\varrho=|\nabla u|^{p-2}$. For $p>2$ this local inverse-weight integrability immediately implies $|Z_u\cap\Omega|=0$ on every bounded $\Omega$ with compact closure in $\mathbb R^N$; for $p=2$ the weight is identically one. Hence $|Z_u|=0$ by exhaustion. The sharp lower gradient estimate proved later also shows that $Z_u$ is bounded, but that stronger fact is not needed for the local weighted Sobolev argument below.

\begin{lemma}[Local weighted Sobolev inequality]\label{lem:weighted-sobolev-local}
Assume $p\ge2$ and let $u$ be a positive locally bounded $D^{1,p}(\mathbb R^N)$-weak solution of \eqref{eq:main}. For every ball $B$ with compact closure in $\mathbb R^N$ there exist an exponent $r_B>2$ and a constant $C_B>0$ such that
\begin{equation}\label{eq:weighted-sobolev-local}
\|\phi\|_{L^{r_B}(B)}^2\le C_B\int_B |\nabla u|^{p-2}|\nabla\phi|^2\,dx
\qquad\text{for every }\phi\in C_c^1(B).
\end{equation}
\end{lemma}

\begin{proof}
If $p=2$, the weight is identically one and \eqref{eq:weighted-sobolev-local} is the usual Sobolev inequality, with any $2<r_B\le2N/(N-2)$.

Assume $p>2$, and choose a slightly larger ball $\widetilde B$ with $\overline B\subset\widetilde B$ and $\overline{\widetilde B}\subset\mathbb R^N$. By Lemma \ref{lem:rhs-local}, the Hartree right-hand side $f=W[u]u^{q-1}$ satisfies on $\widetilde B$ the local assumptions used in \cite[Lemma 2.9]{DaiLiLiu2025}. Hence the inverse-gradient estimate (2.13) of that lemma and the finite-covering argument of \cite[Corollary 2.10]{DaiLiLiu2025} give the local estimate (2.14): for every
\[
1\le t<\frac{p-1}{p-2},\qquad 0\le\gamma<N-2,
\]
the quantity
\[
\sup_{x\in B}\int_B
\frac{dy}{|\nabla u(y)|^{(p-2)t}|x-y|^{\gamma}}
\]
is finite, after enlarging the constant if necessary.

Fix
\[
1<t<\frac{p-1}{p-2}
\]
and then choose
\[
\max\{0,N-2t\}<\gamma<N-2.
\]
This is possible because $t>1$. With $\varrho=|\nabla u|^{p-2}$ we therefore have
\[
\sup_{x\in B}\int_B\varrho(y)^{-t}|x-y|^{-\gamma}\,dy<\infty,
\qquad
\gamma<N,
\qquad
t>\frac{N-\gamma}{2},
\qquad
2>1+\frac1t.
\]
These are precisely the abstract hypotheses singled out in \cite[Remark 3.1]{DamascelliSciunzi2006} for the weighted Sobolev and Poincar\'e inequalities proved there; see also \cite[Theorem 2.2]{DamascelliSciunzi2006}. Applying that weighted Sobolev inequality on the fixed ball $B$ yields \eqref{eq:weighted-sobolev-local} for some exponent $r_B>2$. Notice that only the existence of one such exponent is used below.
\end{proof}

We shall repeatedly use the weak comparison principle for the $p$-Laplacian on exterior domains in the following form. Its proof is the standard monotonicity argument, and we include it because the domain is unbounded.

\begin{lemma}\label{lem:comparison}
Let $\Omega$ be either a bounded open subset of $\mathbb{R}^{N}$ or the exterior domain $\mathbb{R}^{N}\setminus\overline{B_{\rho}}$ with $\rho>0$, and let $u,v\in W^{1,p}_{\mathrm{loc}}(\overline{\Omega})$ be such that $(u-v)^{+}\in D^{1,p}(\Omega)$ and $(u-v)^{+}=0$ on $\partial\Omega$. In the exterior-domain case assume in addition that $\nabla u,\nabla v\in L^{p}(\Omega)$. If
\[
\int_{\Omega}|\nabla u|^{p-2}\langle\nabla u,\nabla\varphi\rangle\,dx
\le\int_{\Omega}|\nabla v|^{p-2}\langle\nabla v,\nabla\varphi\rangle\,dx
\]
for every nonnegative $\varphi\in C^{\infty}_{0}(\Omega)$, then $u\le v$ in $\Omega$.
\end{lemma}

\begin{proof}
By density the inequality holds for every nonnegative $\varphi\in D^{1,p}(\Omega)$ with compact support in $\overline{\Omega}$ and $\varphi=0$ on $\partial\Omega$. If $\Omega$ is bounded this already covers $\varphi=(u-v)^{+}$ and we may pass directly to the display below. In the exterior case, let $\eta_{R}$ be cut-offs with $\eta_{R}=1$ on $B_{R}$, $\eta_{R}=0$ outside $B_{2R}$ and $|\nabla\eta_{R}|\le C/R$, and put $\varphi=(u-v)^{+}\eta_{R}$. Writing $w=(u-v)^{+}$ and $A_{R}=B_{2R}\setminus B_{R}$, the error committed in replacing $\varphi$ by $w$ is controlled by
\[
\frac{C}{R}\int_{A_{R}}\big(|\nabla u|^{p-1}+|\nabla v|^{p-1}\big)w\,dx
\le C\big\|\nabla u\big\|^{p-1}_{L^{p}(A_{R})}\|w\|_{L^{p^{*}}(A_{R})}+\cdots,
\]
where we used H\"older's inequality and $R^{-1}|A_{R}|^{\frac1p-\frac{1}{p^{*}}}\le C$, the omitted term being the same expression with $\nabla u$ replaced by $\nabla v$. Both terms tend to $0$ as $R\to\infty$ because $\nabla u,\nabla v\in L^{p}(\Omega)$ and $w\in L^{p^{*}}(\Omega)$. Hence the inequality holds for $\varphi=w$. Testing with $\varphi=(u-v)^{+}$ gives
\[
\int_{\Omega}\langle|\nabla u|^{p-2}\nabla u-|\nabla v|^{p-2}\nabla v,\nabla(u-v)^{+}\rangle\,dx\le0 .
\]
The integrand is nonnegative and vanishes only where $\nabla(u-v)^{+}=0$, so $(u-v)^{+}$ is constant, and it belongs to $D^{1,p}(\Omega)$, so it vanishes.
\end{proof}

We also record the exact $p$-Laplacian of a radial power, which we use as a barrier. For $\nu>0$ and $M>0$,
\begin{equation}\label{eq:barrier}
-\Delta_{p}\big(M|x|^{-\nu}\big)=h(\nu)M^{p-1}|x|^{-\nu(p-1)-p},
\end{equation}
where $h$ is defined in \eqref{eq:hdef}; in particular $h(\nu)>0$ exactly for $0<\nu<\gamma_h$.

For the singular range $1<p<2$ we shall also use the following sharper form of the monotonicity inequality. For every $\xi,\eta\in\mathbb R^N$,
\begin{equation}\label{eq:sharp-monotonicity-singular}
\big\langle |\xi|^{p-2}\xi-|\eta|^{p-2}\eta,\xi-\eta\big\rangle
\ge (p-1)\max\{|\xi|,|\eta|\}^{p-2}|\xi-\eta|^2.
\end{equation}
Indeed, integrating the Hessian of $\xi\mapsto |\xi|^p/p$ along the segment joining $\eta$ and $\xi$, its smallest eigenvalue is $(p-1)|z|^{p-2}$; since $p-2<0$ and every point on the segment has norm at most $\max\{|\xi|,|\eta|\}$, \eqref{eq:sharp-monotonicity-singular} follows. This form, rather than the usual estimate with $(|\xi|+|\eta|)^{p-2}$, preserves the sharp constant $p-1$ and will be useful at the critical tail in the slow phase.

We also record a logarithmically weighted Hardy inequality. Let $s>2-N$ and $\ell\in\mathbb R$. For every $\varepsilon>0$ there exists $R_{\varepsilon}>e$ such that
\begin{equation}\label{eq:log-hardy}
\int_{|x|>R}|x|^{s-2}(\log|x|)^{\ell}\varphi^2\,dx
\le \left(\frac{2}{N+s-2}+\varepsilon\right)^2
\int_{|x|>R}|x|^{s}(\log|x|)^{\ell}|\nabla\varphi|^2\,dx
\end{equation}
for $R\ge R_{\varepsilon}$ and $\varphi\in C_c^1(\mathbb R^N\setminus B_R)$. To see this, apply integration by parts to the vector field
$|x|^{s-2}(\log|x|)^\ell x$; its divergence equals
\[
\Big(N+s-2+\frac{\ell}{\log|x|}\Big)|x|^{s-2}(\log|x|)^\ell,
\]
and then use Cauchy's inequality. When $\ell=0$ the divergence equals $(N+s-2)|x|^{s-2}$ exactly, so \eqref{eq:log-hardy} holds with $\varepsilon=0$ and for every $R>0$. By density, \eqref{eq:log-hardy} is also valid for the reflected positive parts used below.

For the moving-plane argument we need a version which does not impose an artificial zero trace on the sphere $|x|=R$. We shall use the following cut-off consequence. Let $s>2-N$ and $\ell\in\mathbb R$. For every $\varepsilon>0$ there are $R_{\varepsilon}>e$ and $C_{\varepsilon}>0$ such that, for every $R\ge R_{\varepsilon}$ and every locally $W^{1,2}$ function $\varphi$ for which the right-hand side below is finite,
\begin{align}
\int_{|x|>2R}|x|^{s-2}(\log|x|)^{\ell}\varphi^{2}\,dx
&\le \left[\left(\frac{2}{N+s-2}\right)^{2}+\varepsilon\right]
\int_{|x|>R}|x|^{s}(\log|x|)^{\ell}|\nabla\varphi|^{2}\,dx \notag\\
&\quad +C_{\varepsilon}R^{s-2}(\log R)^{\ell}
\int_{R<|x|<2R}\varphi^{2}\,dx .
\label{eq:cutoff-hardy}
\end{align}
For $\ell=0$ the same estimate holds without the restriction $R>e$. Indeed, choose a radial cut-off $\eta_R$ with $\eta_R=0$ on $B_R$, $\eta_R=1$ outside $B_{2R}$ and $|\nabla\eta_R|\le C/R$, apply \eqref{eq:log-hardy} to $\eta_R\varphi$, and use Cauchy's inequality with an arbitrarily small parameter. On the annulus $R<|x|<2R$ the two weights are comparable, which gives the last term in \eqref{eq:cutoff-hardy}. An additional outer cut-off and passage to the limit justify the estimate for the noncompact reflected positive parts used below. Notice that the coefficient of the gradient term can be made arbitrarily close to the sharp Hardy constant; this will be essential in the slow phase.

In every application below, $\varphi$ is a reflected positive part $w=(u-u_\lambda)^+$ and $s=(2-p)(\gamma+1)$, where $\gamma=\gamma_h$ in the classical and resonant phases and $\gamma=\gamma_*$ in the slow phase, and the pointwise bounds $0\le w\le C|x|^{-\gamma}$ and $|\nabla w|\le C|x|^{-\gamma-1}$ hold on $\{w>0\}$ at infinity, up to logarithmic factors in the resonant case that do not affect the strict power inequalities below; these bounds are established where the inequality is applied. In all three cases $N+s-2=\gamma+\big(N-p-(p-1)\gamma\big)>0$, so $s>2-N$. A direct computation gives $N+s-2-2\gamma=p(\gamma_e-\gamma)$. The outer cut-off at radius $T$ therefore costs
\[
CT^{s-2}\int_{T<|x|<2T}w^{2}\le CT^{N+s-2-2\gamma}=CT^{p(\gamma_e-\gamma)},
\]
which tends to zero because $\gamma>\gamma_e$ for both choices of $\gamma$. The same exponent shows that $\int|x|^{s}|\nabla w|^{2}<\infty$, since the integrand is $O(|x|^{s-2\gamma-2})$. Finally, when $1<p<2$ one has $\gamma>\gamma_e>\frac{N-2}{2}$, so that $w\in D^{1,2}$ of the relevant half-space.

\section{The singular set and a blow-up dichotomy}\label{sec:structure}

We begin with the regularity issue that remains outside the asymptotic theory. Define
\begin{equation}\label{eq:Sdef}
S:=\big\{x\in\mathbb R^N:\ u\notin L^\infty(B_r(x))\text{ for every }r>0\big\},
\end{equation}
and write $d(y)=\operatorname{dist}(y,S)$, with $d\equiv+\infty$ when $S=\emptyset$.

\begin{lemma}[Positivity alternative]\label{lem:positivity-alternative}
Let $N\ge3$, $p>1$ and $2p<\sigma<N$, and assume that $q:=p^*_{\sigma}>1$. Let $u$ be a nonnegative $D^{1,p}(\mathbb R^N)$-weak solution of \eqref{eq:main}. Then either $u\equiv0$ a.e. in $\mathbb R^N$ or $u>0$ a.e. in $\mathbb R^N$. In the latter case the continuous representative of $u$ on $\mathbb R^N\setminus S$ is strictly positive at every point.
\end{lemma}

\begin{proof}
Since the right-hand side of \eqref{eq:main} is nonnegative, $u$ is a nonnegative weak supersolution of $-\Delta_pu\ge0$. The weak Harnack inequality for nonnegative $p$-supersolutions implies that on concentric balls $\overline{B_r}\subset B_{2r}$,
\[
\left(\frac1{|B_{2r}|}\int_{B_{2r}}u^\theta\right)^{1/\theta}\le C\,\operatorname*{ess\,inf}_{B_r}u
\]
for some $\theta>0$. If $u\not\equiv0$, one ball has positive $L^\theta$ mass, hence $u$ has positive essential infimum on a smaller ball. A chain of overlapping balls propagates this positivity through the connected space $\mathbb R^N$, giving $u>0$ a.e. Finally, on $\mathbb R^N\setminus S$ the local $C^{1,\alpha}$ representative furnished by interior regularity agrees a.e. with $u$, so the same local weak-Harnack estimate gives pointwise positivity there.
\end{proof}

\begin{theorem}\label{thm:structure}
Let $N\ge3$, $p>1$ and $2p<\sigma<N$, and assume that $q:=p^*_{\sigma}>1$. Let $u$ be a nonnegative $D^{1,p}(\mathbb{R}^{N})$-weak solution of \eqref{eq:main}. Then $u\in C^{1,\alpha}_{\mathrm{loc}}(\mathbb{R}^{N}\setminus S)$ for some $0<\alpha<\min\{1,\frac{1}{p-1}\}$, and there is $\Lambda_{0}=\Lambda_{0}(u)<\infty$ with
\begin{equation}\label{eq:interior}
u(y)\,\min\{d(y),1\}^{\gamma_{e}}\le\Lambda_{0}
\qquad\text{for every }y\in\mathbb{R}^{N}\setminus S .
\end{equation}
Moreover
\begin{equation}\label{eq:vanish}
u(y)\,d(y)^{\gamma_{e}}\longrightarrow0
\qquad\text{as }d(y)\to0 .
\end{equation}
\end{theorem}

\begin{corollary}\label{cor:structure}
Under the assumptions of Theorem \ref{thm:structure}, suppose $S=\emptyset$, that is, $u\in L^{\infty}_{\mathrm{loc}}(\mathbb{R}^{N})$. Then
\[
u\in L^{\infty}(\mathbb{R}^{N})\cap C^{1,\alpha}(\mathbb{R}^{N}),
\qquad
\|u\|_{L^{\infty}(\mathbb{R}^{N})}\le\Lambda_{0},
\qquad
u(x)\longrightarrow0\ \text{ as }|x|\to\infty .
\]
\end{corollary}

Throughout this section $N\ge3$, $p>1$, $2p<\sigma<N$, $q>1$, and $u$ is a nonnegative $D^{1,p}(\mathbb{R}^{N})$-weak solution of \eqref{eq:main}. No further assumption is made. Recall the closed set $S$ of \eqref{eq:Sdef} and $d(y)=\mathrm{dist}(y,S)$.

\begin{lemma}\label{lem:Wloc}
Let $u\in L^{p^{*}}(\mathbb{R}^{N})$ be nonnegative, let $x_{0}\in\mathbb{R}^{N}$ and $0<r\le1$, and suppose $u\in L^{\infty}(B_{2r}(x_{0}))$. Then
\[
\|W\|_{L^{\infty}(B_{r}(x_{0}))}
\le C\Big(\|u\|^{q}_{L^{\infty}(B_{2r}(x_{0}))}\,r^{N-\sigma}
+\|u\|^{q}_{L^{p^{*}}(\mathbb{R}^{N})}\,r^{-\sigma/2}\Big),
\]
with $C=C(N,\sigma,q)$. In particular $W\in L^{\infty}_{\mathrm{loc}}(\mathbb{R}^{N}\setminus S)$.
\end{lemma}

\begin{proof}
Fix $x\in B_{r}(x_{0})$ and split $W(x)=W_{1}(x)+W_{2}(x)$ according to $|x-y|<r$ and $|x-y|\ge r$. If $|x-y|<r$ then $y\in B_{2r}(x_{0})$, so
\[
W_{1}(x)\le\|u\|^{q}_{L^{\infty}(B_{2r}(x_{0}))}\int_{|z|<r}|z|^{-\sigma}dz
=C\|u\|^{q}_{L^{\infty}(B_{2r}(x_{0}))}r^{N-\sigma},
\]
the integral being finite because $\sigma<N$. For $W_{2}$, H\"older's inequality with the exponents $p^{*}/q$ and $2N/\sigma$, which are conjugate because $q/p^{*}=(2N-\sigma)/(2N)$, gives
\[
W_{2}(x)\le\|u\|^{q}_{L^{p^{*}}(\mathbb{R}^{N})}\Big(\int_{|z|\ge r}|z|^{-2N}dz\Big)^{\frac{\sigma}{2N}}
=C\|u\|^{q}_{L^{p^{*}}(\mathbb{R}^{N})}r^{-\sigma/2}.
\]
Both bounds are independent of $x\in B_{r}(x_{0})$.
\end{proof}

\begin{lemma}\label{lem:regoffS}
We have $u\in C^{1,\alpha}_{\mathrm{loc}}(\mathbb{R}^{N}\setminus S)$ for some $0<\alpha<\min\{1,\frac1{p-1}\}$.
\end{lemma}

\begin{proof}
Let $x_{0}\notin S$, so $u\in L^{\infty}(B_{2r}(x_{0}))$ for some $0<r\le1$. By Lemma \ref{lem:Wloc}, $W\in L^{\infty}(B_{r}(x_{0}))$, hence $Wu^{q-1}\in L^{\infty}(B_{r}(x_{0}))$ because $q>1$. The regularity results of \cite{DiBenedetto1983,Tolksdorf1984} apply.
\end{proof}

The scaling that leaves \eqref{eq:main} invariant is
\begin{equation}\label{eq:scaling}
u_{y,\lambda}(z):=\lambda^{\gamma_{e}}u(y+\lambda z),
\qquad y\in\mathbb{R}^{N},\ \lambda>0 .
\end{equation}

\begin{lemma}\label{lem:scaling}
If $u$ solves \eqref{eq:main}, so does $u_{y,\lambda}$, and for every $R>0$
\begin{equation}\label{eq:scaleinv}
\int_{B_{R}}|\nabla u_{y,\lambda}|^{p}dz=\int_{B_{\lambda R}(y)}|\nabla u|^{p}dx,
\qquad
\int_{B_{R}}u_{y,\lambda}^{p^{*}}dz=\int_{B_{\lambda R}(y)}u^{p^{*}}dx .
\end{equation}
In particular $\|u_{y,\lambda}\|_{L^{p^{*}}(\mathbb{R}^{N})}=\|u\|_{L^{p^{*}}(\mathbb{R}^{N})}$.
\end{lemma}

\begin{proof}
A change of variables gives $W[u_{y,\lambda}](z)=\lambda^{q\gamma_{e}-N+\sigma}W[u](y+\lambda z)$, so the right hand side of \eqref{eq:main} evaluated at $u_{y,\lambda}$ equals $\lambda^{(2q-1)\gamma_{e}+\sigma-N}\,(W[u]u^{q-1})(y+\lambda z)$, while $-\Delta_{p}u_{y,\lambda}=\lambda^{\gamma_{e}(p-1)+p}(-\Delta_{p}u)(y+\lambda z)$. The two exponents agree because $\gamma_{e}(2q-p)=N+p-\sigma$, which is the definition of $\gamma_{e}$ rewritten. The identities \eqref{eq:scaleinv} follow from $(\gamma_{e}+1)p=N$ and $\gamma_{e}p^{*}=N$.
\end{proof}

We shall use the doubling lemma of Pol\'a\v{c}ik, Quittner and Souplet, in the following form, see \cite[Lemma 5.1]{PolacikQuittnerSouplet2007}.

\begin{lemma}[{\cite[Lemma 5.1]{PolacikQuittnerSouplet2007}}]\label{lem:doubling}
Let $(X,\rho)$ be a complete metric space, let $\emptyset\ne D\subset\Sigma\subset X$ with $\Sigma$ closed, and set $\Gamma:=\Sigma\setminus D$. Let $M:D\to(0,\infty)$ be bounded on compact subsets of $D$, and let $k>0$. If $y_{0}\in D$ satisfies $M(y_{0})\,\mathrm{dist}(y_{0},\Gamma)>2k$, then there is $x\in D$ with
\[
M(x)\,\mathrm{dist}(x,\Gamma)>2k,
\qquad
M(x)\ge M(y_{0}),
\qquad
M(z)\le2M(x)\ \text{ for all }z\in D\cap \overline{B_{k/M(x)}(x)} .
\]
\end{lemma}

Finally we record the uniform smallness of the $L^{p^{*}}$ mass on small balls, which is where the finiteness of the energy enters.

\begin{lemma}\label{lem:unifsmall}
Let $u\in L^{p^{*}}(\mathbb{R}^{N})$. Then $\omega(\rho):=\sup_{x\in\mathbb{R}^{N}}\int_{B_{\rho}(x)}u^{p^{*}}dy\to0$ as $\rho\to0$.
\end{lemma}

\begin{proof}
Given $\varepsilon>0$, pick $g\in C_{0}(\mathbb{R}^{N})$ with $\|u-g\|_{L^{p^{*}}}^{p^{*}}<\varepsilon$. Then $\int_{B_{\rho}(x)}u^{p^{*}}\le2^{p^{*}}(\varepsilon+\|g\|^{p^{*}}_{\infty}|B_{\rho}|)$ for every $x$, and the last term tends to $0$.
\end{proof}

\begin{proof}[Proof of Theorem \ref{thm:structure}]
The regularity is Lemma \ref{lem:regoffS}. If $u\equiv0$, both \eqref{eq:interior} and \eqref{eq:vanish} are immediate. Hence assume $u\not\equiv0$. By Lemma \ref{lem:positivity-alternative}, the continuous representative of $u$ is strictly positive on $D:=\mathbb R^N\setminus S$.

\emph{Step 1: the interior bound.} Set $M:=u^{1/\gamma_{e}}$. Then $M:D\to(0,\infty)$ is continuous by Lemma \ref{lem:regoffS}, hence bounded on compact subsets of $D$. Suppose, for contradiction, that
\[
\sup_{y\in D}\ M(y)\,\min\{\mathrm{dist}(y,S),1\}=\infty .
\]
Fix $k\in\mathbb{N}$ and choose $y_k\in D$ so that
\[
M(y_k)\min\{\operatorname{dist}(y_k,S),1\}>2k.
\]
Set
\[
\Sigma_k:=\overline{B_1(y_k)},\qquad
D_k:=B_1(y_k)\setminus S,
\qquad
\Gamma_k:=\Sigma_k\setminus D_k=(S\cap\Sigma_k)\cup\partial B_1(y_k).
\]
Then $\Sigma_k$ is closed, $y_k\in D_k$, and
\[
\operatorname{dist}(y_k,\Gamma_k)
=\min\{\operatorname{dist}(y_k,S),1\},
\]
so $M(y_k)\operatorname{dist}(y_k,\Gamma_k)>2k$. Lemma \ref{lem:doubling} produces $x_{k}\in D_{k}$ with
\[
M(x_{k})\,\mathrm{dist}(x_{k},\Gamma_{k})>2k,
\qquad
M(z)\le2M(x_{k})\ \text{ for }z\in D_{k}\cap \overline{B_{k/M(x_{k})}(x_{k})} .
\]
Since $\mathrm{dist}(x_{k},\Gamma_{k})\le1$ we get $\lambda_{k}:=1/M(x_{k})<1/(2k)\to0$, and $\overline{B_{k\lambda_{k}}(x_{k})}\subset D_{k}$.

Let $v_{k}:=u_{x_{k},\lambda_{k}}$ as in \eqref{eq:scaling}. Then $v_{k}$ solves \eqref{eq:main}, and
\[
v_{k}(0)=\lambda_{k}^{\gamma_{e}}u(x_{k})=1,
\qquad
0\le v_{k}(z)=\lambda_{k}^{\gamma_{e}}u(x_{k}+\lambda_{k}z)\le2^{\gamma_{e}}
\quad\text{for }|z|\le k .
\]
By Lemma \ref{lem:Wloc} applied to $v_{k}$ with $r=1$ and centers in $B_{k-2}$, and by $\|v_{k}\|_{L^{p^{*}}(\mathbb{R}^{N})}=\|u\|_{L^{p^{*}}(\mathbb{R}^{N})}$ from Lemma \ref{lem:scaling}, the potentials $W[v_{k}]$ are bounded on $B_{k-2}$ by a constant independent of $k$. Since $q>1$ and $\|v_{k}\|_{L^{\infty}(B_{k})}\le2^{\gamma_{e}}$, the right hand sides $W[v_{k}]v_{k}^{q-1}$ are bounded in $L^{\infty}(B_{k-2})$ uniformly in $k$, so \cite{DiBenedetto1983,Tolksdorf1984} give, for each fixed $R>0$, a bound for $\|v_{k}\|_{C^{1,\alpha}(\overline{B_{R}})}$ independent of $k$.

On the other hand, by \eqref{eq:scaleinv} and Lemma \ref{lem:unifsmall},
\[
\int_{B_{R}}v_{k}^{p^{*}}dz=\int_{B_{R\lambda_{k}}(x_{k})}u^{p^{*}}dx\le\omega(R\lambda_{k})\longrightarrow0
\qquad(k\to\infty)
\]
for every fixed $R$. By the Arzel\`a--Ascoli theorem and a diagonal argument, a subsequence of $(v_{k})$ converges in $C^{1}_{\mathrm{loc}}(\mathbb{R}^{N})$ to some $v_{\infty}\ge0$ with $v_{\infty}(0)=1$. Fatou's lemma then gives $\int_{B_{R}}v_{\infty}^{p^{*}}=0$ for every $R$, so $v_{\infty}\equiv0$, contradicting $v_{\infty}(0)=1$. This proves \eqref{eq:interior} with some finite $\Lambda_{0}$.

\emph{Step 2: the vanishing factor.} Suppose \eqref{eq:vanish} fails. Then there are $\delta>0$ and $y_{j}\in D$ with $d_{j}:=d(y_{j})\to0$ and $u(y_{j})d_{j}^{\gamma_{e}}\ge\delta$. Put $\lambda_{j}:=d_{j}/2$ and $v_{j}:=u_{y_{j},\lambda_{j}}$. For $|z|\le1$ we have $\mathrm{dist}(y_{j}+\lambda_{j}z,S)\ge d_{j}-\lambda_{j}=\lambda_{j}$, so \eqref{eq:interior} gives $u(y_{j}+\lambda_{j}z)\le\Lambda_{0}\lambda_{j}^{-\gamma_{e}}$ for $j$ large, that is, $v_{j}\le\Lambda_{0}$ on $B_{1}$. Exactly as in Step 1, the $v_{j}$ are bounded in $C^{1,\alpha}(\overline{B_{1/2}})$, while
\[
\int_{B_{1}}v_{j}^{p^{*}}dz=\int_{B_{\lambda_{j}}(y_{j})}u^{p^{*}}dx\le\omega(\lambda_{j})\to0 .
\]
A subsequence converges uniformly on $\overline{B_{1/2}}$ to a continuous $v_{\infty}\ge0$ with $\int_{B_{1/2}}v_{\infty}^{p^{*}}=0$, hence $v_{\infty}\equiv0$. But $v_{j}(0)=\lambda_{j}^{\gamma_{e}}u(y_{j})=2^{-\gamma_{e}}d_{j}^{\gamma_{e}}u(y_{j})\ge2^{-\gamma_{e}}\delta$, a contradiction.
\end{proof}

\begin{proof}[Proof of Corollary \ref{cor:structure}]
Assume $S=\emptyset$, so $d\equiv+\infty$, $\min\{d,1\}\equiv1$, and \eqref{eq:interior} reads $u\le\Lambda_{0}$ on $\mathbb{R}^{N}$. Hence $u\in L^{\infty}(\mathbb{R}^{N})$. Lemma \ref{lem:Wloc}, applied with $r=1$ and arbitrary centers, gives $W\in L^{\infty}(\mathbb{R}^{N})$; since $q>1$, $Wu^{q-1}\in L^{\infty}(\mathbb{R}^{N})$, and therefore $u\in C^{1,\alpha}(\mathbb{R}^{N})$ by \cite{DiBenedetto1983,Tolksdorf1984}.

For the decay, suppose $|y_{j}|\to\infty$ and $u(y_{j})\ge\delta>0$. The translates $v_{j}:=u(y_{j}+\cdot)$ solve \eqref{eq:main}, since $W[v_{j}](z)=W[u](y_{j}+z)$, and they satisfy $\|v_{j}\|_{L^{\infty}}\le\Lambda_{0}$ and $\|v_{j}\|_{L^{p^{*}}(\mathbb{R}^{N})}=\|u\|_{L^{p^{*}}(\mathbb{R}^{N})}$. As in Step 1 of the proof of Theorem \ref{thm:structure}, $(v_{j})$ is bounded in $C^{1,\alpha}(\overline{B_{1/2}})$, while
\[
\int_{B_{1}}v_{j}^{p^{*}}dz=\int_{B_{1}(y_{j})}u^{p^{*}}dx\longrightarrow0
\]
because $u\in L^{p^{*}}(\mathbb{R}^{N})$ and the balls escape to infinity. A subsequence converges uniformly on $\overline{B_{1/2}}$ to a continuous limit with vanishing $L^{p^{*}}$ norm, hence to $0$, contradicting $v_{j}(0)\ge\delta$.
\end{proof}

\subsection{A frozen-weight obstruction}

The following frozen-weight example explains why local boundedness cannot be deduced from the natural integrability class of the Hartree potential alone; the self-consistency of the Hartree potential must be used.

\begin{proposition}\label{prop:frozen}
Let $p>1$ and $2p<\sigma<N$, and let $q$ be as in \eqref{eq:psigma}. For every $a\in(p,\sigma/2)$ put
\[
\kappa:=\frac{a-p}{p-q}\in\Big(0,\tfrac{N-p}{p}\Big),
\qquad
c:=\big[\kappa^{p-1}\big(N-p-(p-1)\kappa\big)\big]^{-\frac{1}{p-q}} .
\]
Then $u(x)=c|x|^{-\kappa}$ belongs to $W^{1,p}(B_{1})\cap L^{p^{*}}(B_{1})$, is unbounded near the origin, and solves
\[
-\Delta_{p}u=\omega(x)u^{q-1}\quad\text{in }B_{1},
\qquad \omega(x)=|x|^{-a}\in L^{\frac{2N}{\sigma}}(B_{1}).
\]
The interval $(p,\sigma/2)$ is nonempty precisely when $\sigma>2p$.
\end{proposition}

\begin{proof}[Proof of Proposition \ref{prop:frozen}]
Let $a\in(p,\sigma/2)$, so that $\kappa=(a-p)/(p-q)>0$ because $q<p$ and $a>p$. From $p-q=p(\sigma-2p)/(2(N-p))$ we get
\[
\kappa<\frac{N-p}{p}\iff a-p<\frac{(N-p)(p-q)}{p}=\frac{\sigma-2p}{2}\iff a<\frac{\sigma}{2},
\]
so $\kappa\in(0,\gamma_{e})$. In particular $\kappa<\gamma_{h}$, so $h(\kappa)=\kappa^{p-1}(N-p-(p-1)\kappa)>0$ and, by \eqref{eq:barrier},
\[
-\Delta_{p}\big(c|x|^{-\kappa}\big)=h(\kappa)c^{p-1}|x|^{-\kappa(p-1)-p} .
\]
Since $\kappa(p-q)=a-p$, we have $\kappa(p-1)+p=a+(q-1)\kappa$, so the right hand side equals $h(\kappa)c^{p-q}\cdot|x|^{-a}\big(c|x|^{-\kappa}\big)^{q-1}$, and the stated choice of $c$ makes $h(\kappa)c^{p-q}=1$. The singularity is not a Dirac mass because $\kappa<\gamma_{h}$, and the right hand side is in $L^{1}(B_{1})$ because $\kappa(p-1)+p<N$, so $u$ is a weak solution in $B_{1}$. Finally $\int_{B_{1}}|\nabla u|^{p}\simeq\int_{0}^{1}r^{N-1-(\kappa+1)p}dr<\infty$ and $\int_{B_{1}}u^{p^{*}}<\infty$, both because $\kappa<\gamma_{e}$, while $|x|^{-a}\in L^{2N/\sigma}(B_{1})$ because $2Na/\sigma<N$ is equivalent to $a<\sigma/2$. The interval $(p,\sigma/2)$ is nonempty if and only if $\sigma>2p$.
\end{proof}

\section{Automatic local boundedness}\label{sec:localboundedness}

Throughout this section, $N\ge3$, $p>1$, $2p<\sigma<N$, $q:=p^*_{\sigma}>1$, and $u$ is a nonnegative $D^{1,p}(\mathbb R^N)$-weak solution of \eqref{eq:main}. We now close the local-regularity gap left open by the blow-up discussion above. The proof exploits the self-consistency of the Hartree potential twice: first in a hole-filling estimate for the Dirichlet energy and then in a Moser iteration in which the integrability of the potential improves together with that of the solution.

We shall use the following elementary consequences of the critical exponent identities:
\begin{equation}\label{eq:lb-identities}
\frac{q}{p^*}=\frac{2N-\sigma}{2N},\qquad
q\gamma_e=N-\frac\sigma2>\frac\sigma2,
\qquad
2q-p^*=\frac{p(N-\sigma)}{N-p}>0,
\end{equation}
\begin{equation}\label{eq:lb-identities2}
\frac{2q-1}{p-1}>1,
\qquad
p-q=\frac{p(\sigma-2p)}{2(N-p)}>0.
\end{equation}
In particular, if $t':=2N/(2N-\sigma)$, then $qt'=p^*$.

\begin{lemma}\label{lem:lb-admissible}
Let $u\in D^{1,p}(\mathbb R^N)$ be a nonnegative weak solution of \eqref{eq:main}, with $q>1$, and put $f=W[u]u^{q-1}$. Then
\[
W[u]\in L^{2N/\sigma}(\mathbb R^N),
\qquad
f\in L^{(p^*)'}(\mathbb R^N).
\]
Consequently the weak formulation \eqref{eq:weak} extends by density to every compactly supported $\varphi\in L^{p^*}$ with $\nabla\varphi\in L^p$.
\end{lemma}

\begin{proof}
The Hardy--Littlewood--Sobolev inequality applied to $u^q$, with $p^*/q=2N/(2N-\sigma)$, gives
\[
\|W[u]\|_{L^{2N/\sigma}}\le C\|u\|_{L^{p^*}}^q.
\]
Moreover
\[
\frac{\sigma}{2N}+\frac{q-1}{p^*}=1-\frac1{p^*},
\]
so H\"older's inequality gives $f\in L^{(p^*)'}$. The extension of the test class follows from density in $D^{1,p}$ on compact subsets.
\end{proof}

\begin{lemma}[Far-field bound]\label{lem:lb-far}
For every $x_0\in\mathbb R^N$, $\rho>0$ and $x\in B_{\rho/2}(x_0)$,
\[
\int_{|y-x_0|\ge\rho}\frac{u^q(y)}{|x-y|^\sigma}\,dy
\le C\|u\|_{L^{p^*}(\mathbb R^N)}^q\rho^{-\sigma/2}.
\]
\end{lemma}

\begin{proof}
For such $x,y$ one has $|x-y|\ge |y-x_0|/2$. H\"older's inequality with the conjugate exponents $p^*/q$ and $2N/\sigma$ yields
\[
\int_{|y-x_0|\ge\rho}\frac{u^q(y)}{|x-y|^\sigma}\,dy
\le C\|u\|_{L^{p^*}}^q
\left(\int_{|z|\ge\rho}|z|^{-2N}\,dz\right)^{\sigma/(2N)},
\]
which is the asserted estimate.
\end{proof}

Fix $x_0\in\mathbb R^N$ and write
\[
\mathcal D(r):=\int_{B_r(x_0)}|\nabla u|^p,
\qquad
\mathcal Z(r):=\|u\|_{L^{p^*}(B_r(x_0))}.
\]

\begin{lemma}[Hole filling]\label{lem:lb-hole}
There are $\theta_0\in(0,1)$ and $C>0$, depending only on $N,p$, such that
\begin{equation}\label{eq:lb-hole}
\mathcal D(r)\le\theta_0\mathcal D(2r)+C\,T(r),
\end{equation}
where $A_r=B_{2r}(x_0)\setminus B_r(x_0)$, $\bar u_{A_r}$ denotes the mean of $u$ over $A_r$, and
\[
T(r):=\int_{B_{2r}(x_0)}W[u]u^{q-1}(u-\bar u_{A_r})^+.
\]
Moreover
\begin{equation}\label{eq:lb-SP}
\|(u-\bar u_{A_r})^+\|_{L^{p^*}(B_{2r}(x_0))}
\le C\mathcal D(2r)^{1/p}.
\end{equation}
\end{lemma}

\begin{proof}
Choose $\eta\in C_c^\infty(B_{2r}(x_0))$ with $0\le\eta\le1$, $\eta=1$ on $B_r(x_0)$ and $|\nabla\eta|\le2/r$. By Lemma \ref{lem:lb-admissible}, $(u-\bar u_{A_r})\eta^p$ is an admissible test function. Since the right-hand side of the equation is nonnegative,
\[
\int\eta^p|\nabla u|^p
+p\int\eta^{p-1}(u-\bar u_{A_r})|\nabla u|^{p-2}\nabla u\cdot\nabla\eta
\le T(r).
\]
Young's inequality and the scale-invariant Poincar\'e inequality on $A_r$ give
\[
\mathcal D(r)\le C_2\big(\mathcal D(2r)-\mathcal D(r)\big)+2T(r),
\]
which yields \eqref{eq:lb-hole} after rearrangement.  To prove \eqref{eq:lb-SP}, let $\bar u_{B_{2r}}$ be the mean on $B_{2r}(x_0)$. The Sobolev--Poincar\'e inequality gives
\[
\|u-\bar u_{B_{2r}}\|_{L^{p^*}(B_{2r})}\le C\mathcal D(2r)^{1/p},
\]
and H\"older's inequality controls $|\bar u_{A_r}-\bar u_{B_{2r}}||B_{2r}|^{1/p^*}$ by the same quantity. This proves the claim.
\end{proof}

\begin{lemma}[Hartree remainder]\label{lem:lb-hartree}
Fix $\rho_0>0$. For $0<r\le\rho_0/4$,
\begin{equation}\label{eq:lb-TB}
T(r)\le C\mathcal B(r)\mathcal D(2r)^{1/p},
\end{equation}
where, with $C_u:=\|u\|_{L^{p^*}(\mathbb R^N)}^q$,
\begin{align*}
\mathcal B(r):={}&\mathcal Z(4r)^{2q-1}
+\left(\sum_{\substack{j\ge1\\2^{j+1}r<\rho_0}}
2^{-j\sigma/2}\mathcal Z(2^{j+2}r)^q\right)\mathcal Z(2r)^{q-1}\\
&+C_u\left(\frac r{\rho_0}\right)^{\sigma/2}\mathcal Z(2r)^{q-1}.
\end{align*}
Consequently, for every $\varepsilon>0$,
\begin{equation}\label{eq:lb-TY}
T(r)\le\varepsilon\mathcal D(2r)+C_\varepsilon\mathcal B(r)^{p'}.
\end{equation}
\end{lemma}

\begin{proof}
Split $W=W_1+W_2+W_3$ according to the regions $B_{4r}(x_0)$, $B_{\rho_0}(x_0)\setminus B_{4r}(x_0)$ and $\mathbb R^N\setminus B_{\rho_0}(x_0)$, and let $T_i$ denote the corresponding parts of $T(r)$. Put
\[
g:=u^{q-1}(u-\bar u_{A_r})^+\chi_{B_{2r}(x_0)}.
\]
For the near part, HLS and $qt'=p^*$ give
\[
T_1\le C\|g\|_{L^{t'}}\|u^q\|_{L^{t'}(B_{4r})}
\le C\mathcal Z(4r)^{2q-1}\mathcal D(2r)^{1/p},
\]
where we used \eqref{eq:lb-SP} and $q>1$.

We shall also use
\begin{equation}\label{eq:lb-common}
\int_{B_{2r}}u^{q-1}(u-\bar u_{A_r})^+
\le C\mathcal Z(2r)^{q-1}\mathcal D(2r)^{1/p}r^{\sigma/2},
\end{equation}
which follows from H\"older's inequality, \eqref{eq:lb-SP}, and
$N(1-q/p^*)=\sigma/2$.

For the middle part, cover $B_{\rho_0}(x_0)\setminus B_{4r}(x_0)$ by the dyadic annuli
\[
D_j=\{2^{j+1}r\le |y-x_0|<2^{j+2}r\},
\qquad j\ge1,
\]
with $2^{j+1}r<\rho_0$. For $x\in B_{2r}(x_0)$, H\"older's inequality gives
\[
\int_{D_j}\frac{u^q(y)}{|x-y|^\sigma}\,dy
\le C(2^jr)^{-\sigma/2}\mathcal Z(2^{j+2}r)^q.
\]
Combining this with \eqref{eq:lb-common} gives the second term in \eqref{eq:lb-TB}. The far part follows from Lemma \ref{lem:lb-far} and \eqref{eq:lb-common}. Summing the three contributions proves \eqref{eq:lb-TB}; Young's inequality gives \eqref{eq:lb-TY}.
\end{proof}

\begin{lemma}[Control of the local $L^{p^*}$ norm]\label{lem:lb-chain}
Let $r_k=2^{-k}\rho_0$, $e_k=\mathcal D(r_k)$ and $z_k=\mathcal Z(r_k)$. Then
\begin{equation}\label{eq:lb-chain}
z_k\le C\left(\sum_{i=0}^k2^{-(k-i)\gamma_e}e_i^{1/p}
+2^{-k\gamma_e}z_0\right).
\end{equation}
\end{lemma}

\begin{proof}
Let $\bar u_k$ be the mean of $u$ on $B_{r_k}(x_0)$. Sobolev--Poincar\'e gives
\[
z_k\le Ce_k^{1/p}+C|\bar u_k|r_k^{\gamma_e}.
\]
Moreover
\[
|\bar u_{i+1}-\bar u_i|\le Cr_i^{-\gamma_e}e_i^{1/p},
\qquad
|\bar u_0|\le C\rho_0^{-\gamma_e}z_0.
\]
Telescoping the means and using $r_k/r_i=2^{-(k-i)}$ yields \eqref{eq:lb-chain}.
\end{proof}

\begin{lemma}[Discrete decay]\label{lem:lb-discrete}
Let $(e_k)$ be nonnegative, $c_*>0$, and
\[
\widetilde E_k:=\max_{0\le i\le k}2^{-c_*(k-i)}e_i.
\]
If for some $\theta\in(0,1)$, $\beta>0$ and $A\ge0$,
\[
e_{k+1}\le\theta\widetilde E_k+A2^{-k\beta}\qquad(k\ge1),
\]
then there is $\alpha>0$ and $C>0$ such that $e_k\le\widetilde E_k\le C2^{-k\alpha}$.
\end{lemma}

\begin{proof}
Since
\[
\widetilde E_{k+1}
=\max\{e_{k+1},2^{-c_*}\widetilde E_k\}
\le\theta_2\widetilde E_k+A2^{-k\beta},
\qquad
\theta_2:=\max\{\theta,2^{-c_*}\}<1,
\]
iteration gives
\[
\widetilde E_k\le (\widetilde E_1+Ak)\Lambda^{k-1},
\qquad
\Lambda:=\max\{\theta_2,2^{-\beta}\}<1.
\]
Any $0<\alpha<-\log_2\Lambda$ gives the conclusion.
\end{proof}

\begin{proposition}[Morrey decay of the energy]\label{prop:lb-morrey}
For every compact $K\subset\mathbb R^N$ there are
\[
\alpha\in(0,N-p),\qquad \rho_0\in(0,1],\qquad C>0,
\]
where $\alpha$ may be chosen depending only on $N,p,\sigma$, such that
\begin{equation}\label{eq:lb-morrey}
\int_{B_r(x_0)}|\nabla u|^p\le Cr^\alpha
\end{equation}
for every $x_0\in K$ and $0<r\le\rho_0$.
\end{proposition}

\begin{proof}
By the uniform absolute continuity of the integrals of $|\nabla u|^p$ and $u^{p^*}$, for every $\varepsilon_0>0$ one may choose $\rho_0\in(0,1]$, uniformly for $x_0\in K$, so that
\[
\mathcal D(\rho_0)\le\varepsilon_0^p,
\qquad
\mathcal Z(\rho_0)\le\varepsilon_0.
\]
Set
\[
c_*:=\frac12\min\left\{p\gamma_e,\frac{p\sigma}{2q}\right\},
\qquad
\beta_1:=\min\left\{\gamma_e,\frac\sigma{2q}\right\},
\qquad
P_k:=2^{-k\beta_1},
\]
and define $\widetilde E_k$ as in Lemma \ref{lem:lb-discrete}. Since $e_i\le2^{c_*(k-i)}\widetilde E_k$, Lemma \ref{lem:lb-chain} and $c_*/p<\gamma_e$ give, for $0\le m\le k$,
\begin{equation}\label{eq:lb-zm}
z_m\le C\left(2^{c_*(k-m)/p}\widetilde E_k^{1/p}+2^{-m\gamma_e}\right).
\end{equation}

Apply Lemma \ref{lem:lb-hartree} with $r=r_{k+1}$. Then $2r=r_k$, $4r=r_{k-1}$, and the dyadic sum has $1\le j\le k-1$. From \eqref{eq:lb-zm},
\[
z_k+z_{k-1}\le C\big(\widetilde E_k^{1/p}+P_k\big).
\]
For the middle sum, \eqref{eq:lb-zm} with $m=k-1-j$ gives
\begin{align*}
\sum_{j=1}^{k-1}2^{-j\sigma/2}z_{k-1-j}^q
&\le C\widetilde E_k^{q/p}
\sum_{j\ge1}2^{-j(\sigma/2-c_*q/p)}\\
&\quad+C\sum_{j=1}^{k-1}
2^{-j\sigma/2-(k-1-j)q\gamma_e}.
\end{align*}
The first series converges because $c_*q/p<\sigma/2$. Since $q\gamma_e-\sigma/2>0$, the second is a geometric sum dominated by its last term and is bounded by $C2^{-k\sigma/2}$. By the definition of $\beta_1$, $2^{-k\sigma/2}\le P_k^q$. Therefore
\[
\sum_{j=1}^{k-1}2^{-j\sigma/2}z_{k-1-j}^q
\le C\big(\widetilde E_k^{1/p}+P_k\big)^q.
\]
The final term in $\mathcal B(r_{k+1})$ is bounded by $C_u2^{-k\sigma/2}$ because $z_k\le z_0\le\varepsilon_0\le1$. Hence
\[
\mathcal B(r_{k+1})
\le C\big(\widetilde E_k^{1/p}+P_k\big)^{2q-1}
+C_u2^{-k\sigma/2}.
\]
Raising this inequality to $p'$ and separating the pure dyadic powers gives
\begin{equation}\label{eq:lb-Bpower}
\mathcal B(r_{k+1})^{p'}
\le C\widetilde E_k^{(2q-1)/(p-1)}+C2^{-k\beta_2}
\end{equation}
for some $\beta_2>0$, where the last constant may depend on $u$. By \eqref{eq:lb-identities2},
\[
\widetilde E_k^{(2q-1)/(p-1)}
=\widetilde E_k^{(2q-p)/(p-1)}\widetilde E_k
\le\varepsilon_0^{p(2q-p)/(p-1)}\widetilde E_k.
\]
Combining \eqref{eq:lb-hole}, \eqref{eq:lb-TY} and \eqref{eq:lb-Bpower}, first choose the Young parameter so that the coefficient inherited from \eqref{eq:lb-hole} remains $<1$, and then choose $\varepsilon_0$ so small that the superlinear term above is absorbed. We obtain
\[
e_{k+1}\le\theta\widetilde E_k+A2^{-k\beta}
\]
with $\theta\in(0,1)$ and $\beta>0$, uniformly in $x_0\in K$. Lemma \ref{lem:lb-discrete} therefore yields $e_k\le C2^{-k\alpha}$ for some $\alpha>0$. Decreasing $\alpha$ if necessary, we may assume $\alpha<N-p$. Interpolation between consecutive dyadic radii gives \eqref{eq:lb-morrey}.
\end{proof}

\begin{proposition}[First integrability gain]\label{prop:lb-first-gain}
There exists $\widetilde p>p^*$ such that
\[
u\in L^{\widetilde p}_{\mathrm{loc}}(\mathbb R^N).
\]
\end{proposition}

\begin{proof}
Fix $z\in\mathbb R^N$ and apply Proposition \ref{prop:lb-morrey} with $K=\overline{B_2(z)}$. Let $B=B_{\rho_0}(z)$ and put $g=|\nabla u|\chi_B$. We verify the global Morrey bound needed below. If $B_r(y)\cap B=\emptyset$, there is nothing to prove. If $B_r(y)\cap B\ne\emptyset$ and $r<\rho_0$, then $|y-z|<r+\rho_0<2$, hence $y\in K$, and Proposition \ref{prop:lb-morrey} gives
\[
\int_{B_r(y)}g^p\le Cr^\alpha.
\]
If $r\ge\rho_0$, then
\[
\int_{B_r(y)}g^p\le\|\nabla u\|_{L^p(\mathbb R^N)}^p
\le\|\nabla u\|_{L^p}^p\rho_0^{-\alpha}r^\alpha.
\]
Thus
\[
\sup_{y\in\mathbb R^N,\,r>0}r^{-\alpha}\int_{B_r(y)}g^p<\infty,
\]
so $g$ belongs to the Morrey class $M^{p,\alpha}$. The pointwise potential estimate \cite[Lemma 7.16]{GilbargTrudinger2001} gives, for a.e. $x\in B$,
\[
|u(x)-\bar u_B|
\le C\int_B\frac{|\nabla u(y)|}{|x-y|^{N-1}}\,dy
\le C I_1g(x).
\]
Adams' theorem \cite{Adams1975}, in the Morrey normalization
\[
\|g\|_{M^{p,\alpha}}^p
=\sup_{y,r}r^{-\alpha}\int_{B_r(y)}|g|^p,
\]
states that, for $0<\alpha<N-p$,
\[
I_1:M^{p,\alpha}\longrightarrow M^{\widetilde p,\alpha},
\qquad
\frac1{\widetilde p}=\frac1p-\frac1{N-\alpha}.
\]
Since $\alpha>0$,
\[
\frac1{\widetilde p}<\frac1p-\frac1N=\frac1{p^*},
\]
and hence $\widetilde p>p^*$. Thus $u\in L^{\widetilde p}(B)$, and a finite covering proves the local assertion.
\end{proof}

Set
\begin{equation}\label{eq:lb-rH}
r_H:=\frac{qN}{N-\sigma},
\qquad
\mu:=\min\left\{\frac{2q}{p},\frac{p^*}{p}\right\}>1.
\end{equation}
Notice that $r_H>p^*$.

\begin{lemma}[One Moser step]\label{lem:lb-moser}
Assume $u\in L^R_{\mathrm{loc}}(\mathbb R^N)$ for some $R>p^*$ with $R\ne r_H$. Then $u\in L^{R'}_{\mathrm{loc}}(\mathbb R^N)$, where
\[
R'-p^*\ge\mu(R-p^*)\quad\text{if }R<r_H,
\qquad
R'=\frac{p^*}{p}R\quad\text{if }R>r_H.
\]
\end{lemma}

\begin{proof}
Fix a ball $B_\rho=B_\rho(x_1)$ and split $W=W_1+W_2$, where $W_1$ is the convolution restricted to $B_{2\rho}(x_1)$. Lemma \ref{lem:lb-far} gives $W_2\in L^\infty(B_\rho)$. If $R<r_H$, HLS gives
\[
W_1\in L^t(B_\rho),
\qquad
\frac1t=\frac qR-\frac{N-\sigma}{N}>0.
\]
If $R>r_H$, H\"older's inequality gives $W_1\in L^\infty(B_\rho)$; in this case set $1/t=0$.

Let $s>1$, $k>0$, $u_k=\min\{u,k\}$, and $\beta=(s-1)p$. To avoid any chain-rule issue at the zero set when $0<\beta<1$, define
\[
F_{k,\beta}(t):=t\min\{t,k\}^{\beta},\qquad t\ge0.
\]
This function is globally Lipschitz on $[0,\infty)$ and satisfies
\[
F_{k,\beta}'(t)
=\min\{t,k\}^{\beta}
+\beta t\min\{t,k\}^{\beta-1}\chi_{\{t<k\}}
\quad\text{for a.e. }t>0.
\]
Choose $\eta\in C_c^\infty(B_\rho)$ with $0\le\eta\le1$ and $\eta=1$ on $B_{\rho/2}$. The test function $\varphi=\eta^pF_{k,\beta}(u)$ is admissible by Lemma \ref{lem:lb-admissible}. Testing and using Young's inequality gives
\begin{equation}\label{eq:lb-moser1}
\int\eta^pu_k^\beta|\nabla u|^p
+\beta\int\eta^pu_k^\beta|\nabla u_k|^p
\le C\int|\nabla\eta|^pu^pu_k^\beta
+2\int\eta^pWu^qu_k^\beta.
\end{equation}
Define likewise $G_{k,s}(t)=t\min\{t,k\}^{s-1}$ and $w_k=G_{k,s}(u)$. Then $G_{k,s}$ is globally Lipschitz and
\[
|\nabla w_k|
\le u_k^{s-1}|\nabla u|+(s-1)u_k^{s-1}|\nabla u_k|.
\]
Using \eqref{eq:lb-moser1} and the Sobolev inequality for $\eta w_k$ yields
\begin{equation}\label{eq:lb-moser2}
\left(\int(\eta w_k)^{p^*}\right)^{p/p^*}
\le Cs^{p-1}\left(
\int_{B_\rho}|\nabla\eta|^pu^{sp}
+\int_{B_\rho}\eta^pWu^{q+(s-1)p}
\right).
\end{equation}
The right-hand side is finite, uniformly in $k$, provided
\begin{equation}\label{eq:lb-moser-cond}
sp\le R,
\qquad
\frac1t+\frac{q+(s-1)p}{R}\le1.
\end{equation}
Then monotone convergence gives $u\in L^{sp^*}(B_{\rho/2})$.

If $R<r_H$, the second condition in \eqref{eq:lb-moser-cond} becomes
\[
(s-1)p\le\frac{2N-\sigma}{N}(R-p^*).
\]
Choose
\[
s=1+\min\left\{
\frac{2N-\sigma}{Np}(R-p^*),\frac Rp-1
\right\}.
\]
Then
\[
sp^*=\min\left\{
p^*+\frac{2q}{p}(R-p^*),\frac{p^*}{p}R
\right\},
\]
so $sp^*-p^*\ge\mu(R-p^*)$. If $R>r_H$, take $s=R/p$. Since $q<p$ by $\sigma>2p$, the second condition in \eqref{eq:lb-moser-cond} holds, and $sp^*=(p^*/p)R$. This proves the lemma.
\end{proof}

\begin{proof}[Proof of Theorem \ref{thm:local-boundedness}]
By Proposition \ref{prop:lb-first-gain}, $u\in L^{\widetilde p}_{\mathrm{loc}}$ for some $\widetilde p>p^*$. Choose $R_0\in(p^*,\widetilde p]$ with $R_0\ne r_H$ and set
\[
\mu_1:=\frac{1+\mu}{2}>1.
\]
Iterate Lemma \ref{lem:lb-moser}. If at a step with $R_j<r_H$ the exponent furnished by the lemma equals $r_H$, then $u\in L^{r_H}_{\mathrm{loc}}$, hence also $u\in L^{r_H-\delta}_{\mathrm{loc}}$ for every sufficiently small $\delta>0$. Because
\[
r_H-p^*\ge\mu(R_j-p^*)>\mu_1(R_j-p^*),
\]
we may choose $\delta$ so small that
\[
r_H-\delta-p^*\ge\mu_1(R_j-p^*)
\]
and continue from $R_{j+1}=r_H-\delta$. Thus the endpoint $r_H$ causes no loss of a uniform expansion factor. While $R_j<r_H$, the quantity $R_j-p^*$ therefore grows at least geometrically with ratio $\mu_1>1$, so after finitely many steps one obtains an exponent $R_j>r_H$. Thereafter Lemma \ref{lem:lb-moser} gives
\[
R_{j+1}=\frac{p^*}{p}R_j,
\]
and hence $R_j\to\infty$. Therefore
\[
u\in L^s_{\mathrm{loc}}(\mathbb R^N)
\qquad\text{for every }s<\infty.
\]

Fix a ball $B_\rho$. Taking $R>r_H$ in the first part of the proof of Lemma \ref{lem:lb-moser} gives $W\in L^\infty(B_\rho)$. Since $q>1$, it follows that $Wu^{q-1}\in L^m(B_\rho)$ for every finite $m$. In particular choose $m>N/p$. Serrin's local boundedness theorem for quasilinear equations \cite{Serrin1964} then yields
\[
u\in L^\infty_{\mathrm{loc}}(\mathbb R^N),
\]
as claimed.
\end{proof}

\section{Automatic power gain at infinity}\label{sec:powergain}

\begin{theorem}\label{thm:powergain}
Let $N\ge3$, $p>1$ and $2p<\sigma<N$, and assume that $q:=p^*_{\sigma}>1$. Let $u$ be a nonnegative $D^{1,p}(\mathbb R^N)$-weak solution of \eqref{eq:main}. Then there are $C>0$ and $\tau>0$ such that
\begin{equation}\label{eq:automatic-power}
u(x)\le C|x|^{-\gamma_{e}-\tau}\qquad\text{for }|x|\ge1.
\end{equation}
In fact,
\begin{equation}\label{eq:critical-vanish-infinity}
|x|^{\gamma_{e}}u(x)\longrightarrow0\qquad\text{as }|x|\to\infty.
\end{equation}
\end{theorem}

By Theorem \ref{thm:local-boundedness}, the singular set $S$ of Section \ref{sec:structure} is empty for every solution considered in this section. We first show that this automatic local boundedness forces the critical scale to vanish at infinity.

\begin{lemma}\label{lem:critical-infinity}
Under the assumptions of Theorem \ref{thm:powergain}, either $u\equiv0$, or $u>0$ in $\mathbb{R}^{N}$ and
\begin{equation}\label{eq:critical-bound-infinity}
\sup_{|x|\ge1}|x|^{\gamma_{e}}u(x)<\infty,
\end{equation}
and \eqref{eq:critical-vanish-infinity} holds.
\end{lemma}

\begin{proof}
By Corollary \ref{cor:structure}, $u\in L^{\infty}(\mathbb{R}^{N})\cap C^{1,\alpha}(\mathbb{R}^{N})$ and $u(x)\to0$ as $|x|\to\infty$. Since $-\Delta_{p}u\ge0$, the weak Harnack inequality implies that either $u\equiv0$ or $u>0$ everywhere. We assume the latter and put $M:=u^{1/\gamma_{e}}$.

Suppose first that $\sup_{x\ne0}M(x)|x|=\infty$. Fix $k\in\mathbb{N}$ and choose $y_{k}\ne0$ with $M(y_{k})|y_{k}|>2k$. Apply Lemma \ref{lem:doubling} with $X=\Sigma=\mathbb{R}^{N}$, $D=\mathbb{R}^{N}\setminus\{0\}$ and $\Gamma=\{0\}$. We obtain $x_{k}\in D$ such that
\[
M(x_{k})|x_{k}|>2k,
\qquad
M(z)\le2M(x_{k})\quad\text{if }|z-x_{k}|\le\frac{k}{M(x_{k})}.
\]
Let $\lambda_{k}:=M(x_{k})^{-1}$ and
\[
v_{k}(z):=\lambda_{k}^{\gamma_{e}}u(x_{k}+\lambda_{k}z).
\]
Then $v_{k}$ solves \eqref{eq:main}, $v_{k}(0)=1$, and $0<v_{k}\le2^{\gamma_{e}}$ on $B_{k}$. Since $u$ is globally bounded, $|x_{k}|\to\infty$, while
\[
\frac{\lambda_{k}}{|x_{k}|}<\frac1{2k}.
\]
For every fixed $R>0$ and all large $k$, Lemma \ref{lem:Wloc}, the scaling identity \eqref{eq:scaleinv}, and the local regularity estimates of \cite{DiBenedetto1983,Tolksdorf1984} give a uniform $C^{1,\alpha}(\overline{B_{R}})$ bound for $v_{k}$. On the other hand,
\[
\int_{B_{R}}v_{k}^{p^{*}}dz
=\int_{B_{R\lambda_{k}}(x_{k})}u^{p^{*}}dx\longrightarrow0,
\]
because
\[
B_{R\lambda_{k}}(x_{k})\subset
\left\{|x|\ge |x_{k}|\left(1-\frac{R}{2k}\right)\right\}
\]
and the right-hand side escapes to infinity. A subsequence therefore converges locally uniformly to a continuous function which is identically zero, contradicting $v_{k}(0)=1$. This proves \eqref{eq:critical-bound-infinity}.

We next prove \eqref{eq:critical-vanish-infinity}. If it fails, there are $\delta>0$ and $y_{j}$ with $r_{j}:=|y_{j}|\to\infty$ such that $r_{j}^{\gamma_{e}}u(y_{j})\ge\delta$. Put $\lambda_{j}:=r_{j}/2$ and
\[
v_{j}(z):=\lambda_{j}^{\gamma_{e}}u(y_{j}+\lambda_{j}z).
\]
For $|z|\le3/2$ one has $|y_{j}+\lambda_{j}z|\ge r_{j}/4$, so \eqref{eq:critical-bound-infinity} gives a uniform $L^{\infty}(B_{3/2})$ bound for $v_{j}$. Lemma \ref{lem:Wloc} and the local regularity theory then give a uniform $C^{1,\alpha}(\overline{B_{1/2}})$ bound, whereas
\[
\int_{B_{1}}v_{j}^{p^{*}}dz
=\int_{B_{\lambda_{j}}(y_{j})}u^{p^{*}}dx\longrightarrow0
\]
because $B_{\lambda_{j}}(y_{j})\subset\{|x|\ge r_{j}/2\}$. Hence $v_{j}\to0$ uniformly on $B_{1/2}$ along a subsequence, while
\[
v_{j}(0)=2^{-\gamma_{e}}r_{j}^{\gamma_{e}}u(y_{j})\ge2^{-\gamma_{e}}\delta,
\]
a contradiction.
\end{proof}

\begin{proof}[Proof of Theorem \ref{thm:powergain}]
By Theorem \ref{thm:local-boundedness} and Corollary \ref{cor:structure}, $u\in L^\infty(\mathbb R^N)\cap C^{1,\alpha}(\mathbb R^N)$ and $u(x)\to0$ as $|x|\to\infty$. If $u\equiv0$ there is nothing to prove, so by Lemma \ref{lem:critical-infinity} we may assume $u>0$. For $r\ge1$ set
\[
A(r):=r^{\gamma_{e}}\max_{r\le|x|\le2r}u(x).
\]
By Lemma \ref{lem:critical-infinity}, $A$ is bounded and $A(r)\to0$ as $r\to\infty$.

Suppose, for contradiction, that \eqref{eq:automatic-power} fails for every $\tau>0$. Then
\begin{equation}\label{eq:no-power}
\sup_{r\ge1}r^{\tau}A(r)=\infty\qquad\text{for every }\tau>0.
\end{equation}
Choose
\begin{equation}\label{eq:eps-choice}
0<\varepsilon<\min\left\{\gamma_{e},\ \gamma_{h}-\gamma_{e},\ \frac{\sigma}{2q}\right\}.
\end{equation}
Write $a(t):=A(e^{t})$ for $t\ge0$ and define the weighted envelope
\[
E(t):=\sup_{s\ge0}a(s)e^{-\varepsilon|s-t|}.
\]
Since $a$ is bounded and tends to zero, $E(t)\to0$ as $t\to\infty$. By \eqref{eq:no-power} with $\tau=\varepsilon/2$, we can choose $t_{k}\to\infty$ so that
\[
e^{\varepsilon t_{k}/2}a(t_{k})\longrightarrow\infty.
\]
Choose $s_{k}\ge0$ such that
\[
a(s_{k})e^{-\varepsilon|s_{k}-t_{k}|}\ge\frac12E(t_{k}).
\]
Then $s_{k}\to\infty$. Indeed, a bounded subsequence of $(s_{k})$ would give $a(t_{k})\le C e^{-\varepsilon t_{k}}$, contradicting the preceding display. Put
\[
r_{k}:=e^{s_{k}},\qquad A_{k}:=a(s_{k})=A(r_{k}).
\]
Then $A_{k}\to0$, and for every $s\ge0$,
\begin{equation}\label{eq:envelope-ratio}
\frac{a(s)}{A_{k}}
\le2e^{\varepsilon|s-s_{k}|}.
\end{equation}
Taking $s=0$ also gives
\begin{equation}\label{eq:Ak-lower}
A_{k}\ge\frac{a(0)}{2}r_{k}^{-\varepsilon}.
\end{equation}

Define
\begin{equation}\label{eq:wk-blowdown}
w_{k}(x):=\frac{r_{k}^{\gamma_{e}}}{A_{k}}u(r_{k}x).
\end{equation}
Since $r_{k}^{\gamma_{e}}u(r_{k}\cdot)$ solves \eqref{eq:main}, homogeneity gives
\begin{equation}\label{eq:wk-equation}
-\Delta_{p}w_{k}
=A_{k}^{2q-p}\big(|x|^{-\sigma}*w_{k}^{q}\big)w_{k}^{q-1}.
\end{equation}
Moreover,
\begin{equation}\label{eq:wk-normalization}
\max_{1\le|x|\le2}w_{k}(x)=1.
\end{equation}
If $r_{k}|x|\ge1$, then \eqref{eq:envelope-ratio} yields
\begin{equation}\label{eq:wk-growth}
w_{k}(x)
\le2|x|^{-\gamma_{e}}e^{\varepsilon|\log|x||}
\le
\begin{cases}
2|x|^{-\gamma_{e}-\varepsilon},&0<|x|\le1,\\
2|x|^{-\gamma_{e}+\varepsilon},&|x|\ge1.
\end{cases}
\end{equation}

We claim that $W[w_{k}]:=|x|^{-\sigma}*w_{k}^{q}$ is uniformly bounded on every compact annulus. Fix $0<\rho<R<\infty$ and $x$ with $\rho\le|x|\le R$. The contribution of $|y|<r_{k}^{-1}$ is bounded, using Theorem \ref{thm:local-boundedness}, by
\[
C\int_{|y|<r_{k}^{-1}}w_{k}^{q}dy
=C A_{k}^{-q}r_{k}^{q\gamma_{e}-N}\int_{B_{1}}u^{q}dz
=C A_{k}^{-q}r_{k}^{-\sigma/2}
\le C r_{k}^{q\varepsilon-\sigma/2},
\]
which is uniformly bounded, indeed tends to zero, by \eqref{eq:eps-choice} and \eqref{eq:Ak-lower}. On the region $r_{k}^{-1}\le|y|\le\rho/2$, \eqref{eq:wk-growth} is integrable because
\[
q(\gamma_{e}+\varepsilon)<q\gamma_{e}+\frac{\sigma}{2}=N.
\]
On a fixed annulus containing $x$, \eqref{eq:wk-growth} gives a uniform bound for $w_{k}$ and the kernel is locally integrable because $\sigma<N$. Finally, for large $|y|$,
\[
|x-y|^{-\sigma}w_{k}^{q}(y)
\le C|y|^{-\sigma-q(\gamma_{e}-\varepsilon)},
\]
and this is integrable at infinity because
\[
\sigma+q(\gamma_{e}-\varepsilon)>N.
\]
The claim follows.

Since $2q-p>0$ by Lemma \ref{lem:exponents}(a), the right-hand side of \eqref{eq:wk-equation} tends to zero uniformly on compact annuli. The local regularity theory therefore yields, after passing to a subsequence,
\[
w_{k}\longrightarrow w\qquad\text{in }C^{1}_{\mathrm{loc}}(\mathbb{R}^{N}\setminus\{0\}),
\]
where $w\ge0$ is $p$-harmonic in $\mathbb{R}^{N}\setminus\{0\}$. Choose $z_{k}$ with $1\le|z_{k}|\le2$ and $w_{k}(z_{k})=1$. Passing to a further subsequence, $z_{k}\to z_{\infty}$, and local uniform convergence gives $w(z_{\infty})=1$. Thus $w\not\equiv0$. Passing to the limit in \eqref{eq:wk-growth},
\begin{equation}\label{eq:w-limit-growth}
w(x)\le
\begin{cases}
2|x|^{-\gamma_{e}-\varepsilon},&0<|x|\le1,\\
2|x|^{-\gamma_{e}+\varepsilon},&|x|\ge1.
\end{cases}
\end{equation}
Because $\gamma_{e}+\varepsilon<\gamma_{h}$, Serrin's classification of isolated singularities for nonnegative $p$-harmonic functions \cite{Serrin1965} implies that the origin is removable: a nonremovable singularity is comparable to the fundamental $p$-harmonic profile $|x|^{-\gamma_{h}}$. Thus $w$ extends to a nonnegative entire $p$-harmonic function. Since $\varepsilon<\gamma_{e}$, \eqref{eq:w-limit-growth} gives $w(x)\to0$ as $|x|\to\infty$. The maximum principle on expanding balls then yields $w\equiv0$, contradicting \eqref{eq:wk-normalization}. Hence \eqref{eq:automatic-power} holds for some $\tau>0$.
\end{proof}

\section{The decay bootstrap}\label{sec:bootstrap}

\begin{theorem}\label{thm:bootstrap}
Let $N\ge3$, $p>1$ and $2p<\sigma<N$, and assume that $q:=p^*_{\sigma}>1$ (this is automatic when $\Lambda>0$ by Lemma \ref{lem:qbigger1}). Let $u$ be a positive $D^{1,p}(\mathbb{R}^{N})$-weak solution of \eqref{eq:main}. Then for every $\nu<\min\{\gamma_{h},\gamma_{*}\}$ there is $C_{\nu}>0$ with
\[
u(x)\le C_{\nu}|x|^{-\nu}
\qquad\text{for }|x|\ge1 .
\]
\end{theorem}

\subsection{Boundedness of the potential and positivity}

Throughout this section we assume
\begin{equation}\label{eq:standing}
N\ge3,\qquad 1<p<\frac{N}{2},\qquad 2p<\sigma<N,\qquad q>1,
\end{equation}
and $u$ is a nonnegative $D^{1,p}(\mathbb{R}^{N})$-weak solution of \eqref{eq:main}. By Lemma \ref{lem:qbigger1} the last condition in \eqref{eq:standing} is implied by the others together with $\Lambda>0$, so it is not an extra assumption in Theorem \ref{thm:main}. By Theorem \ref{thm:local-boundedness}, Corollary \ref{cor:structure}, and Theorem \ref{thm:powergain}, we have $u\in L^{\infty}(\mathbb{R}^{N})\cap C^{1,\alpha}(\mathbb{R}^{N})$ and, for some $\tau>0$, $u(x)\le C|x|^{-\gamma_{e}-\tau}$ at infinity.

\begin{lemma}\label{lem:Wbounded}
If in addition $u\in L^{\infty}(\mathbb{R}^{N})$, then $W\in L^{\infty}(\mathbb{R}^{N})$ with
\[
\|W\|_{L^{\infty}(\mathbb{R}^{N})}\le C\big(\|u\|^{q}_{L^{\infty}(\mathbb{R}^{N})}+\|u\|^{q}_{L^{p^{*}}(\mathbb{R}^{N})}\big).
\]
\end{lemma}

\begin{proof}
Apply Lemma \ref{lem:Wloc} with $r=1$ and an arbitrary center.
\end{proof}

\begin{lemma}\label{lem:regularity}
We have $u\in C^{1,\alpha}_{\mathrm{loc}}(\mathbb{R}^{N})$ for some $0<\alpha<\min\{1,\frac{1}{p-1}\}$. Moreover either $u\equiv0$ or $u>0$ in $\mathbb{R}^{N}$, and in the latter case $\inf_{K}u>0$ for every compact $K\subset\mathbb{R}^{N}$.
\end{lemma}

\begin{proof}
By Lemma \ref{lem:Wbounded} and $q>1$, the right hand side of \eqref{eq:main} satisfies $Wu^{q-1}\in L^{\infty}(\mathbb{R}^{N})$. Hence $u$ is a bounded weak solution of $-\Delta_{p}u=f$ with $f\in L^{\infty}_{\mathrm{loc}}$, and the regularity results of \cite{DiBenedetto1983,Tolksdorf1984} give $u\in C^{1,\alpha}_{\mathrm{loc}}$.

Since $Wu^{q-1}\ge0$, the function $u$ is $p$-superharmonic. By the weak Harnack inequality \cite{Trudinger1967}, for every ball $B_{2r}(x_{0})$ and every $0<t<N(p-1)/(N-p)$,
\[
\inf_{B_{r}(x_{0})}u\ge c\Big(\frac{1}{|B_{2r}|}\int_{B_{2r}(x_{0})}u^{t}\,dy\Big)^{1/t},
\]
so the set where $u$ vanishes on a ball is open and closed, whence $u\equiv0$ or $u>0$. In the latter case the same inequality, together with continuity and compactness, gives $\inf_{K}u>0$.
\end{proof}

From now on we assume $u>0$. We record the consequence of \eqref{eq:SZ} that we shall use:
\begin{equation}\label{eq:lower}
u(x)\ge c_{1}|x|^{-\gamma_{h}}\qquad\text{for }|x|\ge R_{1},
\end{equation}
for some $c_{1}>0$ and $R_{1}\ge1$.

\subsection{The bootstrap map}

We keep the standing assumptions above and assume $u>0$. Define, for $\delta>\gamma_{e}$,
\begin{equation}\label{eq:adelta}
a(\delta):=\min\{\sigma,\ q\delta+\sigma-N\},
\qquad
T(\delta):=\frac{a(\delta)+(q-1)\delta-p}{p-1}.
\end{equation}

\begin{lemma}\label{lem:Tprops}
Let $N\ge3$, $p>1$ and $2p<\sigma<N$. Then $a$ and $T$ are nondecreasing and continuous on $(\gamma_{e},\infty)$, and
\begin{enumerate}[label=\textup{(\alph*)}]
\item $T(\gamma_{e})=\gamma_{e}$;
\item on the set where $q\delta+\sigma-N<\sigma$, the map $T$ is affine with slope $(2q-1)/(p-1)>1$;
\item on the set where $q\delta+\sigma-N\ge\sigma$, the map $T$ is affine with slope $(q-1)/(p-1)<1$ and unique fixed point $\gamma_{*}$;
\item $q\gamma_{*}>N$, so $\gamma_{*}$ lies in the saturated branch, and $\gamma_{e}$ and $\gamma_{*}$ are the only fixed points of $T$ in $(0,\infty)$;
\item $T$ is strictly increasing, $T(\delta)>\gamma_{e}$ for every $\delta>\gamma_{e}$, and $T(\delta)>\delta$ for every $\delta\in(\gamma_{e},\gamma_{*})$.
\end{enumerate}
\end{lemma}

\begin{proof}
The map $a$ is nondecreasing and continuous, and $\delta\mapsto(q-1)\delta$ is strictly increasing because $q>1$, so $T$ is strictly increasing and continuous. For (a), $q\gamma_{e}=N-\sigma/2<N$, so $a(\gamma_{e})=q\gamma_{e}+\sigma-N=\sigma/2$, and
\[
T(\gamma_{e})=\frac{\sigma/2+(q-1)\gamma_{e}-p}{p-1}
=\frac{\sigma/2+N-\sigma/2-\gamma_{e}-p}{p-1}
=\frac{(p-1)(N-p)}{p(p-1)}=\gamma_{e},
\]
where we used $q\gamma_{e}=N-\sigma/2$. Claims (b) and (c) are immediate from \eqref{eq:adelta}, together with $2q>p$ and $q<p$ from Lemma \ref{lem:exponents}(a), and the computation of the fixed points of the two affine branches: in case (b) it is $(N+p-\sigma)/(2q-p)=\gamma_{e}$, and in case (c) it is $(\sigma-p)/(p-q)=\gamma_{*}$. For (d), a computation gives
\[
q\gamma_{*}-N=\frac{\sigma(N+p-\sigma)}{\sigma-2p}>0
\]
because $\sigma<N$ and $\sigma>2p$, so $\gamma_{*}$ indeed belongs to the saturated branch, while $q\gamma_{e}=N-\sigma/2<N$ places $\gamma_{e}$ in the unsaturated one. Since each branch is affine with a single fixed point, these are the only two. For (e), strict monotonicity was observed above, and $T(\delta)>T(\gamma_{e})=\gamma_{e}$ for $\delta>\gamma_{e}$. On the unsaturated branch $T$ is affine with slope larger than $1$ and fixed point $\gamma_{e}$, so $T(\delta)>\delta$ there for $\delta>\gamma_{e}$. On the saturated branch $T$ is affine with slope smaller than $1$ and fixed point $\gamma_{*}$, so $T(\delta)>\delta$ there exactly for $\delta<\gamma_{*}$.
\end{proof}

\begin{lemma}\label{lem:onestep}
Assume the assumptions of Theorem \ref{thm:bootstrap}, let $\delta>\gamma_{e}$ and suppose that $u(x)\le C_{\delta}|x|^{-\delta}$ for $|x|\ge1$. Then for every
\[
\nu<\min\{\gamma_{h},\ T(\delta)\}
\]
there is $C_{\nu}>0$ with $u(x)\le C_{\nu}|x|^{-\nu}$ for $|x|\ge1$.
\end{lemma}

\begin{proof}
Since $\delta>\gamma_{e}$ we have $q\delta>q\gamma_{e}=N-\sigma/2>N-\sigma$, so Lemma \ref{lem:conv}(ii) applies with $\beta=\delta$ and gives, for $|x|\ge2$,
\begin{equation}\label{eq:Wdecay}
W(x)\le C|x|^{-a'},
\qquad
a'=
\begin{cases}
a(\delta) &\text{if }q\delta\neq N,\\
\text{any } a'<\sigma &\text{if }q\delta=N,
\end{cases}
\end{equation}
where in the borderline case we absorbed the logarithm. Note that $W\in L^{\infty}$ by Lemma \ref{lem:Wbounded}, and that in all cases $a'$ may be taken arbitrarily close to $a(\delta)$ from below. Consequently the right hand side of \eqref{eq:main} obeys
\begin{equation}\label{eq:fdecay}
f:=Wu^{q-1}\le C|x|^{-m},
\qquad m:=a'+(q-1)\delta,
\qquad |x|\ge2 ,
\end{equation}
where we used $q>1$ and $u\le C_{\delta}|x|^{-\delta}$.

Fix $\nu<\min\{\gamma_{h},T(\delta)\}$ and, using \eqref{eq:Wdecay}, choose $a'$ so close to $a(\delta)$ that $\nu(p-1)+p\le m$. Set $\rho:=2$ and, on $\Omega:=\mathbb{R}^{N}\setminus\overline{B_{\rho}}$, let $v(x):=M|x|^{-\nu}$. By \eqref{eq:barrier}, $h(\nu)>0$, and for $|x|\ge\rho$,
\[
-\Delta_{p}v=h(\nu)M^{p-1}|x|^{-\nu(p-1)-p}
\ge h(\nu)M^{p-1}\rho^{\,m-\nu(p-1)-p}|x|^{-m}
\ge C|x|^{-m}\ge f,
\]
provided $M$ is large enough, where we used $\nu(p-1)+p\le m$ in the first inequality. Enlarging $M$ further we also have $v\ge u$ on $\partial B_{\rho}$, because $u$ is bounded.

We check the assumptions of Lemma \ref{lem:comparison}. We may assume $\nu>\gamma_{e}$, since otherwise the asserted bound is weaker than the assumption $u\le C_{\delta}|x|^{-\delta}$ and there is nothing to prove. Both $u$ and $v$ lie in $W^{1,p}_{\mathrm{loc}}(\overline{\Omega})$, and $\nabla v\in L^{p}(\Omega)$ because
\[
\int_{\Omega}|\nabla v|^{p}\,dx\simeq\int_{\rho}^{\infty}r^{N-1-(\nu+1)p}\,dr<\infty
\qquad\Longleftrightarrow\qquad \nu>\gamma_{e}.
\]
Hence $(u-v)^{+}$ is dominated by $u\in D^{1,p}(\mathbb{R}^{N})$, vanishes on $\partial\Omega$, and has gradient $\nabla u-\nabla v\in L^{p}$ on the set where it is positive, so $(u-v)^{+}\in D^{1,p}(\Omega)$. Since $-\Delta_{p}u=f\le-\Delta_{p}v$ weakly in $\Omega$, Lemma \ref{lem:comparison} gives $u\le v$ in $\Omega$. Together with the boundedness of $u$ on $B_{\rho}$, this is the assertion.
\end{proof}

\begin{proof}[Proof of Theorem \ref{thm:bootstrap}]
By Theorem \ref{thm:powergain}, there are $C>0$ and $\tau>0$ such that $u(x)\le C|x|^{-\gamma_{e}-\tau}$ for $|x|\ge1$. We may therefore start from $\delta_{0}:=\gamma_{e}+\tau>\gamma_{e}$. Fix
\[
\nu<\min\{\gamma_{h},\gamma_{*}\}.
\]
If $\nu\le\delta_{0}$, the conclusion follows immediately from Theorem \ref{thm:powergain}. We therefore assume $\delta_{0}<\nu$ and set
\[
F(\delta):=\min\{T(\delta),\gamma_{h}\}.
\]
As long as $\delta_{k}<\nu$, define
\begin{equation}\label{eq:strict-bootstrap}
\delta_{k+1}:=\frac{\delta_{k}+F(\delta_{k})}{2}.
\end{equation}
Since $\delta_{k}<\nu<\min\{\gamma_{h},\gamma_{*}\}$, Lemma \ref{lem:Tprops}(e) gives $T(\delta_{k})>\delta_{k}$, while also $\gamma_{h}>\delta_{k}$. Hence
\[
\delta_{k}<\delta_{k+1}<F(\delta_{k})=\min\{T(\delta_{k}),\gamma_{h}\}.
\]
Thus Lemma \ref{lem:onestep} applies at every step and, by induction, there is $C_{k}>0$ such that
\[
u(x)\le C_{k}|x|^{-\delta_{k}}\qquad\text{for }|x|\ge1 .
\]

It remains to show that $\delta_{k}\ge\nu$ for some finite $k$. Suppose otherwise. Then $(\delta_{k})$ is strictly increasing and bounded above by $\nu$, hence $\delta_{k}\to\delta_{\infty}$ for some
\[
\delta_{0}<\delta_{\infty}\le\nu<\min\{\gamma_{h},\gamma_{*}\}.
\]
Passing to the limit in \eqref{eq:strict-bootstrap} and using the continuity of $F$ gives $F(\delta_{\infty})=\delta_{\infty}$. Since $\delta_{\infty}<\gamma_{h}$, this implies $T(\delta_{\infty})=\delta_{\infty}$. By Lemma \ref{lem:Tprops}(d), the only positive fixed points of $T$ are $\gamma_{e}$ and $\gamma_{*}$, contradicting
$\gamma_{e}<\delta_{0}<\delta_{\infty}<\gamma_{*}$. Therefore some $\delta_{k}\ge\nu$, and the corresponding decay estimate proves the theorem.
\end{proof}

\section{The classical phase \texorpdfstring{$\Lambda>0$}{Lambda > 0}}\label{sec:classical}

\subsection{Sharp continuation of the Dai--Li--Liu profile}\label{sec:sharp}

\begin{theorem}\label{thm:main}
Let $N\ge3$, $p>1$ and $2p<\sigma<N$, and assume that $q:=p^*_{\sigma}>1$ and $\Lambda>0$; the inequality $q>1$ follows automatically from Lemma \ref{lem:qbigger1}, and $\Lambda>0$ is automatic if $p\ge2$. Let $u$ be a nonnegative $D^{1,p}(\mathbb{R}^{N})$-weak solution of \eqref{eq:main}. Then $u\in C^{1,\alpha}(\mathbb{R}^{N})\cap L^{\infty}(\mathbb{R}^{N})$ for some $0<\alpha<\min\{1,\frac{1}{p-1}\}$, and either $u\equiv0$, or $u>0$ in $\mathbb{R}^{N}$ and the following hold.
\begin{enumerate}[label=\textup{(\roman*)}]
\item The sharp asymptotic estimates \eqref{eq:sharp-u} and \eqref{eq:sharp-grad} hold.
\item The function $u$ is radially symmetric about some point $x_{0}\in\mathbb{R}^{N}$ and strictly decreasing in the radial variable; equivalently, $u(x)=U(|x-x_{0}|)$ for a positive profile $U$ with $U'(r)<0$ for every $r>0$.
\end{enumerate}
\end{theorem}

We now assume in addition that $\Lambda>0$, and we prove part (i) of Theorem \ref{thm:main}. Recall the potential $V=Wu^{q-p}$ of \eqref{eq:VForm}.

\begin{proposition}\label{prop:Vdecay}
Assume the assumptions of Theorem \ref{thm:main} and $u>0$. Then there are $\beta>p$, $C_{V}>0$ and $R_{2}\ge1$ such that
\begin{equation}\label{eq:Vdecay}
0\le V(x)\le C_{V}|x|^{-\beta}\qquad\text{for }|x|\ge R_{2},
\end{equation}
and $V\in L^{\infty}_{\mathrm{loc}}(\mathbb{R}^{N})\cap L^{N/p}(\mathbb{R}^{N})$.
\end{proposition}

\begin{proof}
By Lemma \ref{lem:exponents}(c) we have $\gamma_{*}>\gamma_{h}$, so Theorem \ref{thm:bootstrap} provides, for every $\nu<\gamma_{h}$, a constant with $u\le C_{\nu}|x|^{-\nu}$ for $|x|\ge1$. Fix such a $\nu$, to be chosen. As in \eqref{eq:Wdecay}, Lemma \ref{lem:conv}(ii) gives $W\le C|x|^{-a'}$ for $|x|\ge2$ with $a'$ arbitrarily close to $a(\nu)=\min\{\sigma,q\nu+\sigma-N\}$ from below. By \eqref{eq:lower} and $q-p<0$,
\[
u^{q-p}(x)\le c_{1}^{q-p}|x|^{(p-q)\gamma_{h}}\qquad\text{for }|x|\ge R_{1},
\]
so \eqref{eq:Vdecay} holds with $\beta=a'-(p-q)\gamma_{h}$. Letting $\nu\uparrow\gamma_{h}$ and $a'\uparrow a(\nu)$ we get
\[
\beta\ \longrightarrow\ \min\{\sigma,\ q\gamma_{h}+\sigma-N\}-(p-q)\gamma_{h}
=\begin{cases}
p+\Lambda &\text{if }q\gamma_{h}>N,\\[2pt]
p+\dfrac{N+p-\sigma}{p-1} &\text{if }q\gamma_{h}\le N,
\end{cases}
\]
by Lemma \ref{lem:exponents}(d). In the first case the limit exceeds $p$ because $\Lambda>0$, and in the second because $\sigma<N$. Hence $\nu$ and $a'$ may be fixed so that $\beta>p$.

For the integrability, $V\le C_{K}W$ on every compact $K$ by Lemma \ref{lem:regularity}, and $W\in L^{\infty}(\mathbb{R}^{N})$ by Lemma \ref{lem:Wbounded}, so $V\in L^{\infty}_{\mathrm{loc}}$. Consequently $\int_{B_{R_{2}}}V^{N/p}<\infty$, while by \eqref{eq:Vdecay},
\[
\int_{|x|\ge R_{2}}V^{N/p}\,dx\le C\int_{R_{2}}^{\infty}r^{N-1-\beta N/p}\,dr<\infty
\]
because $\beta>p$.
\end{proof}

\begin{proof}[Proof of Theorem \ref{thm:main}(i)]
By Lemma \ref{lem:regularity}, $u\in C^{1,\alpha}_{\mathrm{loc}}\cap L^{\infty}$ and either $u\equiv0$ or $u>0$. Assume $u>0$. By Proposition \ref{prop:Vdecay}, $u$ is a positive $D^{1,p}(\mathbb{R}^{N})$-weak solution of $-\Delta_{p}u=V(x)u^{p-1}$ with $0\le V\in L^{N/p}(\mathbb{R}^{N})$, so \cite[Theorem 1.3]{DaiLiLiu2025} applies once its assumptions are checked.

The requirement $u\in L^{r}_{\mathrm{loc}}$ for all $r<\infty$ holds because $u\in L^{\infty}$. The requirement $V\in L^{t}_{\mathrm{loc}}$ for all $N/p<t<\infty$ holds because $V\in L^{\infty}_{\mathrm{loc}}$, and it gives $u\in C^{1,\alpha}(\mathbb{R}^{N})$.

We verify the decay property \cite[(1.9)]{DaiLiLiu2025} with constants uniform in the exponent $\bar p$, as required there. By Theorem \ref{thm:bootstrap} there are $\nu>\gamma_e$, $C_u>0$ and $R_u\ge1$ such that $u(x)\le C_u|x|^{-\nu}$ for $|x|\ge R_u$. Set
\[
\tau_0:=\frac{\nu-\gamma_e}{2}>0
\]
and choose $\widehat p$ so large that $\nu\widehat p>N$ and
$N/\widehat p\le(\nu-\gamma_e)/2$. Then, for every $\bar p\ge\widehat p$ and $R\ge R_u$,
\begin{align*}
\|u\|_{L^{\bar p}(\mathbb R^N\setminus B_{2R})}
&\le C_u\left(\frac{\omega_{N-1}}{\nu\bar p-N}\right)^{1/\bar p}(2R)^{-\nu+N/\bar p}\\
&\le C R^{-\gamma_e-\tau_0}
 =\frac{C}{R^{\frac{N-p}{p}+\tau_0}},
\end{align*}
where $C$ is independent of both $\bar p\ge\widehat p$ and $R\ge R_u$, because
$\sup_{\bar p\ge\widehat p}(\omega_{N-1}/(\nu\bar p-N))^{1/\bar p}<\infty$.
Thus \cite[(1.9)]{DaiLiLiu2025} holds in its stated uniform form. The uniform bound \cite[(1.10)]{DaiLiLiu2025} holds because \eqref{eq:Vdecay} gives $\|V\|_{L^{t}(B_{1}(x))}\le C|x|^{-\beta}\le C$ for $|x|\ge2R_{2}$. Finally the decay assumption on $V$ required for the two-sided estimates is exactly \eqref{eq:Vdecay} with $\beta>p$. Therefore \cite[Theorem 1.3]{DaiLiLiu2025} yields \eqref{eq:sharp-u} and \eqref{eq:sharp-grad}.
\end{proof}

\subsection{Radial symmetry}\label{sec:mp}

The useful decomposition in the complementary regime is
\begin{equation}\label{eq:splitting}
W[u]u^{q-1}-W[u_\lambda]u_\lambda^{q-1}
=W[u]\big(u^{q-1}-u_\lambda^{q-1}\big)
+\big(W[u]-W[u_\lambda]\big)u_\lambda^{q-1}.
\end{equation}
Unlike the splitting through $u^{q-p}$, this has the correct monotonicity whenever $q>1$.

We prove part (ii) of Theorem \ref{thm:main}. We follow the scheme of \cite[Section 4]{DaiLiLiu2025}, which in turn refines \cite{DamascelliMerchanMontoroSciunzi2014,Sciunzi2016,Vetois2016}, with two changes. The nonlinearity is split as in \eqref{eq:splitting}, and at the limiting position of the plane the strong comparison principle is replaced by Lemma \ref{lem:strict-separation} below, which exploits the nonlocal term. We give the continuation step in full, since the argument is not a verbatim transcription.

For a direction $\nu\in S^{N-1}$ and $\lambda\in\mathbb{R}$ set
\[
\Sigma^{\nu}_{\lambda}:=\{x\in\mathbb{R}^{N}:\langle x,\nu\rangle<\lambda\},
\qquad
x^{\nu}_{\lambda}:=x+2(\lambda-\langle x,\nu\rangle)\nu,
\qquad
u^{\nu}_{\lambda}(x):=u(x^{\nu}_{\lambda}).
\]
We drop the superscript $\nu$ when no confusion arises and write $w:=(u-u_{\lambda})^{+}$.

\begin{lemma}\label{lem:nonlocal-diff}
Let $u$ be as in Theorem \ref{thm:main} with $u>0$, and let $x\in\Sigma_{\lambda}$. Then
\begin{equation}\label{eq:Wdiff}
W(x)-W(x_{\lambda})\le P_{\lambda}(x):=q\int_{\Sigma_{\lambda}}\frac{u^{q-1}(y)w(y)}{|x-y|^{\sigma}}\,dy .
\end{equation}
Moreover there are $C>0$ and $R_{3}\ge1$, independent of $\lambda\le0$, such that, on $\Sigma_{\lambda}\cap\{|x|\ge R_{3}\}$,
\begin{equation}\label{eq:Kbound}
W(x)\big(u^{q-1}(x)-u^{q-1}_{\lambda}(x)\big)\le K(x)\,w(x),
\qquad
K(x)\le\frac{C}{|x|^{\beta+(p-2)\gamma_{h}}},
\end{equation}
with $\beta>p$ as in Proposition \ref{prop:Vdecay}. More generally, for every $L>0$ the same estimate holds uniformly for $|\lambda|\le L$, after replacing $C,R_{3}$ by constants $C_{L},R_{3,L}$.
\end{lemma}

\begin{proof}
Splitting $\mathbb{R}^{N}=\Sigma_{\lambda}\cup\Sigma_{\lambda}^{c}$ and changing variables $y\mapsto y_{\lambda}$ in the second part gives the classical identity
\[
W(x)-W(x_{\lambda})=\int_{\Sigma_{\lambda}}\Big(\frac{1}{|x-y|^{\sigma}}-\frac{1}{|x_{\lambda}-y|^{\sigma}}\Big)\big(u^{q}(y)-u^{q}_{\lambda}(y)\big)\,dy .
\]
For $x,y\in\Sigma_{\lambda}$ one has $|x-y|<|x_{\lambda}-y|$, so the kernel factor is positive and bounded by $|x-y|^{-\sigma}$. Dropping the part of the integral where $u\le u_{\lambda}$, which is nonpositive, and using $u^{q}-u^{q}_{\lambda}\le qu^{q-1}(u-u_{\lambda})$ on the remaining set, we obtain \eqref{eq:Wdiff}.

For \eqref{eq:Kbound}, the mean value theorem gives $u^{q-1}-u^{q-1}_{\lambda}\le(q-1)\zeta^{q-2}w$ with $\zeta$ between $u_{\lambda}(x)$ and $u(x)$, so $K=(q-1)W\zeta^{q-2}$. If $q\ge2$ then $\zeta^{q-2}\le u^{q-2}$, hence $K\le(q-1)Wu^{q-p}u^{p-2}=(q-1)Vu^{p-2}$, and \eqref{eq:sharp-u} together with \eqref{eq:Vdecay} gives $K\le C|x|^{-\beta-(p-2)\gamma_{h}}$. If $q<2$ then $\zeta^{q-2}\le u_{\lambda}^{q-2}(x)$. For $\lambda\le0$ and $x\in\Sigma_{\lambda}$ one has $|x_{\lambda}|\le|x|$, so the lower bound in \eqref{eq:sharp-u} gives
\[
u_{\lambda}(x)=u(x_{\lambda})\ge c_{0}(1+|x|^{\gamma_{h}})^{-1}.
\]
For the bounded-parameter version, if $|\lambda|\le L$ then $|x_{\lambda}|\le|x|+2L\le2|x|$ whenever $|x|\ge2L$, and the same lower estimate holds with a modified constant. Thus, uniformly either for all $\lambda\le0$ or for $|\lambda|\le L$ after increasing the radius,
$\zeta^{q-2}\le C|x|^{(2-q)\gamma_{h}}$. By \eqref{eq:Wdecay} with $a'$ as in Proposition \ref{prop:Vdecay},
\[
K\le C|x|^{-a'+(2-q)\gamma_{h}}=C|x|^{-\beta-(p-2)\gamma_{h}},
\]
because $\beta=a'-(p-q)\gamma_{h}$.
\end{proof}

\begin{lemma}\label{lem:I2}
With the notation above, for every $\lambda\in\mathbb R$,
\begin{equation}\label{eq:I2}
\int_{\Sigma_{\lambda}}P_{\lambda}(x)u^{q-1}_{\lambda}(x)w(x)\,dx
\le C\|u\|^{2q-p^{*}}_{L^{p^{*}}(\Sigma_{\lambda})}\int_{\Sigma_{\lambda}}u^{p^{*}-2}w^{2}\,dx .
\end{equation}
\end{lemma}

\begin{proof}
Since $u_{\lambda}\le u$ on the support of $w$, the left hand side is bounded by
\[
q\int_{\Sigma_{\lambda}}\int_{\Sigma_{\lambda}}\frac{u^{q-1}(x)w(x)\,u^{q-1}(y)w(y)}{|x-y|^{\sigma}}\,dx\,dy
\le C\big\|u^{q-1}w\big\|^{2}_{L^{t}(\Sigma_{\lambda})},
\qquad t=\frac{2N}{2N-\sigma},
\]
by the Hardy--Littlewood--Sobolev inequality. Write $u^{q-1}w=u^{c}\cdot u^{(p^{*}-2)/2}w$ with $c:=q-p^{*}/2$, which is positive by Lemma \ref{lem:exponents}(a). H\"older's inequality with the exponents $A:=2N/(N-\sigma)$ and $2$, which satisfy $1/A+1/2=1/t$, gives
\[
\big\|u^{q-1}w\big\|_{L^{t}}\le\|u\|^{c}_{L^{cA}}\Big(\int_{\Sigma_{\lambda}}u^{p^{*}-2}w^{2}\Big)^{1/2},
\]
and a computation from \eqref{eq:psigma} shows $cA=p^{*}$. Since $2c=2q-p^{*}$, this is \eqref{eq:I2}. Note that the exponent $(p^{*}-2)/2$ is negative when $p<2N/(N+2)$, which does occur in the range \eqref{eq:new-range}; the H\"older step above is unaffected, and in the subsequent use of \eqref{eq:I2} the bound $u^{p^{*}-2}\le C|x|^{-(p^{*}-2)\gamma_{h}}$ is then furnished by the lower bound in \eqref{eq:sharp-u} rather than by the upper one.
\end{proof}

The limiting position of the plane is handled by the following observation, in which the nonlocal term does the work usually done by a strong comparison principle.

\begin{lemma}[Strict separation from the nonlocal term]\label{lem:strict-separation}
Let $N\ge3$, $p>1$ and $2p<\sigma<N$, and assume that $q:=p^*_{\sigma}>1$. Let $u$ be a positive $D^{1,p}(\mathbb{R}^{N})$-weak solution of \eqref{eq:main}. Let $\lambda\in\mathbb R$ be such that $u\le u_\lambda$ in $\Sigma_\lambda$ and $u\not\equiv u_\lambda$ in $\Sigma_\lambda$. Then
\begin{enumerate}[label=\textup{(\roman*)}]
\item $W[u](x)<W[u](x_\lambda)$ for every $x\in\Sigma_\lambda$;
\item $u<u_\lambda$ almost everywhere in $\Sigma_\lambda$.
\end{enumerate}
\end{lemma}

\begin{proof}
By Corollary \ref{cor:structure}, $u\in C^{1,\alpha}(\mathbb R^N)$ and $\nabla u$ is bounded on compact sets.

(i) Splitting $\mathbb R^N$ into $\Sigma_\lambda$ and its complement and reflecting the second part, as in the proof of Lemma \ref{lem:nonlocal-diff},
\[
W[u](x)-W[u](x_\lambda)=\int_{\Sigma_\lambda}k_\lambda(x,y)\big(u^{q}(y)-u_\lambda^{q}(y)\big)\,dy,
\qquad
k_\lambda(x,y):=\frac{1}{|x-y|^{\sigma}}-\frac{1}{|x_\lambda-y|^{\sigma}},
\]
where all integrals are finite because $W[u]$ is finite everywhere. For $x,y\in\Sigma_\lambda$ one has $|x-y|<|x_\lambda-y|$, hence $k_\lambda(x,y)>0$. The factor $u^{q}-u_\lambda^{q}$ is nonpositive in $\Sigma_\lambda$ and, by continuity and $u\not\equiv u_\lambda$, strictly negative on a nonempty open subset of $\Sigma_\lambda$. Hence the integral is negative.

(ii) Put $v:=u_\lambda-u\ge0$ in $\Sigma_\lambda$ and $Z_0:=\{x\in\Sigma_\lambda:v(x)=0\}$. By the reflection invariance of $\Delta_p$, $u_\lambda$ is a weak solution of $-\Delta_pu_\lambda=W[u](x_\lambda)u_\lambda^{q-1}$ in $\Sigma_\lambda$. Set
\[
F:=|\nabla u_\lambda|^{p-2}\nabla u_\lambda-|\nabla u|^{p-2}\nabla u,
\qquad
g(x):=W[u](x_\lambda)\,u_\lambda^{q-1}(x)-W[u](x)\,u^{q-1}(x),
\]
so that $\int_{\Sigma_\lambda}\langle F,\nabla\varphi\rangle=\int_{\Sigma_\lambda}g\varphi$ for every Lipschitz $\varphi$ with compact support in $\Sigma_\lambda$. On $Z_0$ one has $u=u_\lambda>0$, so by (i)
\[
g=\big(W[u](x_\lambda)-W[u](x)\big)u^{q-1}>0\qquad\text{on }Z_0 .
\]
Fix $0\le\eta\in C^{\infty}_{c}(\Sigma_\lambda)$ and, for $\varepsilon>0$, test with $\varphi=\eta\,\psi_\varepsilon(v)$, where $\psi_\varepsilon(t):=(1-t/\varepsilon)^{+}$. Since $\psi_\varepsilon'=-1/\varepsilon$ on $\{0\le v<\varepsilon\}$ and $\langle F,\nabla v\rangle\ge0$ by the monotonicity of $\xi\mapsto|\xi|^{p-2}\xi$,
\[
\int_{\Sigma_\lambda}g\,\eta\,\psi_\varepsilon(v)
=\int_{\Sigma_\lambda}\psi_\varepsilon(v)\langle F,\nabla\eta\rangle-\frac{1}{\varepsilon}\int_{\{v<\varepsilon\}}\eta\,\langle F,\nabla v\rangle
\le\int_{\Sigma_\lambda}\psi_\varepsilon(v)\langle F,\nabla\eta\rangle .
\]
As $\varepsilon\to0$, $\psi_\varepsilon(v)\to\chi_{Z_0}$ pointwise, and $g$ and $F$ are bounded on the support of $\eta$, so dominated convergence gives
\[
\int_{Z_0}g\,\eta\le\int_{Z_0}\langle F,\nabla\eta\rangle=0 .
\]
The last equality holds because $\nabla v=0$ almost everywhere on the level set $Z_0$ by Stampacchia's theorem, so that $\nabla u=\nabla u_\lambda$ and $F=0$ almost everywhere on $Z_0$. Since $g>0$ on $Z_0$ and $\eta\ge0$ is arbitrary, $|Z_0|=0$.
\end{proof}

We record the weight comparison needed at infinity. Put
\[
s:=(2-p)(\gamma_h+1)=-\frac{N-1}{p-1}(p-2),
\qquad
A_\lambda:=\{x\in\Sigma_\lambda:u(x)>u_\lambda(x)\}.
\]
We use the weight
\[
\rho_\lambda(x):=
\begin{cases}
\big(|\nabla u(x)|+|\nabla u_\lambda(x)|\big)^{p-2},
& |\nabla u(x)|+|\nabla u_\lambda(x)|>0,\\
0,& |\nabla u(x)|+|\nabla u_\lambda(x)|=0,\ p\ne2,\\
1,& p=2.
\end{cases}
\]
For $1<p<2$ the value assigned on the common critical set is only a convention for the weighted energy: there $\nabla w=0$, so we set $\rho_\lambda|\nabla w|^2=0$. Every pointwise estimate involving the singular power is understood off the common critical set and then extended to the corresponding energy inequality by this convention.

\begin{lemma}[Tail weight in the classical phase]\label{lem:classical-weight}
Assume the assumptions of Theorem \ref{thm:main} and $u>0$. There are $\kappa\in(0,1]$, $c_w>0$ and $R_w\ge R_0$ such that, for every $\lambda\in\mathbb R$,
\begin{equation}\label{eq:classical-weight}
|x_\lambda|\ge\kappa|x|
\quad\text{and}\quad
\rho_\lambda(x)\ge c_w|x|^{s}
\qquad\text{for }x\in A_\lambda,\ |x|\ge R_w .
\end{equation}
\end{lemma}

\begin{proof}
Let $x\in A_\lambda$ with $|x|$ large. Then $u(x)>u(x_\lambda)$, and \eqref{eq:sharp-u} gives $c_0(1+|x_\lambda|^{\gamma_h})^{-1}<C_0|x|^{-\gamma_h}$. Once $(c_0/(2C_0))|x|^{\gamma_h}\ge1$, this forces $|x_\lambda|^{\gamma_h}\ge(c_0/(2C_0))|x|^{\gamma_h}$, which is the first assertion with $\kappa=(c_0/(2C_0))^{1/\gamma_h}$. Note that no restriction on $\lambda$ was used. If $p\ge2$ then $s\le0$ and, by the lower bound in \eqref{eq:sharp-grad}, $\rho_\lambda(x)\ge|\nabla u(x)|^{p-2}\ge c_0^{p-2}|x|^{-(\gamma_h+1)(p-2)}=c_0^{p-2}|x|^{s}$. If $1<p<2$, the upper bound in \eqref{eq:sharp-grad} applied at $x$ and at $x_\lambda$, the latter with $|x_\lambda|\ge\kappa|x|\ge R_0$, gives $|\nabla u(x)|+|\nabla u(x_\lambda)|\le C|x|^{-\gamma_h-1}$, hence $\rho_\lambda(x)\ge c|x|^{s}$ because $p-2<0$. All constants are independent of $\lambda$.
\end{proof}

\begin{proof}[Proof of Theorem \ref{thm:main}(ii)]
Assume $u>0$, and write $w=(u-u_\lambda)^+$ and $\mathcal E_\lambda:=\int_{\Sigma_\lambda}\rho_\lambda|\nabla w|^2$. Testing \eqref{eq:weak} for $u$ and for $u_\lambda$ with the standard compactly supported cut-off approximations of $w$ in $\Sigma_\lambda$ (extended by zero across $T_\lambda$), subtracting, and then letting the outer cut-off tend to one, the vector inequality \cite[(2.10)]{DaiLiLiu2025}, the splitting \eqref{eq:splitting} and Lemma \ref{lem:nonlocal-diff} give, for a constant $C_1=C_1(p)>0$,
\begin{equation}\label{eq:classical-energy}
C_1\mathcal E_\lambda
\le\int_{A_\lambda}K\,w^2+\int_{A_\lambda}P_\lambda\,u_\lambda^{q-1}w .
\end{equation}
Here $\mathcal E_\lambda<\infty$: with $M_\lambda:=\max\{|\nabla u|,|\nabla u_\lambda|\}$ one has $\rho_\lambda|\nabla w|^2\le 2^{\max\{p,2\}}M_\lambda^{p}$ for every $1<p<\infty$, since $|\nabla w|\le2M_\lambda$ and $\rho_\lambda\le M_\lambda^{p-2}$ if $p<2$, $\rho_\lambda\le(2M_\lambda)^{p-2}$ if $p\ge2$. So the integrand is bounded, and by \eqref{eq:sharp-grad} and \eqref{eq:classical-weight} it is $O(|x|^{-p(\gamma_h+1)})$ on $A_\lambda$ at infinity, which is integrable because $p(\gamma_h+1)-N=p(\gamma_h-\gamma_e)>0$.

We use three pointwise bounds on $A_\lambda\cap\{|x|\ge R\}$, valid for $R$ large, uniformly for $\lambda\le0$ and, with constants depending on $L$, for $|\lambda|\le L$; only the first one depends on $L$, through Lemma \ref{lem:nonlocal-diff}. First, \eqref{eq:Kbound} together with the identity $-\beta-(p-2)\gamma_h=s-2-(\beta-p)$ gives $K\le C|x|^{s-2}|x|^{-(\beta-p)}$. Second, since $u\simeq|x|^{-\gamma_h}$ by \eqref{eq:sharp-u}, $u^{p^*-2}\le C|x|^{s-2}|x|^{-\beta_2}$ with $\beta_2:=(p^*-2)\gamma_h+s-2>0$. Third, $|x|^{s}\le c_w^{-1}\rho_\lambda$ by \eqref{eq:classical-weight}. Moreover $\beta>p$ by Proposition \ref{prop:Vdecay}, $s>2-N$, and \eqref{eq:cutoff-hardy} with $\ell=0$ applies to $w$ by the discussion following it.

\emph{Step 1: the plane starts.} For $\lambda\le-R_w$ one has $\Sigma_\lambda\subset\{|x|>|\lambda|\}$. Applying \eqref{eq:log-hardy} with $\ell=0$ and $R=|\lambda|$ to $w$ extended by zero, and then Lemma \ref{lem:I2},
\[
\int_{A_\lambda}Kw^2+\int_{A_\lambda}P_\lambda u_\lambda^{q-1}w
\le C\big(|\lambda|^{-(\beta-p)}+\|u\|^{2q-p^*}_{L^{p^*}(\Sigma_\lambda)}|\lambda|^{-\beta_2}\big)\Big(\frac{2}{N+s-2}\Big)^2c_w^{-1}\mathcal E_\lambda .
\]
For $|\lambda|$ large the coefficient is smaller than $C_1$, so \eqref{eq:classical-energy} forces $\mathcal E_\lambda=0$. If $1<p<2$, this gives $\nabla w=0$ a.e. on the complement of the common critical set, while $\nabla w=0$ also on the common critical set itself. If $p\ge2$, the critical set of $u$ has measure zero by \cite[Corollary 2.10, Remark 2.11]{DaiLiLiu2025}, hence $\rho_\lambda>0$ a.e. and again $\nabla w=0$ a.e. Since $w$ vanishes on $T_\lambda$ and at infinity, $w\equiv0$. Thus $u\le u_\lambda$ in $\Sigma_\lambda$ for all $\lambda$ sufficiently negative.

\emph{Step 2: the plane reaches a symmetry position.} Let $\lambda_0:=\sup\{\lambda:u\le u_\mu\text{ in }\Sigma_\mu\text{ for every }\mu\le\lambda\}$. It is finite, since $u>0$ tends to zero at infinity, and $u\le u_{\lambda_0}$ in $\Sigma_{\lambda_0}$ by continuity. Suppose $u\not\equiv u_{\lambda_0}$ in $\Sigma_{\lambda_0}$. By Lemma \ref{lem:strict-separation}, $u<u_{\lambda_0}$ a.e. in $\Sigma_{\lambda_0}$. Since $u_\lambda\to u_{\lambda_0}$ pointwise as $\lambda\to\lambda_0$, every such point lies outside $A_\lambda$ for $\lambda$ close to $\lambda_0$, while $|(\Sigma_\lambda\setminus\Sigma_{\lambda_0})\cap B_{3R}|\to0$. Dominated convergence of characteristic functions therefore gives
\begin{equation}\label{eq:classical-measure-small}
|A_\lambda\cap B_{3R}|\longrightarrow0\qquad\text{as }\lambda\downarrow\lambda_0,\ \text{for every fixed }R .
\end{equation}
Put $L:=|\lambda_0|+1$ and let $\lambda\in(\lambda_0,\lambda_0+1]$. Split both integrals in \eqref{eq:classical-energy} at $|x|=2R$ and set
\[
X_\lambda:=\int_{A_\lambda\cap B_{2R}}w^2 .
\]
On $B_{2R}$ the functions $K$ and $u^{p^*-2}$ are bounded uniformly in $\lambda$, because $u$ and $u_\lambda$ are bounded above and below by positive constants there. On $|x|\ge2R$ we use the three pointwise bounds and \eqref{eq:cutoff-hardy} with $\ell=0$; its annular error is bounded by $C_RX_\lambda$. This gives
\[
C_1\mathcal E_\lambda\le C\big(R^{-(\beta-p)}+R^{-\beta_2}\big)c_w^{-1}\mathcal E_\lambda+C_RX_\lambda .
\]
We now fix $R$ so large that the first coefficient is at most $C_1/2$; the constants depend only on $L$. Hence
\begin{equation}\label{eq:classical-reduced}
\tfrac12C_1\mathcal E_\lambda\le C_RX_\lambda .
\end{equation}
It remains to show $X_\lambda\le\theta_\lambda\mathcal E_\lambda$ with $\theta_\lambda\to0$ as $\lambda\downarrow\lambda_0$.

If $1<p<2$, let $G:=\|\nabla u\|_{L^\infty(\mathbb R^N)}<\infty$. On the set where $|\nabla u|+|\nabla u_\lambda|>0$ one has $\rho_\lambda\ge(2G)^{p-2}$, while on the common critical set $\nabla w=0$. Hence
\[
\int_{\Sigma_\lambda}|\nabla w|^2\le(2G)^{2-p}\mathcal E_\lambda.
\]
Since $w\in D^{1,2}(\Sigma_\lambda)$ by the discussion following \eqref{eq:cutoff-hardy}, H\"older's and Sobolev's inequalities give
\[
X_\lambda\le|A_\lambda\cap B_{2R}|^{2/N}\Big(\int_{\Sigma_\lambda}w^{2^*}\Big)^{2/2^*}\le C(2G)^{2-p}|A_\lambda\cap B_{2R}|^{2/N}\mathcal E_\lambda .
\]

If $p\ge2$, let $\xi\in C^\infty_c(B_{3R})$ with $0\le\xi\le1$, $\xi=1$ on $B_{2R}$ and $|\nabla\xi|\le2/R$, and put $\varrho:=|\nabla u|^{p-2}$. Let $r_R>2$ be furnished by Lemma \ref{lem:weighted-sobolev-local} on $B_{3R}$ and set $\delta_R:=1-2/r_R>0$. Since $\xi w$ is supported in $A_\lambda\cap B_{3R}$, H\"older's inequality followed by Lemma \ref{lem:weighted-sobolev-local} gives
\[
X_\lambda\le\int_{A_\lambda\cap B_{3R}}(\xi w)^2
\le |A_\lambda\cap B_{3R}|^{\delta_R}\|\xi w\|_{L^{r_R}(B_{3R})}^2
\le C_R|A_\lambda\cap B_{3R}|^{\delta_R}\int_{B_{3R}}\varrho|\nabla(\xi w)|^2.
\]
Because $\varrho\le\rho_\lambda$ for $p\ge2$ and $\varrho$ is bounded on $B_{3R}$,
\[
X_\lambda\le C_R|A_\lambda\cap B_{3R}|^{\delta_R}
\big(\mathcal E_\lambda+Y_\lambda\big),
\qquad
Y_\lambda:=\int_{A_\lambda\cap(B_{3R}\setminus B_{2R})}w^2.
\]
Since $s-2<0$, one has $|x|^{s-2}\ge(3R)^{s-2}$ on the annulus, and \eqref{eq:cutoff-hardy} with the third pointwise bound gives $Y_\lambda\le C_R(\mathcal E_\lambda+X_\lambda)$. Therefore
\[
X_\lambda\le C_R|A_\lambda\cap B_{3R}|^{\delta_R}(\mathcal E_\lambda+X_\lambda).
\]
By \eqref{eq:classical-measure-small}, the coefficient tends to zero as $\lambda\downarrow\lambda_0$; for $\lambda$ sufficiently close to $\lambda_0$ it can be absorbed, yielding $X_\lambda\le o(1)\mathcal E_\lambda$.

In both cases $\theta_\lambda\to0$ by \eqref{eq:classical-measure-small}. Inserting into \eqref{eq:classical-reduced} gives $\mathcal E_\lambda=0$ for $\lambda-\lambda_0>0$ small, hence $w\equiv0$ as in Step 1, that is $u\le u_\lambda$ in $\Sigma_\lambda$. This contradicts the definition of $\lambda_0$. Therefore $u\equiv u_{\lambda_0}$ in $\Sigma_{\lambda_0}$.

\emph{Step 3: conclusion.} Since the direction was arbitrary, $u$ is symmetric with respect to a hyperplane orthogonal to every direction, hence radially symmetric about some point $x_0$. Writing $u(x)=U(|x-x_0|)$ with $U\in C^1([0,\infty))$, the function $f=W[u]u^{q-1}$ is radial about $x_0$ and positive, and testing the equation with radial functions gives
\[
r^{N-1}|U'(r)|^{p-2}U'(r)=-\int_0^r t^{N-1}f(t)\,dt<0\qquad(r>0),
\]
the boundary term at $r=0$ vanishing because $U'$ is bounded. Hence $U'(r)<0$ for every $r>0$. This proves Theorem \ref{thm:main}(ii).
\end{proof}

\begin{remark}
The only structural input of the moving planes argument that distinguishes the present range from that of \cite{DaiLiLiu2025} is the monotonicity of $t\mapsto t^{q-1}$, that is, $q>1$, which is supplied by Lemma \ref{lem:qbigger1}. The exponent $p$ plays no role in the sign of either term of \eqref{eq:splitting}.
\end{remark}

\section{The nonclassical regimes \texorpdfstring{$\Lambda=0$ and $\Lambda<0$}{Lambda = 0 and Lambda < 0}}\label{sec:nonclassical}

\begin{remark}\label{rem:qbigger1-bad}
Unlike in Theorem \ref{thm:main}, the condition $q>1$ is a genuine assumption in Theorems \ref{thm:structure} and \ref{thm:slowrate}. At the resonant boundary it is automatic, because $\Lambda=0$ gives
\[
q-1=\frac{N-\sigma}{\gamma_h}>0.
\]
In the strict slow phase, $q\le1$ occurs exactly when $\sigma\ge2N-2N/p+2$, in which case the local factor $u^{q-1}$ becomes singular and \eqref{eq:main} is a different problem. Both nonclassical regimes are nonempty, and both require $p<2N/(N+2)$, since $\Lambda\le0$ forces $\sigma\ge2p/(2-p)$ and $\sigma<N$. For $\Lambda<0$, take $N=10$, $p=3/2$ and $\sigma=7$, which give $q=39/34$, $\gamma_e<\gamma_*<\gamma_h$ and $q\gamma_*>N$. For $\Lambda=0$ with the same $N$ and $p$, take $\sigma=2p/(2-p)=6$.
\end{remark}

\subsection{Saturation of the Hartree potential}

The following elementary consequence of integrability will be used in both nonclassical regimes.

\begin{lemma}\label{lem:Wsat}
Let $q>1$ and let $u$ be a positive $D^{1,p}(\mathbb{R}^{N})$-weak solution of \eqref{eq:main}. Suppose that for some $\delta_{\sharp}>0$,
\[
u(x)\le C_{\sharp}|x|^{-\delta_{\sharp}}\qquad(|x|\ge1),
\qquad q\delta_{\sharp}>N.
\]
Then $m_{0}:=\int_{\mathbb{R}^{N}}u^{q}\,dy$ is finite and positive, and there are $A_{1},A_{2}>0$ and $R_{5}\ge2$ such that
\begin{equation}\label{eq:Wsat-two}
A_{1}|x|^{-\sigma}\le W(x)\le A_{2}|x|^{-\sigma}
\qquad\text{for }|x|\ge R_{5}.
\end{equation}
One may take
\begin{equation}\label{eq:Wsat-low}
A_{1}:=\Big(\frac{2}{3}\Big)^{\sigma}\frac{m_{0}}{2}.
\end{equation}
\end{lemma}

\begin{proof}
Since $u\in L^{\infty}(\mathbb{R}^{N})$ by Corollary \ref{cor:structure}, the assumed decay and $q\delta_{\sharp}>N$ give $m_{0}<\infty$, while positivity gives $m_{0}>0$. Choose $R_{5}\ge2$ so large that $\int_{B_{R_{5}/2}}u^{q}\ge m_{0}/2$. For $|x|\ge R_{5}$ and $|y|\le|x|/2$ we have $|x-y|\le3|x|/2$, hence
\[
W(x)\ge\int_{|y|\le|x|/2}\frac{u^{q}(y)}{|x-y|^{\sigma}}\,dy
\ge\Big(\frac{3|x|}{2}\Big)^{-\sigma}\int_{B_{|x|/2}}u^{q}\,dy
\ge A_{1}|x|^{-\sigma}.
\]
For the upper bound split $W=J_{1}+J_{2}+J_{3}$ according to $|y|\le|x|/2$, $\{|y|>|x|/2,\ |x-y|<|x|/4\}$ and the rest. Then
\[
J_{1}\le(|x|/2)^{-\sigma}m_{0},\qquad
J_{3}\le(|x|/4)^{-\sigma}m_{0},
\]
while on the second region $|y|\ge|x|/2$ and therefore
\[
J_{2}\le C_{\sharp}^{q}\Big(\frac{|x|}{2}\Big)^{-q\delta_{\sharp}}
\int_{|z|<|x|/4}|z|^{-\sigma}\,dz
\le C|x|^{N-\sigma-q\delta_{\sharp}}
\le C|x|^{-\sigma}.
\]
This proves \eqref{eq:Wsat-two}.
\end{proof}

The preceding two-sided estimate can be sharpened to an exact far-field coefficient once $u^q$ is integrable.

\begin{lemma}[Mass asymptotics of the Hartree potential]\label{lem:Wmass}
Assume $u^q\in L^1(\mathbb R^N)$ and, for some $\delta>N/q$, $u(x)\le C|x|^{-\delta}$ for large $|x|$. Put
\[
m_0:=\int_{\mathbb R^N}u^q(y)\,dy.
\]
Then
\begin{equation}\label{eq:Wmass}
|x|^{\sigma}W(x)\longrightarrow m_0
\qquad\text{as }|x|\to\infty,
\end{equation}
uniformly with respect to the direction $x/|x|$.
\end{lemma}

\begin{proof}
Let $r=|x|$. Fix $M>1$ and split the integral defining $W$ into $|y|\le M$ and $|y|>M$. On $B_M$, one has
$r^{\sigma}|x-y|^{-\sigma}\to1$ uniformly in the direction of $x$, so the contribution tends to $\int_{B_M}u^q$. On the part $|y|>M$, split again according to $|x-y|\ge r/2$ and $|x-y|<r/2$. The first contribution is bounded by $2^{\sigma}\int_{|y|>M}u^q$, uniformly in $x$. On the second region one has $|y|\ge r/2$, hence, for $r$ large,
\[
r^{\sigma}\int_{|x-y|<r/2}\frac{u^q(y)}{|x-y|^{\sigma}}\,dy
\le Cr^{\sigma-q\delta}\int_{|z|<r/2}|z|^{-\sigma}\,dz
\le Cr^{N-q\delta}\longrightarrow0.
\]
First let $r\to\infty$ and then $M\to\infty$.
\end{proof}

\subsection{The resonant boundary \texorpdfstring{$\Lambda=0$}{Lambda = 0}}\label{sec:resonance}

\begin{theorem}[Resonant logarithmic profile and symmetry]\label{thm:resonant}
Let $N\ge3$, $p>1$ and $2p<\sigma<N$, and assume that $q:=p^*_{\sigma}>1$ and $\Lambda=0$; here $q>1$ is automatic. Let $u$ be a positive $D^{1,p}(\mathbb{R}^{N})$-weak solution of \eqref{eq:main}. Put
\[
a_{0}:=\frac{1}{p-q}.
\]
There are $c,C>0$ and $R\ge e$ such that
\begin{equation}\label{eq:resonant-rate}
\frac{c(\log|x|)^{a_{0}}}{|x|^{\gamma_h}}
\le u(x)\le
\frac{C(\log|x|)^{a_{0}}}{|x|^{\gamma_h}}
\qquad\text{for }|x|\ge R,
\end{equation}
and
\begin{equation}\label{eq:resonant-grad}
\frac{c(\log|x|)^{a_{0}}}{|x|^{\gamma_h+1}}
\le |\nabla u(x)|\le
\frac{C(\log|x|)^{a_{0}}}{|x|^{\gamma_h+1}}
\qquad\text{for }|x|\ge R.
\end{equation}
Moreover $u$ is radially symmetric about some point $x_0\in\mathbb R^N$ and strictly decreasing in the radial variable.
\end{theorem}

\begin{proof}[Proof of the asymptotic and gradient assertions in Theorem \ref{thm:resonant}]
At $\Lambda=0$ one has $\gamma_{*}=\gamma_{h}$ and
\begin{equation}\label{eq:resonance-identities}
\sigma+(q-1)\gamma_{h}=N,
\qquad
q-1=\frac{N-\sigma}{\gamma_h}>0.
\end{equation}
Thus $1<q<p$. Moreover Lemma \ref{lem:exponents}(e) gives $q\gamma_h>N$. Set
\[
\gamma:=\gamma_h,
\qquad
\theta:=\frac{q-1}{p-1}\in(0,1),
\qquad
a_{0}:=\frac1{p-q}.
\]
Then
\begin{equation}\label{eq:a0theta}
a_{0}(1-\theta)=\frac1{p-1}.
\end{equation}
Choose
\[
\max\{\gamma_e,N/q\}<\delta_{0}<\gamma
\]
and write $\varepsilon_{0}:=\gamma-\delta_{0}>0$. Theorem \ref{thm:bootstrap} gives $C_{0}>0$ such that
\begin{equation}\label{eq:res-start-upper}
u(x)\le C_{0}|x|^{-\delta_{0}}
\qquad(|x|\ge1).
\end{equation}
Hence Lemma \ref{lem:Wsat} supplies $A_{1},A_{2}>0$ and $R_{5}\ge2$ for which \eqref{eq:Wsat-two} holds. On the other hand, the Serrin--Zou lower bound \eqref{eq:lower} gives
\begin{equation}\label{eq:res-start-lower}
u(x)\ge c_{0}|x|^{-\gamma}
\qquad(|x|\ge R_{1})
\end{equation}
for some $c_{0}>0$ and $R_{1}\ge1$.

\emph{Upper bound.}
Increase $\rho\ge\max\{R_{1},R_{5},e\}$ if necessary and set
\[
\delta_{k}:=\gamma-\varepsilon_{0}\theta^{k},
\qquad k=0,1,\ldots .
\]
Since $q\delta_{0}>N$, the saturated branch of $T$ is operative, and \eqref{eq:resonance-identities} gives
\[
T(\delta_{k})=\delta_{k+1}.
\]
We claim that there are constants $C_k>0$ such that
\begin{equation}\label{eq:res-upper-k}
u(x)\le C_k|x|^{-\delta_k}\qquad(|x|\ge\rho)
\end{equation}
and
\begin{equation}\label{eq:Ck-growth}
C_k\le M\theta^{-a_{0}k}
\end{equation}
with $M$ independent of $k$. The assertion for $k=0$ follows from \eqref{eq:res-start-upper}. If \eqref{eq:res-upper-k} holds, then by \eqref{eq:Wsat-two}
\[
Wu^{q-1}\le A_{2}C_k^{q-1}|x|^{-\delta_{k+1}(p-1)-p}.
\]
Put
\[
C_{k+1}:=\max\left\{
\left(\frac{A_{2}C_k^{q-1}}{h(\delta_{k+1})}\right)^{1/(p-1)},
\ \big(\max_{|x|=\rho}u\big)\rho^{\delta_{k+1}}
\right\}.
\]
Because $\gamma$ is the zero of $h$,
\[
h(\delta_{k+1})
=(p-1)\delta_{k+1}^{p-1}(\gamma-\delta_{k+1})
\ge c_h\varepsilon_{0}\theta^{k+1},
\qquad
c_h:=(p-1)\delta_0^{p-1}>0.
\]
The exterior comparison Lemma \ref{lem:comparison}, applied exactly as in Lemma \ref{lem:onestep}, gives \eqref{eq:res-upper-k} at level $k+1$. If
\[
D:=\left(\frac{A_{2}}{c_h\varepsilon_{0}}\right)^{1/(p-1)},
\qquad
B:=\big(\max_{|x|=\rho}u\big)\rho^{\gamma},
\]
then
\begin{equation}\label{eq:Crec}
C_{k+1}\le\max\left\{D\theta^{-(k+1)/(p-1)}C_k^{\theta},B\right\}.
\end{equation}
Using \eqref{eq:a0theta}, one checks that
\[
a_{0}\theta k+\frac{k+1}{p-1}=a_{0}(k+1)-a_{0}\theta.
\]
Therefore, after choosing $M\ge\max\{C_0,B\}$ so large that
$D M^{\theta}\theta^{a_{0}\theta}\le M$, induction in \eqref{eq:Crec} proves \eqref{eq:Ck-growth}.

Now let $r=|x|$ be large and choose $k=k(r)$ so that
\begin{equation}\label{eq:kchoice}
1\le\varepsilon_{0}\theta^{k}\log r<\theta^{-1}.
\end{equation}
Then \eqref{eq:res-upper-k}--\eqref{eq:Ck-growth} give
\[
u(x)
\le M r^{-\gamma}\theta^{-a_{0}k}
\exp\!\big(\varepsilon_{0}\theta^{k}\log r\big)
\le C r^{-\gamma}(\log r)^{a_{0}},
\]
which proves the upper half of \eqref{eq:resonant-rate}.

\emph{Lower bound.}
For $b>0$ define
\[
\phi_b(r):=r^{-\gamma}(\log r)^b.
\]
If $L=\log r$, a direct radial computation, using $\gamma+1=(N-1)/(p-1)$, gives
\begin{equation}\label{eq:log-plap-general}
-\Delta_p\phi_b
=H_b(L)r^{-N}L^{b(p-1)-1},
\end{equation}
where
\begin{equation}\label{eq:Hb}
H_b(L):=b(p-1)L^{-(p-1)}
(\gamma L-b)^{p-2}(\gamma L-b+1).
\end{equation}
Set
\[
b_0:=0,
\qquad
b_{k+1}:=\frac{1+(q-1)b_k}{p-1}.
\]
Then
\begin{equation}\label{eq:bk}
b_k=a_{0}(1-\theta^k)\uparrow a_{0},
\qquad
b_{k+1}(p-1)-1=b_k(q-1).
\end{equation}
Choose $\rho$ still larger, if necessary, so that every $\phi_b$ with $b\in[b_1,a_0]$ is decreasing on $[\rho,\infty)$ and
\begin{equation}\label{eq:Hbuniform}
0<H_b(\log r)\le D_{0}
\qquad(r\ge\rho,\ b\in[b_1,a_0])
\end{equation}
for some $D_0<\infty$. We prove inductively that
\begin{equation}\label{eq:res-lower-k}
u(x)\ge c_k|x|^{-\gamma}(\log|x|)^{b_k}
\qquad(|x|\ge\rho)
\end{equation}
with $c_k$ bounded below away from zero. The case $k=0$ is \eqref{eq:res-start-lower}, after reducing $c_0$ if necessary. Suppose \eqref{eq:res-lower-k} holds. By \eqref{eq:Wsat-two}, \eqref{eq:resonance-identities} and \eqref{eq:bk},
\[
Wu^{q-1}\ge
A_{1}c_k^{q-1}r^{-N}(\log r)^{b_k(q-1)}.
\]
Let
\[
\widehat c_{k+1}:=
\min\left\{
\left(\frac{A_{1}c_k^{q-1}}{D_0}\right)^{1/(p-1)},
\frac{\min_{|x|=\rho}u}
{\max_{b\in[b_1,a_0]}\phi_b(\rho)}
\right\}.
\]
For $R>\rho$ put
\[
v_R(x):=\widehat c_{k+1}
\big(\phi_{b_{k+1}}(|x|)-\phi_{b_{k+1}}(R)\big)
\qquad(\rho<|x|<R).
\]
By \eqref{eq:log-plap-general}--\eqref{eq:Hbuniform},
\[
-\Delta_pv_R\le A_{1}c_k^{q-1}r^{-N}(\log r)^{b_k(q-1)}
\le -\Delta_pu.
\]
Moreover $v_R=0\le u$ on $|x|=R$ and $v_R\le u$ on $|x|=\rho$. Lemma \ref{lem:comparison} gives $v_R\le u$ in the annulus, and letting $R\to\infty$ yields \eqref{eq:res-lower-k} at level $k+1$ with $c_{k+1}:=\widehat c_{k+1}$.

Put
\[
D_1:=\left(\frac{A_1}{D_0}\right)^{1/(p-1)},
\qquad
b_{\rho}:=\frac{\min_{|x|=\rho}u}
{\max_{b\in[b_1,a_0]}\phi_b(\rho)}>0.
\]
Then $c_{k+1}=\min\{D_1c_k^{\theta},b_{\rho}\}$. Since $0<\theta<1$, the map $t\mapsto D_1t^{\theta}$ has a positive fixed point $t_*$. Hence
\[
\underline c:=\min\{c_0,t_*,b_{\rho}\}>0
\]
satisfies $c_k\ge\underline c$ for every $k$. Letting $k\to\infty$ in \eqref{eq:res-lower-k} and using \eqref{eq:bk} gives
\[
u(x)\ge\underline c |x|^{-\gamma}(\log|x|)^{a_0}
\qquad(|x|\ge\rho),
\]
which completes \eqref{eq:resonant-rate}.

It remains to prove \eqref{eq:resonant-grad}. Let $R\ge2\rho$, put $L_R:=\log R$, and define on the fixed annulus $\mathcal A:=\{1/2<|y|<2\}$
\[
v_R(y):=R^{\gamma_h}L_R^{-a_{0}}u(Ry).
\]
The two-sided estimate just proved gives $\|v_R\|_{L^{\infty}(\mathcal A)}\le C$. Moreover, using \eqref{eq:Wsat-two},
\begin{align*}
-\Delta_pv_R(y)
&=R^{\gamma_h(p-1)+p}L_R^{-a_{0}(p-1)}
W(Ry)u(Ry)^{q-1}\\
&\le C L_R^{-a_{0}(p-q)}\le C,
\end{align*}
because $\gamma_h(p-1)+p=N=\sigma+(q-1)\gamma_h$ and $a_{0}(p-q)=1$. The interior gradient estimates of \cite{DiBenedetto1983,Tolksdorf1984} therefore give $\|\nabla v_R\|_{L^{\infty}(\{3/4<|y|<3/2\})}\le C$ uniformly in $R$. Scaling back gives the upper bound in \eqref{eq:resonant-grad}.

\emph{Lower gradient bound.}
We adapt the blow-down argument of Sciunzi \cite[Theorem 2.2]{Sciunzi2016}. Since $q\gamma>N$, the rate \eqref{eq:resonant-rate} implies $u^q\in L^1(\mathbb R^N)$; moreover it implies $u(x)\le C|x|^{-\delta}$ for some $\delta>N/q$. Hence Lemma \ref{lem:Wmass} applies. Suppose, by contradiction, that the lower bound in \eqref{eq:resonant-grad} fails. Then there are $x_n$ with $R_n:=|x_n|\to\infty$ such that
\[
R_n^{\gamma+1}(\log R_n)^{-a_0}|\nabla u(x_n)|\longrightarrow0.
\]
Set
\[
v_n(y):=R_n^{\gamma}(\log R_n)^{-a_0}u(R_ny).
\]
For every fixed annulus $0<a<|y|<A<\infty$, \eqref{eq:resonant-rate} gives
\[
c_A|y|^{-\gamma}\le v_n(y)\le C_a|y|^{-\gamma}
\]
for all large $n$. Furthermore, by \eqref{eq:resonance-identities}, $a_0(p-q)=1$ and Lemma \ref{lem:Wmass},
\[
-\Delta_pv_n
=(\log R_n)^{-1}\big(R_n^{\sigma}W(R_ny)\big)v_n^{q-1}
\longrightarrow0
\]
uniformly on compact annuli. The local $C^{1,\alpha}$ estimates therefore yield, by a diagonal argument,
\[
v_n\longrightarrow v_\infty
\qquad\text{in }C^1_{\mathrm{loc}}(\mathbb R^N\setminus\{0\}),
\]
where $v_\infty$ is positive and $p$-harmonic in $\mathbb R^N\setminus\{0\}$. Passing to the two-sided bounds and then letting the annuli expand gives
\[
v_\infty(x)\asymp |x|^{-\gamma_h},
\qquad
v_\infty(x)\to\infty\ (|x|\to0),
\qquad
v_\infty(x)\to0\ (|x|\to\infty).
\]
By the $p$-harmonic rigidity theorem used by Sciunzi \cite[Theorem 2.1]{Sciunzi2016},
\[
v_\infty(x)=\alpha |x|^{-\gamma_h}
\]
for some $\alpha>0$. Let $y_n=x_n/R_n$. After passing to a subsequence $y_n\to y_\infty\in S^{N-1}$. The assumed failure of the gradient lower bound gives $|\nabla v_n(y_n)|\to0$, hence $\nabla v_\infty(y_\infty)=0$, a contradiction. This proves the lower bound in \eqref{eq:resonant-grad}.
\end{proof}

\subsection{The slow phase \texorpdfstring{$\Lambda<0$}{Lambda < 0}}\label{sec:slow}

\begin{theorem}[Slow profile, exact far field and symmetry]\label{thm:slowrate}
Let $N\ge3$, $p>1$ and $2p<\sigma<N$, and assume that $q:=p^*_{\sigma}>1$ and $\Lambda<0$. Let $u$ be a positive $D^{1,p}(\mathbb{R}^{N})$-weak solution of \eqref{eq:main}. Then $u^q\in L^1(\mathbb R^N)$. If
\[
m_0:=\int_{\mathbb R^N}u^q\,dx,
\qquad
A_*:=\left(\frac{m_0}{h(\gamma_*)}\right)^{\!1/(p-q)},
\]
then
\begin{equation}\label{eq:slow-exact-u}
\sup_{|\omega|=1}\left|r^{\gamma_*}u(r\omega)-A_*\right|\longrightarrow0
\qquad\text{as }r\to\infty,
\end{equation}
and
\begin{equation}\label{eq:slow-exact-grad}
\sup_{|\omega|=1}\left|r^{\gamma_*+1}\nabla u(r\omega)+A_*\gamma_*\omega\right|\longrightarrow0.
\end{equation}
In particular, there are $c,C>0$ and $R\ge1$ such that
\begin{equation}\label{eq:slowrate}
\frac{c}{|x|^{\gamma_{*}}}\le u(x)\le\frac{C}{|x|^{\gamma_{*}}},
\qquad
\frac{c}{|x|^{\gamma_*+1}}\le |\nabla u(x)|\le\frac{C}{|x|^{\gamma_*+1}}
\end{equation}
for $|x|\ge R$. Moreover $u$ is radially symmetric about some point $x_0\in\mathbb R^N$ and strictly decreasing in the radial variable.
\end{theorem}

Throughout this subsection $N\ge3$, $1<p<2$, $2p<\sigma<N$, $\Lambda<0$ and $q>1$, and $u$ is a positive $D^{1,p}(\mathbb{R}^{N})$-weak solution of \eqref{eq:main}. By Lemma \ref{lem:exponents}(b),(c) and Lemma \ref{lem:Tprops}(d) we then have
\begin{equation}\label{eq:slow-order}
\gamma_{e}<\gamma_{*}<\gamma_{h},
\qquad
q\gamma_{*}>N,
\qquad
q\gamma_{h}>N,
\end{equation}
the last because $\Lambda<0$ forces $\sigma<2N/p$ by Lemma \ref{lem:exponents}(e). We write
\[
\theta:=\frac{q-1}{p-1}\in(0,1),
\qquad
T(\delta)=\frac{\sigma+(q-1)\delta-p}{p-1}
\quad\text{for }q\delta\ge N,
\]
the saturated branch of the map \eqref{eq:adelta}, which is affine with slope $\theta$ and fixed point $\gamma_{*}$ by Lemma \ref{lem:Tprops}(c). In particular, $T$ maps $[\gamma_{*},\gamma_{h}]$ into itself, $T(\delta)<\delta$ for $\delta>\gamma_{*}$ and $T(\delta)>\delta$ for $\delta<\gamma_{*}$, and by \eqref{eq:slow-order} every $\delta\in[\gamma_{*},\gamma_{h}]$ satisfies $q\delta\ge N$, so the saturated formula is the operative one throughout.

\subsection{One step of the iteration}

\begin{lemma}\label{lem:slow-step}
Fix $\rho\ge R_{5}$.
\begin{enumerate}[label=\textup{(\roman*)}]
\item Let $\delta\in[\gamma_{*},\gamma_{h}]$ and set $\nu:=T(\delta)$. If $u(x)\ge c|x|^{-\delta}$ for $|x|\ge\rho$, then $u(x)\ge c'|x|^{-\nu}$ for $|x|\ge\rho$, where
\[
c'=\min\Big\{\Big(\frac{A_{1}c^{\,q-1}}{h(\nu)}\Big)^{\frac{1}{p-1}},\ \big(\min_{|x|=\rho}u\big)\rho^{\nu}\Big\}.
\]
\item Let $\delta>0$ satisfy $q\delta>N$ and
\[
\gamma_{e}<\nu:=T(\delta)<\gamma_{h}.
\]
If \eqref{eq:Wsat-two} holds and $u(x)\le C|x|^{-\delta}$ for $|x|\ge\rho$, then $u(x)\le C'|x|^{-\nu}$ for $|x|\ge\rho$, where
\[
C'=\max\Big\{\Big(\frac{A_{2}C^{\,q-1}}{h(\nu)}\Big)^{\frac{1}{p-1}},\ \big(\max_{|x|=\rho}u\big)\rho^{\nu}\Big\}.
\]
\end{enumerate}
In (i), $q\delta>N$ by \eqref{eq:slow-order}, and $\nu\in[\gamma_{*},T(\gamma_{h})]\subset(0,\gamma_{h})$; in (ii), $h(\nu)>0$ is part of the stated range. Thus all barriers below are admissible.
\end{lemma}

\begin{proof}
In (i), $q\delta>N$ by \eqref{eq:slow-order}, while this is assumed explicitly in (ii). Hence the saturated formula for $T$ is operative in both cases and gives
\begin{equation}\label{eq:slow-exponents}
\nu(p-1)+p=\sigma+(q-1)\delta .
\end{equation}
Thus the two sides of every inequality below carry the same power of $|x|$ and no radius enters the constants.

(i) By \eqref{eq:Wsat-low} and $q>1$, for $|x|\ge\rho$,
\[
f:=Wu^{q-1}\ge A_{1}c^{\,q-1}|x|^{-\sigma-(q-1)\delta}
=A_{1}c^{\,q-1}|x|^{-\nu(p-1)-p}.
\]
Let $R>\rho$ and put $v_{R}(x):=c'\big(|x|^{-\nu}-R^{-\nu}\big)$ on $\Omega_{R}:=\{\rho<|x|<R\}$. Since adding a constant does not change the $p$-Laplacian, \eqref{eq:barrier} gives
\[
-\Delta_{p}v_{R}=h(\nu)(c')^{p-1}|x|^{-\nu(p-1)-p}\le A_{1}c^{\,q-1}|x|^{-\nu(p-1)-p}\le f
\qquad\text{in }\Omega_{R},
\]
by the first entry in the definition of $c'$. On $|x|=R$ we have $v_{R}=0\le u$, and on $|x|=\rho$ we have $v_{R}\le c'\rho^{-\nu}\le\min_{|x|=\rho}u$ by the second entry. Lemma \ref{lem:comparison} on the bounded domain $\Omega_{R}$, applied to the pair $v_{R},u$ in place of $u,v$, yields $v_{R}\le u$ in $\Omega_{R}$. Letting $R\to\infty$ gives $u\ge c'|x|^{-\nu}$ for $|x|\ge\rho$.

(ii) Symmetrically, by \eqref{eq:Wsat-two}, $f\le A_{2}C^{\,q-1}|x|^{-\nu(p-1)-p}$ for $|x|\ge\rho$. Put $v(x):=C'|x|^{-\nu}$ on $\Omega:=\{|x|>\rho\}$. Then $-\Delta_{p}v=h(\nu)(C')^{p-1}|x|^{-\nu(p-1)-p}\ge f$ by the first entry in the definition of $C'$, and $v\ge u$ on $\partial\Omega$ by the second. Since $\nu>\gamma_{e}$, we have $\nabla v\in L^{p}(\Omega)$, so $(u-v)^{+}\in D^{1,p}(\Omega)$ as in the proof of Lemma \ref{lem:onestep}. The exterior-domain assumptions of Lemma \ref{lem:comparison} are therefore satisfied, and it gives $u\le v$ in $\Omega$.
\end{proof}

The exact coefficient in the slow phase comes from the following rigidity statement for the scale-invariant limit equation.

\begin{lemma}[Rigidity of the slow blow-down equation]\label{lem:slow-rigidity}
Let $1<q<p<2$, let $0<\gamma<\gamma_h$ satisfy
\[
\gamma(p-q)=\sigma-p,
\]
and let $m>0$. Suppose $V\in C^1_{\mathrm{loc}}(\mathbb R^N\setminus\{0\})$ is positive and solves
\begin{equation}\label{eq:limit-slow}
-\Delta_pV=m|x|^{-\sigma}V^{q-1}
\qquad\text{in }\mathbb R^N\setminus\{0\},
\end{equation}
with
\[
c|x|^{-\gamma}\le V(x)\le C|x|^{-\gamma}
\qquad(x\ne0).
\]
Then
\begin{equation}\label{eq:slow-rigid-profile}
V(x)=A|x|^{-\gamma},
\qquad
A=\left(\frac{m}{h(\gamma)}\right)^{1/(p-q)}.
\end{equation}
\end{lemma}

\begin{proof}
The function $U(x):=A|x|^{-\gamma}$ with $A$ as in \eqref{eq:slow-rigid-profile} solves \eqref{eq:limit-slow} by \eqref{eq:barrier}. The quotient $V/U$ is bounded above and below by positive constants. Set
\[
t_0:=\sup_{\mathbb R^N\setminus\{0\}}\frac{V}{U}<\infty.
\]
Assume $t_0>1$. Choose $x_k$ with $V(x_k)/(t_0U(x_k))\to1$, put $r_k=|x_k|$, and choose rotations $Q_k$ with $Q_ke_1=x_k/r_k$. The scale-invariance of \eqref{eq:limit-slow} gives
\[
V_k(x):=r_k^{\gamma}V(r_kQ_kx),
\]
which solves the same equation and satisfies the same two-sided bounds. Local regularity yields, up to a subsequence, $V_k\to V_\infty$ in $C^1_{\mathrm{loc}}(\mathbb R^N\setminus\{0\})$. Since $U$ is invariant under the same rotations and scaling,
\[
V_\infty\le t_0U,
\qquad
V_\infty(e_1)=t_0U(e_1).
\]
On the other hand,
\[
-\Delta_p(t_0U)
=t_0^{p-q}m|x|^{-\sigma}(t_0U)^{q-1}
>m|x|^{-\sigma}(t_0U)^{q-1}
\ge -\Delta_pV_\infty .
\]
At this interior contact point, $\nabla V_\infty(e_1)=t_0\nabla U(e_1)\ne0$. Hence both gradients are nonzero in a sufficiently small neighborhood of $e_1$, and the strong comparison principle \cite[Lemma 2.7]{DaiLiLiu2025}, applied with $\tilde\Lambda=0$ in a small ball around $e_1$ to the pair $V_\infty\le t_0U$, gives $V_\infty\equiv t_0U$ near $e_1$. There, however,
\[
-\Delta_pV_\infty=-\Delta_p(t_0U)>m|x|^{-\sigma}(t_0U)^{q-1}=m|x|^{-\sigma}V_\infty^{q-1},
\]
contradicting \eqref{eq:limit-slow}. Thus $t_0\le1$ and $V\le U$.

For the reverse inequality put $s_0:=\sup U/V<\infty$ and assume $s_0>1$. As above, $-\Delta_p(s_0V)=s_0^{p-q}m|x|^{-\sigma}(s_0V)^{q-1}>m|x|^{-\sigma}(s_0V)^{q-1}$. The blow-down is again applied to $V$, along a sequence realising $s_0$, while $U$ is left unchanged because it is invariant under the rescaling. The limit satisfies $U\le s_0V_\infty$ with equality at $e_1$, where $\nabla U(e_1)=s_0\nabla V_\infty(e_1)\ne0$. Since $-\Delta_pU=m|x|^{-\sigma}U^{q-1}\le m|x|^{-\sigma}(s_0V_\infty)^{q-1}<-\Delta_p(s_0V_\infty)$ and $q>1$, the strong comparison principle gives $U\equiv s_0V_\infty$ near $e_1$, which contradicts the strict inequality exactly as before. Hence $U\le V$. Therefore $V\equiv U$.
\end{proof}

\subsection{Proof of Theorem \ref{thm:slowrate}}

\begin{proof}[Proof of the asymptotic and gradient assertions in Theorem \ref{thm:slowrate}]
The constants are chosen in the order $\delta_{\sharp}\to C_{\sharp}\to A_{1},A_{2},R_{5}\to\rho\to b,B$, and no later choice feeds back into an earlier one.

Since $q\gamma_{*}>N$ and $\gamma_{*}>\gamma_{e}$, we may fix
\[
\max\{\gamma_{e},N/q\}<\delta_{\sharp}<\gamma_{*}.
\]
Theorem \ref{thm:bootstrap} then provides $C_{\sharp}$ with $u\le C_{\sharp}|x|^{-\delta_{\sharp}}$ for $|x|\ge1$. Thus Lemma \ref{lem:Wsat} applies in full, and we fix $A_{1},A_{2},R_{5}$ accordingly. Let $R_{1}$ be the radius in \eqref{eq:lower} and set $\rho:=\max\{R_{5},R_{1}\}\ge2$, so that in particular $u\ge c_{0}|x|^{-\gamma_{h}}$ for $|x|\ge\rho$ with $c_{0}:=c_{1}$. Put
\[
b:=\big(\min_{|x|=\rho}u\big)\rho^{\gamma_{*}}>0,
\qquad
B:=\big(\max_{|x|=\rho}u\big)\rho^{\gamma_{*}}<\infty,
\]
which are positive and finite by continuity and positivity of $u$. Every exponent $\nu$ produced below by the lower iteration lies in $[\gamma_{*},\gamma_{h}]$ and every one produced by the upper iteration lies in $(0,\gamma_{*}]$; since $\rho>1$ this gives
\begin{equation}\label{eq:bB}
\big(\min_{|x|=\rho}u\big)\rho^{\nu}\ge b\ \text{ in the first case},
\qquad
\big(\max_{|x|=\rho}u\big)\rho^{\nu}\le B\ \text{ in the second}.
\end{equation}

\emph{The lower bound.} Set $\delta_{0}:=\gamma_{h}$ and $\delta_{k+1}:=T(\delta_{k})$. Since $T$ is affine with slope $\theta\in(0,1)$ and fixed point $\gamma_{*}$,
\[
\delta_{k}=\gamma_{*}+(\gamma_{h}-\gamma_{*})\theta^{k}\ \downarrow\ \gamma_{*},
\qquad
\delta_{k}\in[\gamma_{*},\gamma_{h}]\ \text{ for all }k .
\]
By the choice of $\rho$ we have $u\ge c_{0}|x|^{-\delta_{0}}$ for $|x|\ge\rho$. Define $c_{k+1}$ from $c_{k}$ by Lemma \ref{lem:slow-step}(i), so that $u\ge c_{k}|x|^{-\delta_{k}}$ for $|x|\ge\rho$ and all $k$. Let
\[
h_{\max}:=\max_{[\gamma_{*},\,T(\gamma_{h})]}h>0,
\qquad
\Phi(t):=\Big(\frac{A_{1}}{h_{\max}}\Big)^{\frac{1}{p-1}}t^{\theta},
\]
so that $c_{k+1}\ge\min\{\Phi(c_{k}),b\}$ by \eqref{eq:bB}. The map $\Phi$ is increasing and, because $\theta<1$, it has a unique positive fixed point $t^{*}$ with $\Phi(t)\ge t$ for $0<t\le t^{*}$. Put $\gamma:=\min\{c_{0},t^{*},b\}>0$. If $c_{k}\ge\gamma$ then $\Phi(c_{k})\ge\Phi(\gamma)\ge\gamma$, and $b\ge\gamma$, whence $c_{k+1}\ge\gamma$. By induction $c_{k}\ge\gamma$ for every $k$, so
\[
u(x)\ge\gamma|x|^{-\delta_{k}}\qquad\text{for }|x|\ge\rho\text{ and all }k .
\]
Letting $k\to\infty$ and using $\delta_{k}\downarrow\gamma_{*}$ gives $u(x)\ge\gamma|x|^{-\gamma_{*}}$ for $|x|\ge\rho$.

\emph{The upper bound.} Set $\tilde\delta_{0}:=\delta_{\sharp}$ and $\tilde\delta_{k+1}:=T(\tilde\delta_{k})$, so that $\tilde\delta_{k}\uparrow\gamma_{*}$ and $\tilde\delta_{k}\in[\delta_{\sharp},\gamma_{*}]$ for all $k$; in particular $q\tilde\delta_{k}\ge q\delta_{\sharp}>N$, so the saturated branch is the operative one. Moreover, for every $k$, $\tilde\delta_{k}\ge\delta_{\sharp}>\gamma_{e}$ and $\tilde\delta_{k}<\gamma_{*}$ unless the iteration has already reached its limit, hence Lemma \ref{lem:Tprops}(e) gives
\[
\gamma_{e}<T(\tilde\delta_{k})\le\gamma_{*}<\gamma_{h}.
\]
Thus the assumptions of Lemma \ref{lem:slow-step}(ii) are satisfied at every finite step. Enlarging $C_{\sharp}$ we may assume $u\le C_{\sharp}|x|^{-\tilde\delta_{0}}$ for $|x|\ge\rho$, and we define $C_{k+1}$ from $C_{k}$ by Lemma \ref{lem:slow-step}(ii). With
\[
h_{\min}:=\min_{[T(\delta_{\sharp}),\,\gamma_{*}]}h>0,
\qquad
\Psi(t):=\Big(\frac{A_{2}}{h_{\min}}\Big)^{\frac{1}{p-1}}t^{\theta},
\]
we have $C_{k+1}\le\max\{\Psi(C_{k}),B\}$ by \eqref{eq:bB}. Again $\Psi$ is increasing with a unique positive fixed point $s^{*}$ and $\Psi(t)\le t$ for $t\ge s^{*}$. Put $\Gamma:=\max\{C_{\sharp},s^{*},B\}$. If $C_{k}\le\Gamma$ then $\Psi(C_{k})\le\Psi(\Gamma)\le\Gamma$ and $B\le\Gamma$, so $C_{k+1}\le\Gamma$, and by induction $C_{k}\le\Gamma$ for every $k$. Hence $u(x)\le\Gamma|x|^{-\tilde\delta_{k}}$ for $|x|\ge\rho$ and all $k$, and letting $k\to\infty$ gives $u(x)\le\Gamma|x|^{-\gamma_{*}}$ for $|x|\ge\rho$.

Taking $c=\gamma$, $C=\Gamma$ and $R=\rho$ gives the coarse two-sided bound in \eqref{eq:slowrate}. Since the starting exponent $\delta_\sharp$ was chosen with $q\delta_\sharp>N$, we also have $u^q\in L^1(\mathbb R^N)$. Put $m_0:=\int u^q$. Lemma \ref{lem:Wmass} then gives
\begin{equation}\label{eq:slow-W-exact}
R^{\sigma}W(Ry)\longrightarrow m_0|y|^{-\sigma}
\end{equation}
uniformly on compact subsets of $\mathbb R^N\setminus\{0\}$.

Let $R_n\to\infty$ be arbitrary and define
\[
v_n(y):=R_n^{\gamma_*}u(R_ny).
\]
The coarse estimate gives uniform upper and lower bounds for $v_n$ on every compact annulus. Moreover, by the balance identity
$\gamma_*(p-q)=\sigma-p$,
\[
-\Delta_pv_n(y)
=\big(R_n^{\sigma}W(R_ny)\big)v_n^{q-1}(y).
\]
Thus the local $C^{1,\alpha}$ estimates and \eqref{eq:slow-W-exact} yield, after passing to a subsequence,
\[
v_n\longrightarrow V
\quad\text{in }C^1_{\mathrm{loc}}(\mathbb R^N\setminus\{0\}),
\]
where $V$ is a positive solution of
\[
-\Delta_pV=m_0|x|^{-\sigma}V^{q-1}
\]
and $V(x)\asymp |x|^{-\gamma_*}$. Lemma \ref{lem:slow-rigidity} therefore gives
\[
V(x)=A_*|x|^{-\gamma_*},
\qquad
A_*=\left(\frac{m_0}{h(\gamma_*)}\right)^{1/(p-q)}.
\]
Since the limit is independent of the chosen sequence $R_n$, the whole family converges in $C^1$ on compact annuli. Evaluating on the unit sphere gives precisely \eqref{eq:slow-exact-u} and \eqref{eq:slow-exact-grad}. In particular the two-sided gradient estimate in \eqref{eq:slowrate} follows. This is the same limiting-profile mechanism used by Sciunzi \cite[Theorem 2.2]{Sciunzi2016}, with the $p$-harmonic fundamental solution replaced here by the unique positive scale-invariant solution of the weighted subhomogeneous limit equation.
\end{proof}

\subsection{Moving planes in the nonclassical phases}\label{sec:mp-nonclassical}

We now prove the symmetry assertions in Theorems \ref{thm:resonant} and \ref{thm:slowrate}. Since $\Lambda\le0$ implies $p<2$, only the singular $p$-Laplace case occurs. The point is that in this range an upper estimate for the gradient gives the coercive lower bound for the weight appearing in the monotonicity inequality. Thus the argument is of the same singular type as in the $1<p<2$ symmetry theory underlying \cite{Vetois2016} and \cite[Section 4.2]{DaiLiLiu2025}; no lower gradient estimate is needed to move the plane. The lower gradient estimates proved above are independent consequences of the blow-down analysis.

Fix a direction, which after rotation we take to be $e_1$, and use the notation $\Sigma_\lambda$, $x_\lambda$, $u_\lambda$ and $w=(u-u_\lambda)^+$ from Section \ref{sec:mp}. Put
\[
A_\lambda:=\{x\in\Sigma_\lambda:u(x)>u_\lambda(x)\},
\qquad
M_\lambda(x):=\max\{|\nabla u(x)|,|\nabla u_\lambda(x)|\}.
\]
Testing the equations for $u$ and $u_\lambda$ with the usual cut-off approximations of $w$ and using \eqref{eq:sharp-monotonicity-singular}, then passing to the limit, gives
\begin{equation}\label{eq:nonclassical-energy}
(p-1)\mathcal E_\lambda
\le I_{1,\lambda}+I_{2,\lambda},
\qquad
\mathcal E_\lambda:=\int_{\Sigma_\lambda}M_\lambda^{p-2}|\nabla w|^2,
\end{equation}
Here the integrand is defined to be $0$ on $\{M_\lambda=0\}$; this is consistent because $\nabla w=0$ there. Thus every estimate containing the singular factor $M_\lambda^{p-2}$ is read on $\{M_\lambda>0\}$ and extends to the weighted energy by this convention. The two terms on the right are
\[
I_{1,\lambda}:=\int_{\Sigma_\lambda}W[u]\big(u^{q-1}-u_\lambda^{q-1}\big)w,
\qquad
I_{2,\lambda}:=\int_{\Sigma_\lambda}\big(W[u]-W[u_\lambda]\big)u_\lambda^{q-1}w.
\]
Only $A_\lambda$ contributes. The reflection identity in Lemma \ref{lem:nonlocal-diff} and Lemma \ref{lem:I2} are valid for every $\lambda\in\mathbb R$. Note that $\mathcal E_\lambda<\infty$ and that this uses only upper gradient bounds: since $|\nabla w|\le2M_\lambda$ and $p<2$, one has $M_\lambda^{p-2}|\nabla w|^2\le4M_\lambda^{p}$, which is locally bounded, while at infinity $M_\lambda\le C|x|^{-\gamma-1}$ up to logarithms, with $\gamma\in\{\gamma_h,\gamma_*\}$, and $p(\gamma+1)>N$ because $\gamma>\gamma_e$.

We first isolate the new estimate at infinity.

\begin{lemma}[Critical-tail coercivity]\label{lem:tail-comparison-nonclassical}
Assume either the assumptions of Theorem \ref{thm:resonant} or those of Theorem \ref{thm:slowrate}.
\begin{enumerate}[label=\textup{(\roman*)}]
\item There exists $R_*>1$ such that $u\le u_\lambda$ in $\Sigma_\lambda$ for every $\lambda<-R_*$.
\item Let $I\subset\mathbb R$ be bounded. There exist a \emph{fixed} radius $R_I>1$, a number $\vartheta_I\in(0,1)$ and $C_I>0$ such that, for every $\lambda\in I$,
\begin{equation}\label{eq:tail-with-remainder}
I_{1,\lambda}+I_{2,\lambda}
\le \vartheta_I(p-1)\mathcal E_\lambda
+C_I\int_{A_\lambda\cap B_{2R_I}}w^2.
\end{equation}
Thus the whole exterior contribution is absorbed with a uniform strict margin, while every interaction with the bounded region is kept explicitly in the last term.
\end{enumerate}
\end{lemma}

\begin{proof}
We treat the two phases separately.

\emph{The slow phase.}
Write $\gamma=\gamma_*$, $A=A_*$ and
\[
d:=N-p-(p-1)\gamma>0,
\qquad
s:=(2-p)(\gamma+1)>0.
\]
By \eqref{eq:slow-exact-u}, \eqref{eq:slow-exact-grad} and Lemma \ref{lem:Wmass},
\begin{equation}\label{eq:slow-tail-exact}
u(x)=A|x|^{-\gamma}(1+o(1)),\quad
|\nabla u(x)|=A\gamma |x|^{-\gamma-1}(1+o(1)),\quad
W(x)=m_0|x|^{-\sigma}(1+o(1)),
\end{equation}
uniformly in direction. If $\lambda\to-\infty$ and $x\in A_\lambda$, then both $|x|$ and $|x_\lambda|$ tend to infinity. Indeed, $|x|\ge|\lambda|$, while a bounded reflected point would contradict $u(x)>u(x_\lambda)$ because $u(x)\to0$ and $u$ is positive on compact sets. Since $|x_\lambda|<|x|$ in $\Sigma_\lambda$, the exact first asymptotic in \eqref{eq:slow-tail-exact} gives
\begin{equation}\label{eq:reflection-ratio}
\frac{|x_\lambda|}{|x|}=1+o(1)
\end{equation}
uniformly on $A_\lambda$ as $\lambda\to-\infty$: otherwise a sequence with ratio bounded above by $1-\delta$ would give
$u(x)/u(x_\lambda)\le(1-\delta)^\gamma+o(1)<1$.
For $\lambda$ in a fixed bounded interval, the same relation holds geometrically on $A_\lambda\cap\{|x|>R\}$ as $R\to\infty$.

Since $q<p<2$, on $A_\lambda$ the mean value theorem gives
\[
W(u^{q-1}-u_\lambda^{q-1})w
\le(q-1)Wu_\lambda^{q-2}w^2.
\]
Using \eqref{eq:slow-tail-exact}--\eqref{eq:reflection-ratio}, for every $\varepsilon>0$ and sufficiently far in the tail,
\begin{align}
I_{1,\lambda}(A_\lambda\cap\{|x|>2R\})
&\le (1+\varepsilon)(q-1)m_0A^{q-2}
\int_{|x|>2R}|x|^{s-2}w^2,\label{eq:I1slow-tail}\\
M_\lambda^{p-2}
&\ge (1-\varepsilon)(A\gamma)^{p-2}|x|^s
\qquad\text{on }A_\lambda\cap\{|x|>R\}.\label{eq:weight-slow-tail}
\end{align}
Here $-\sigma-(q-2)\gamma=s-2$, which is equivalent to $\gamma(p-q)=\sigma-p$.
Apply \eqref{eq:cutoff-hardy} with $\ell=0$ to $w$. Its annular error is supported in $R<|x|<2R$ and will be retained in the bounded remainder. Since
\[
N+s-2=\gamma+d,
\qquad
m_0A^{q-p}=h(\gamma)=\gamma^{p-1}d,
\]
the limiting ratio of the exterior part of $I_{1,\lambda}$ to $(p-1)\mathcal E_\lambda$ is
\begin{equation}\label{eq:theta-slow}
\Theta_*:=
\frac{q-1}{p-1}\frac{4\gamma d}{(\gamma+d)^2}<1.
\end{equation}
Indeed $q<p$ gives $(q-1)/(p-1)<1$, while $4\gamma d\le(\gamma+d)^2$. Because the coefficient in \eqref{eq:cutoff-hardy} can be taken arbitrarily close to the sharp Hardy constant, the asymptotic errors and the cut-off loss can be chosen so small that the exterior part of $I_{1,\lambda}$ is bounded by $\vartheta_1(p-1)\mathcal E_\lambda$ with some fixed $\vartheta_1<1$.

For $I_{2,\lambda}$, Lemma \ref{lem:I2} gives
\[
I_{2,\lambda}
\le C\|u\|_{L^{p^*}(\Sigma_\lambda)}^{2q-p^*}
\int_{A_\lambda}u^{p^*-2}w^2.
\]
From \eqref{eq:slow-tail-exact},
$u^{p^*-2}\le C|x|^{-\gamma(p^*-2)}$ in the tail. Put
\[
\beta_*:=(p^*-2)\gamma+s-2=(p^*-p)\gamma-p>0,
\]
because $\gamma>\gamma_e$ and $(p^*-p)\gamma_e=p$. Hence on $|x|>2R$,
\[
u^{p^*-2}w^2
\le CR^{-\beta_*}|x|^{s-2}w^2.
\]
Splitting only the final weighted $w^2$ integral in Lemma \ref{lem:I2} into $B_{2R}$ and its complement, and then applying \eqref{eq:cutoff-hardy} to the latter, shows that, uniformly for $\lambda$ in a bounded interval,
\begin{equation}\label{eq:I2-tail-small}
I_{2,\lambda}
\le o_R(1)\mathcal E_\lambda
+C_R\int_{A_\lambda\cap B_{2R}}w^2.
\end{equation}
Thus no localization of the original double Hartree integral is being used. Here the factor $\|u\|_{L^{p^*}(\Sigma_\lambda)}^{2q-p^*}$ is simply bounded by the global $L^{p^*}$ norm. If instead $\lambda\to-\infty$, this factor tends to zero, so after $R$ has been fixed it supplies the required smallness for the starting position of the plane.

Combining the preceding estimates proves both assertions in the slow phase: for a bounded interval $I$, first choose one fixed $R_I$ so large that the sum of the exterior coefficients is strictly below $1$ and absorb every annular term into $C_I\int_{A_\lambda\cap B_{2R_I}}w^2$. For the starting position choose the same fixed radius and then take $\lambda<-2R_I$ sufficiently negative. Since $A_\lambda\subset\Sigma_\lambda\subset\{|x|>|\lambda|\}$, the bounded remainder then vanishes, and the factor in \eqref{eq:I2} is small; \eqref{eq:nonclassical-energy} forces $\mathcal E_\lambda=0$ and hence $w\equiv0$.

\emph{The resonant phase.}
Set $\gamma=\gamma_h$, $a=a_0$,
\[
s:=(2-p)(\gamma+1),
\qquad
\ell:=a(p-2).
\]
For $\lambda$ in a bounded interval, $|x_\lambda|/|x|\to1$ uniformly as $|x|\to\infty$. For $\lambda\to-\infty$ and $x\in A_\lambda$, both radii again tend to infinity, but the two-sided resonant estimate does not by itself give ratio convergence. What it does give, and what is sufficient, is the uniform comparability
\begin{equation}\label{eq:res-reflection-comparable}
\kappa\le\frac{|x_\lambda|}{|x|}\le1
\qquad\text{on }A_\lambda
\end{equation}
for some $\kappa>0$ once $\lambda$ is sufficiently negative. To prove this, otherwise there would be $x_n\in A_{\lambda_n}$, $\lambda_n\to-\infty$, with $t_n:=|x_{\lambda_n}|/|x_n|\to0$. The two-sided estimate \eqref{eq:resonant-rate} would give
\[
1<\frac{u(x_n)}{u(x_{\lambda_n})}
\le C t_n^\gamma
\left(\frac{\log|x_n|}{\log|x_{\lambda_n}|}\right)^a
\le C t_n^\gamma\bigl(1+\log(1/t_n)\bigr)^a\longrightarrow0,
\]
a contradiction. In particular
$\log|x_\lambda|/\log|x|\to1$ uniformly on the relevant tail. Together with \eqref{eq:resonant-rate}--\eqref{eq:resonant-grad}, this yields
\[
M_\lambda^{p-2}\ge c|x|^s(\log|x|)^\ell,
\]
and, since $q<2$,
\[
W(u^{q-1}-u_\lambda^{q-1})w
\le C|x|^{s-2}(\log|x|)^{\ell-1}w^2.
\]
The loss of exactly one logarithm is the identity $a(p-q)=1$. Applying \eqref{eq:cutoff-hardy}, now with the logarithmic weight, gives
\[
I_{1,\lambda}(\{|x|>2R\})
\le \frac{C}{\log R}\mathcal E_\lambda
+C_R\int_{A_\lambda\cap B_{2R}}w^2.
\]
Thus the local reflected term is arbitrarily small in the tail. For $I_{2,\lambda}$ we again use Lemma \ref{lem:I2}. Since
\[
\beta_h:=(p^*-2)\gamma_h+s-2=(p^*-p)\gamma_h-p>0,
\]
the factor $R^{-\beta_h}$ dominates the logarithmic powers. Splitting, as above, only the final weighted $w^2$ integral supplied by Lemma \ref{lem:I2}, and applying \eqref{eq:cutoff-hardy} to its exterior part, gives
\[
I_{2,\lambda}
\le o_R(1)\mathcal E_\lambda
+C_R\int_{A_\lambda\cap B_{2R}}w^2
\]
uniformly for $\lambda$ in bounded intervals. For $\lambda\to-\infty$ the $L^{p^*}(\Sigma_\lambda)$ factor in Lemma \ref{lem:I2} tends to zero. The same choice procedure as in the slow phase therefore proves (i) and (ii).
\end{proof}

The fixed-radius formulation in Lemma \ref{lem:tail-comparison-nonclassical}(ii) is precisely what is needed for continuation. The next lemma makes that step explicit.

\begin{lemma}[Local continuation after tail absorption]\label{lem:local-continuation}
Assume either the assumptions of Theorem \ref{thm:resonant} or those of Theorem \ref{thm:slowrate}. Let $\lambda_0\in\mathbb R$ satisfy
\[
u\le u_{\lambda_0}\qquad\text{in }\Sigma_{\lambda_0}.
\]
If $u\not\equiv u_{\lambda_0}$ in $\Sigma_{\lambda_0}$, then there exists $\varepsilon_0>0$ such that
\[
u\le u_\lambda\qquad\text{in }\Sigma_\lambda
\quad\text{for every }\lambda\in[\lambda_0,\lambda_0+\varepsilon_0].
\]
\end{lemma}

\begin{proof}
Choose a bounded interval $I$ containing $\lambda_0$ in its interior and let $R_I$, $\vartheta_I$ and $C_I$ be furnished by Lemma \ref{lem:tail-comparison-nonclassical}(ii). By the already proved upper gradient estimates and compact regularity,
\[
G:=\|\nabla u\|_{L^\infty(\mathbb R^N)}<\infty.
\]
Hence $M_\lambda\le G$ for every $\lambda$, and since $p<2$,
\begin{equation}\label{eq:weight-global-lower}
M_\lambda^{p-2}\ge G^{p-2}
\end{equation}
wherever $M_\lambda>0$. On $\{M_\lambda=0\}$ one has $\nabla w=0$, so the convention in \eqref{eq:nonclassical-energy} gives directly
\[
\int_{\Sigma_\lambda}|\nabla w|^2
\le G^{2-p}\mathcal E_\lambda.
\]
The usual Sobolev inequality in the half-space, applied first to compactly supported approximations of $w$ and then by passage to the limit, gives
\begin{align}
\int_{A_\lambda\cap B_{2R_I}}w^2
&\le |A_\lambda\cap B_{2R_I}|^{2/N}
\left(\int_{\Sigma_\lambda}|w|^{2^*}\right)^{2/2^*}\notag\\
&\le C G^{2-p}|A_\lambda\cap B_{2R_I}|^{2/N}\mathcal E_\lambda,
\label{eq:small-measure-sobolev}
\end{align}
where $2^*=2N/(N-2)$.

By Lemma \ref{lem:strict-separation}, applied with $\lambda=\lambda_0$, the nonsymmetric alternative gives
\[
u(x)<u_{\lambda_0}(x)
\quad\text{for a.e. }x\in\Sigma_{\lambda_0}.
\]
Exactly as in the proof of \eqref{eq:classical-measure-small},
\begin{equation}\label{eq:A-measure-small}
|A_\lambda\cap B_{2R_I}|\longrightarrow0
\qquad\text{as }\lambda\downarrow\lambda_0.
\end{equation}

Combining \eqref{eq:nonclassical-energy}, \eqref{eq:tail-with-remainder} and \eqref{eq:small-measure-sobolev}, for $\lambda\in I$ close to $\lambda_0$ we obtain
\[
(p-1)(1-\vartheta_I)\mathcal E_\lambda
\le C_I C G^{2-p}|A_\lambda\cap B_{2R_I}|^{2/N}\mathcal E_\lambda.
\]
By \eqref{eq:A-measure-small}, the coefficient on the right is strictly smaller than $(p-1)(1-\vartheta_I)$ for $\lambda-\lambda_0>0$ sufficiently small. Hence $\mathcal E_\lambda=0$, so $w\equiv0$ and $u\le u_\lambda$ in $\Sigma_\lambda$. This proves the lemma.
\end{proof}

\begin{proof}[Completion of the proofs of Theorems \ref{thm:resonant} and \ref{thm:slowrate}: radial symmetry]
By Lemma \ref{lem:tail-comparison-nonclassical}(i), the plane starts from $-\infty$ in every direction. Let
\[
\lambda_0:=\sup\{\lambda:\ u\le u_\mu\text{ in }\Sigma_\mu\text{ for every }\mu\le\lambda\}.
\]
At $\lambda_0$ we have $u\le u_{\lambda_0}$ in $\Sigma_{\lambda_0}$. If equality did not hold identically, Lemma \ref{lem:local-continuation} would move the plane past $\lambda_0$, contradicting the definition of $\lambda_0$. Hence
\[
u\equiv u_{\lambda_0}\quad\text{in }\Sigma_{\lambda_0}.
\]
Repeating the argument from the opposite side and in every direction yields radial symmetry about some point $x_0$. For strict monotonicity write $u(x)=U(|x-x_0|)$ with $U\in C^{1}([0,\infty))$, and note that $f:=W[u]u^{q-1}$ is then radial about $x_0$ and positive. Testing the equation with radial functions gives
\[
r^{N-1}|U'(r)|^{p-2}U'(r)=-\int_{0}^{r}s^{N-1}f(s)\,ds<0\qquad(r>0),
\]
the boundary term at $r=0$ vanishing because $U'$ is bounded. Hence $U'(r)<0$ for every $r>0$. This proves the remaining assertions of Theorems \ref{thm:resonant} and \ref{thm:slowrate}.
\end{proof}

\begin{remark}
In the nonclassical phases, the symmetry proof is logically independent of the lower gradient bounds. Its new ingredient in the slow phase is the strict critical-tail coefficient \eqref{eq:theta-slow}; in the resonant phase the corresponding gain is the factor $1/\log|x|$. The role of the lower gradient analysis is instead to sharpen the asymptotic picture itself.
\end{remark}

\begin{remark}
The proof uses $\Lambda<0$ through $\gamma_{*}<\gamma_{h}$, which guarantees $h(\gamma_{*})>0$ and keeps the power barriers of Lemma \ref{lem:slow-step} admissible. At $\Lambda=0$ this coefficient vanishes exactly at the limiting exponent; the logarithmic factor in Theorem \ref{thm:resonant} is the next-order correction to that degeneration. When $\Lambda>0$ the iteration is instead capped at $\gamma_h$, which is the classical Dai--Li--Liu phase.
\end{remark}

\begin{proof}[Proof of Theorem \ref{thm:three-regime}]
Theorem \ref{thm:local-boundedness} first gives $u\in L^\infty_{\mathrm{loc}}(\mathbb R^N)$ under the assumptions of Theorem \ref{thm:three-regime}. If $\Lambda>0$, the assertions are exactly Theorem \ref{thm:main}. If $\Lambda=0$, they are Theorem \ref{thm:resonant}; in this case $q>1$ is automatic by Lemma \ref{lem:qbigger1}. If $\Lambda<0$, they are Theorem \ref{thm:slowrate}. Since these three cases are mutually exclusive and exhaustive, the proof is complete.
\end{proof}

\section*{Declarations}

%\smallskip
%\noindent\textbf{Funding.} This research did not receive any specific grant from funding agencies in the public, commercial, or not-for-profit sectors.

\smallskip
\noindent\textbf{Conflict of interest.} The author declares that he has no conflict of interest.

\smallskip
\noindent\textbf{Data availability.} Data sharing is not applicable to this article, as no datasets were generated or analyzed during the current study.

%\smallskip
%\noindent\textbf{Use of AI tools.} AI-assisted tools were used in developing and drafting this work; the author has verified its contents and takes full responsibility.

\end{document}